\documentclass[twocolumn]{amsart}

\usepackage[letterpaper,margin=0.7in,columnsep=18pt]{geometry}
\usepackage[colorlinks=true,breaklinks=true]{hyperref}
\usepackage{graphicx}
\usepackage{amssymb,amsfonts}
\usepackage{mathrsfs}
\usepackage{subfig}
\usepackage{algorithm}
\usepackage{algorithmic}

\newtheorem{theorem}{Theorem}
\newtheorem{definition}{Definition}

\newtheorem{lemma}{Lemma}
\newtheorem*{problem}{Problem}
\newtheorem{remark}{Remark}
\newtheorem{hypothesis}{Hypothesis}

\newtheorem{proposition}{Proposition}

\graphicspath{{./}{./figs/}}
\renewcommand{\vec}[1]{\boldsymbol{#1}}
\newcommand*{\defeq}{\mathrel{\vcenter{\baselineskip0.5ex \lineskiplimit0pt
			\hbox{\scriptsize.}\hbox{\scriptsize.}}}%
	=}

\newcommand{\reals}{\mathbb{R}}

\providecommand{\va}{\ensuremath{\vec{a}}}
\providecommand{\vb}{\ensuremath{\vec{b}}}

\providecommand{\ve}{\ensuremath{\vec{e}}}

\providecommand{\vg}{\ensuremath{\vec{g}}}

\providecommand{\vq}{\ensuremath{\vec{q}}}
\providecommand{\vr}{\ensuremath{\vec{r}}}

\providecommand{\vu}{\ensuremath{\vec{u}}}
\providecommand{\vv}{\ensuremath{\vec{v}}}
\providecommand{\vw}{\ensuremath{\vec{w}}}

\title{A Riemannian View on Active Subspaces}
\author{Zachary J. Grey}
\address{National Institute of Standards and Technology, Information Technology Laboratory, Applied and Computational Mathematics Division, Boulder, CO, USA}
\email{zachary.grey@nist.gov}

\begin{document}
\sloppy

\begin{abstract}
Active subspaces provide an explainable, eigenvalue-ordered principle for studying how scalar-valued quantities of interest change the most, on average, over a reduced basis of Euclidean domains. Composition with parallel transport generalizes this principle from Euclidean space to quantities of interest defined over Riemannian manifolds, and the resulting intrinsic formulation is contrasted with the extrinsic, embedding-based gradient average of manifold learning. Either strategy is studied in an intrinsically local sense, restricted to mean-centered geodesic-balls, and within that scope the two are not identical: on the central tangent space, eigenvalues agree to second order in the geodesic radius of the sampled domain, while dominant eigenspaces agree at the same order relative to the spectral gap. Extending activity beyond that central space then calls for either recomputed decompositions over changing tangent spaces or, intrinsically, parallel transport of a single central frame. Hyperspheres are emphasized throughout as a particular manifold of interest, motivated by applications over preshape spaces for statistical shape analysis. Numerical examples over the $2$-sphere illustrate the formalism, including the derived ridge recovery at a curvature-limited quadratic rate.
\end{abstract}

\maketitle

\section{Introduction}
Ridge approximations reduce a high-dimensional computational model to a function of a few linear combinations of its parameters, with powerful implications for the approximation, integration, optimization, and sensitivity analysis of a scalar quantity of interest. Active subspaces give an interpretable, eigenvalue-ordered approach to subspace-based dimension reduction. Dimension reduction need not be linear, however, and we develop the geometric abstractions that extend active subspaces to smooth functions on smooth manifolds. We pose the parameter-reduction manifold hypothesis (section~\ref{sec:hypothesis}), review active subspaces (section~\ref{sec:AS_intro}), extend them to Riemannian manifolds (section~\ref{sec:intro-Riemannian-geo}), and close with explainable examples over the $2$-sphere (section~\ref{sec:eg_2sphere}). The spherical setting is deliberate: after centering and size or length normalization, the \textit{preshape} spaces of statistical shape analysis are hyperspheres \cite{kendall1984,mio2007shape}, so the sphere results apply directly to functions of preshape discretizations; richer matrix manifolds and physics-based applications are outlined as outlook in the conclusions. Proofs of all numbered results are collected in Appendix~\ref{app:proofs}.

\subsection{Parameter-reduction manifold hypothesis}
\label{sec:hypothesis}

A growing interest~\cite{hokanson2017,mukherjee2010learning,constantine2014active,lewis2016gradient,loudon2016,Cook1998,glawsMHD} suggests that problems in a large number of parameters $x \in \mathcal{X} \subseteq \mathbb{R}^n$---the integration \cite{Glaws2018}, approximation \cite{hokanson2017,Constantine11c}, or optimization \cite{Alonso2017,Grey2017,Constantine2015exploiting} of a scalar \textit{quantity of interest} $f:\mathcal{X} \subseteq \mathbb{R}^n \rightarrow \mathbb{R}$---become tractable by restricting the parameters to a low-dimensional manifold $\mathcal{M} \subseteq \mathbb{R}^n$. This motivates the following hypothesis.

\begin{hypothesis} \label{hypothesis}
	Given a function $f:\mathcal{X} \subseteq \mathbb{R}^n \rightarrow \mathbb{R}$ as a map from high-dimensional input parameters to a scalar-valued quantity of interest, we assume there exists a \textit{smooth submersion} $\varphi^{-1}:\mathcal{X} \rightarrow \mathbb{R}^r$,  $r < n$, a function $h: \mathbb{R}^r \rightarrow \mathbb{R}$, and small $\epsilon \geq 0$ such that 
	$$
	\Vert f(x) - h\left(\varphi^{-1}(x)\right) \Vert_{L^2_{\rho}(\mathcal{X})} \leq \epsilon.
	$$
\end{hypothesis}
The assumptions on $\varphi^{-1}$ above are sufficient to establish the existence of a smooth $r$-dimensional submanifold of $\mathcal{X}$ by the Regular Level Set Theorem for smooth submersions \cite{Lee2003}. Thus, the take-away of Hypothesis \ref{hypothesis} is that $f$ can be approximated up to a stated $\epsilon$ by a function restricted to this $r$-dimensional submanifold. Suppose $\mathcal{M}$ represents an embedded submanifold as an $r$-dimensional subspace of $\mathbb{R}^n$. In other words, $\varphi(y) \defeq Uy$ where $U \in \mathbb{R}^{n \times r}$ is a full-rank matrix whose columns constitute a basis of a subspace in $\mathbb{R}^n$. As a result, a transverse $(n-r)$-embedded submanifold $\mathcal{M}^{\perp}$ is the level set $\text{Null}(U^{\top}) \defeq \lbrace x \in \mathbb{R}^n \,:\, U^{\top}x = 0\rbrace$. By Hypothesis \ref{hypothesis}, $f$ varies negligibly over the $(n-r)$ directions of $\mathcal{M}^{\perp}$, so we fix them and focus on the remaining $r$ coordinates. Taking the columns of $U$ orthonormal, $U^{\top}U = I_r$, the approximation reduces to determining $U$ with $\varphi^{-1} \defeq U^{\top}$ and
\begin{equation} \label{eq:sub_dim_reduction}
\Vert f(x) - h\left(U^{\top}x\right) \Vert_{L^2_{\rho}(\mathcal{X})} \leq \epsilon.
\end{equation}
With orthonormal bases, $\mathcal{M}$ and $\mathcal{M}^{\perp}$ are genuine orthogonal complements---motivating the notation---and the composition $\varphi \circ \varphi^{-1}: \mathcal{X} \subseteq \mathbb{R}^n \rightarrow  \mathcal{X} \subseteq \mathbb{R}^n$ is the linear projection $UU^{\top}$.

\subsection{Nonlinear abstraction}
Per the abstractions of \cite{Lee2003}, we have motivated generalizing the surjective linear map $U^{\top}$ to a smooth submersion $\varphi^{-1}=(\varphi^{-1}_1,\dots,\varphi^{-1}_r)$ whose $r$ component functions supply reduced coordinates on $\mathbb{R}^n$. In the context of Hypothesis~\ref{hypothesis}, we seek these $r$ coordinate maps---defining a manifold $\mathcal{M}$ that captures the variation in $f$---and we want $r$ small, so that the \textit{important} directions can be explored with computational techniques that otherwise scale poorly in the domain dimension.

Other applications seek an approximation of $\varphi^{-1}$ which maps inherently high-dimensional embedded points in the domain $x \in \mathcal{X} \subseteq \mathbb{R}^n$ to low-dimensional \textit{coordinates} in the image of $\varphi^{-1}$. This more general nonlinear interpretation for determining the map $\varphi^{-1}$, particularly in the context of machine learning and artificial intelligence, may involve fitting a homeomorphism to data $\left\lbrace x_i \right\rbrace_{i=1}^N$ given vectors $x_i \in \mathcal{X}$, e.g.,

\begin{equation} \label{eq:homeo_min}
	\underset{\theta \in \mathcal{D}}{\text{minimize}}\,\, \sum_{i=1}^N \left(x_i - (\varphi \circ \varphi^{-1})(x_i; \theta)\right)^2,
\end{equation}
for parameters $\theta \in \mathcal{D} \subseteq \mathbb{R}^{d_\theta}$ of the composed submersion and immersion---for instance an autoencoder \cite{alain2014regularized}. Lacking a function response, \eqref{eq:homeo_min} is \textit{unsupervised}, akin to principal component analysis \cite{Jolliffe2002}, and is complicated by the non-uniqueness of the chart\footnote{Notice the notation convention for the homeomorphism $\varphi$ is opposite of \cite{Lee2003} and traditional definitions of charts since several computational applications are primarily interested in the resulting $\varphi$ as a parametrization of a manifold \cite{Walker2015}. This convention also corresponds nicely with a choice of normal coordinates to be developed in following sections.} $(\mathcal{X}, \varphi^{-1})$---evident in the transition maps between overlapping charts \cite{Lee2003}---so it is regularized through Jacobian (tangent-space) or Hessian (curvature) constraints \cite{alain2014regularized}. Moreover, the \textit{supervised} counterpart which simultaneously guides a nonlinear dimension reduction using a function's response is generally ill-posed without such constraints.

Either submanifold $\mathcal{M}$ or $\mathcal{M}^{\perp}$ may be the informative one \cite{lee2018model,mukherjee2010learning,bridges2019active}: $\varphi^{-1}$ maps to a latent \textit{coordinate representation} (a \textit{chart}) and $\varphi$ \textit{parametrizes} the $r$-submanifold $\mathcal{M} \subseteq \mathbb{R}^n$. Of course, all of this is contingent on the data being manifold-valued (the \textit{manifold hypothesis}) and that the mapping is provably a homeomorphism---assumptions which are often challenging to assert for overly complicated data and model forms.

\subsection{Related work}
\label{sec:lit_review}
Active subspaces \cite{constantine2014active,Constantine2015} provide the linear template we generalize: the dominant eigenspaces of the average gradient outer product \eqref{eq:C} order directions by mean-squared directional derivative \eqref{eq:eigvals} and certify a ridge approximation \eqref{eq:err_apprx}. The construction has since been extended in many directions---near-stationary ridge recovery \cite{constantine2017near,hokanson2018data}, Lipschitz-matrix and inverse-regression variants \cite{hokanson2019lipschitz,glaws2017inverse}, vector-valued targets \cite{zahm2018}, and kernel methods \cite{romor2023kernel}---yet each remains a \textit{linear} subspace over a Euclidean domain; see \cite{serani2025survey} for a recent survey of the design-space reduction methods used in shape optimization. This work instead retains the ordered, eigenvalue-based interpretation over a domain that is itself a Riemannian manifold.

Fully nonlinear alternatives abandon the subspace altogether: active manifolds \cite{bridges2019active} and their neural counterparts \cite{zanoni2025neural} approximate the response along a single one-dimensional curve. Such constructions are expressive but fragile---representing genuinely multi-dimensional variation with a one-dimensional manifold may force the curve to become \textit{space-filling} (absent regularization), forfeiting the interpretability that linear reductions motivate in the first place. We instead take the domain manifold as given---e.g.\ a matrix or shape manifold \cite{Absil2008,Schulz2014}---and seek an \textit{ordered set} of geodesic directions emanating from a central tangent space, recovering low-dimensional structure without a space-filling parametrization.

A separate lineage is also preoccupied with fitting a curve rather than a subspace, beginning with principal curves \cite{hastie1989principal} and continuing onto manifolds with principal flows \cite{panaretos2014principal}, whose curve follows at each point the leading eigenvector of a kernel-localized tangent covariance of the data. That construction is unsupervised---its second moment averages data \textit{positions}---and it is the natural counterpart of the transported frame field of section~\ref{sec:AMG_field}, which instead extends a single response-driven decomposition; localizing the present construction and integrating the resulting field is the subject of forthcoming work. A related family estimates \textit{density} ridges, the set on which a density is locally maximal in a subspace of Hessian eigenvectors, by subspace-constrained mean shift \cite{ozertem2011locally}, with consistency theory in \cite{genovese2014nonparametric}. There is a conflation of vocabulary which is worth caution: \textit{ridge} throughout this article carries the ridge-function sense of \eqref{eq:ridge}---a response varying through only a few linear combinations of its inputs---rather than an intuitive ridge along a density.

Alternatively, a sequential and potentially less degenerate posture first learns the nonlinear domain---a space of equivalence classes, quotienting out actions deemed unimportant---and only then identifies activity over it. Solving for the domain and the reduction \textit{simultaneously} may be tractable but any implicit regularization of a well-posed nonlinear dimension reduction by a better, smoother data domain transfers immediately to the methods discussed in this work. The two postures differ in how the domain is constrained: \textit{explicitly}, by first asking which equivalences best represent the domain and then seeking submanifolds of interest; or \textit{implicitly}, by penalizing the joint optimization so that a class of $\varphi$ carves out well-regularized constraint surfaces. The implicit route is often the easier---penalization and model forms can be chosen empirically---though it lacks a more principled interpretation of the invariances which are learned/represented.

Whether implicitly or explicitly, determining a subset of important directions is the province of geometric statistics: principal geodesic analysis \cite{fletcher2004principal}, its geodesic-PCA refinement for manifolds modulo isometric group actions \cite{huckemann2010intrinsic}, barycentric subspace analysis \cite{pennec2018barycentric}. Two perspectives recur, both posed over the \textit{same} given Riemannian manifold: an \textit{intrinsic} construction assembled from quantities defined on the manifold itself, and an \textit{extrinsic} one read off an isometric embedding in an ambient Euclidean space. We make this dichotomy precise for scalar-valued activity. Parallel-transporting each Riemannian gradient along the unique geodesic to a central tangent space before averaging \eqref{eq:riemannian_view} yields a single coordinate-independent $(0,2)$-tensor whose eigenvalues preserve the mean-squared-directional-derivative reading \eqref{eq:eigvals}. An embedding-based average \eqref{eq:emb_opg} instead depends on the chosen embedding and carries a spectrum conflated by projection (Thm.~\ref{thm:normal_eigvals})---an agreement quantified to second order in the radius of the sampled geodesic balls (Prop.~\ref{prop:local_agreement}). We therefore often favor intrinsic constructions but this is predicated on explicit knowledge of the domain---presumably within reach even for those which are learned~\eqref{eq:homeo_min}.

Statistical dimension reduction has also been carried to non-Euclidean data, along two lines that are best separated by what each assumes about the geometry. The first places the \textit{predictors} on a manifold that is unknown and must be estimated from the sample. Supervised manifold learning of this kind combines kernel dimension reduction with a Laplacian-eigenmap representation of the covariates, and is presented explicitly as carrying sufficient dimension reduction from linear subspaces onto manifolds \cite{nilsson2007regression}. A closely related family instead learns the gradient of the regression function, forming a gradient outer product in ambient coordinates whose convergence rates are governed by the intrinsic dimension of the manifold supporting the data rather than by the ambient dimension \cite{mukherjee2010learning,wu2010learning}. Because the domain is inferred rather than given in both cases, its geometry enters as a regularizer or as a rate, and the reduction is still read off ambient or feature coordinates. 

The second line takes the geometry as given but attaches it to the \textit{response}: Fr\'echet sufficient dimension reduction treats responses valued in a metric space with Euclidean predictors \cite{ying2020frechet}, and intrinsic minimum-average-variance and outer-product-of-gradients estimators have been developed for responses valued in the manifold of symmetric positive-definite matrices \cite{chen2023intrinsic}, where \textit{intrinsic} refers to exploiting the geometry of the response manifold through geodesic distance and parallel transport inside a local Taylor expansion. The problem posed here draws from both: the geometry is given, as in the second line, but it is carried by the \textit{domain}, as in the first. The response is scalar and the domain is a Riemannian manifold, so the gradients being averaged are Riemannian gradients of $f$ and the transport acts on the tangent bundle of the \textit{domain}. Throughout this article, accordingly, \textit{intrinsic} qualifies the geometry of the domain. 

Parallel transport has recently been used to build an unsupervised second-order descriptor of manifold-valued observations, by transporting local variations to a common tangent space \cite{soto2026intrinsic}. Here, by contrast, it is the response gradient that is transported, which is what preserves the reading of important directions as mean-squared directional derivatives \eqref{eq:eigvals}.

Hyperspheres receive particular emphasis because they are not merely a convenient test geometry: they are the \textit{preshape spaces} of statistical shape analysis. Kendall's landmark preshapes---configurations of $k$ planar landmarks after centering and size normalization---constitute the hypersphere $S^{2k-3}$ \cite{kendall1984}, and elastic representations of \textit{open} planar curves, after normalizing length, likewise realize preshapes as elements of a hypersphere \cite{mio2007shape,klassen2004analysis,Joshi2007}. Our results are stated on finite-dimensional $S^n$, so they apply directly to landmark preshapes and to discretized elastic representations (direction functions or square-root velocities sampled at finitely many points), where the preshape sphere is finite-dimensional. The developed closed-form projection--transport identity (Lemma~\ref{lem:sphere_proj_trans}) and the second-order agreement of Prop.~\ref{prop:local_agreement} then govern response-driven dimension reduction for functions of shape along preshape geodesics.

\subsection{Contributions}
\label{sec:contributions}
We make the following contributions:
\begin{itemize}
\item an \textit{intrinsic} generalization of active subspaces to a scalar response on a Riemannian manifold---the parallel-transport gradient outer product \eqref{eq:riemannian_view}, a coordinate-free tensor whose eigenvalues carry an \textit{exact} mean-squared directional derivative interpretation along a transported frame field (Lemma~\ref{lem:exact_eigvals});
\item the \textit{active manifold-geodesics} (Def.~\ref{def:AMG})---an eigenvalue-ordered set of geodesic directions generalizing the ordered active subspace, identifiable and low-complexity, avoiding the space-filling degeneracy of unconstrained nonlinear reductions (Remark~\ref{rem:dichotomy});
\item a precise \textit{intrinsic-versus-extrinsic} comparison---the embedding-based average \eqref{eq:emb_opg} carries a spectrum conflated by a partial-isometry projection (Thm.~\ref{thm:normal_eigvals}), yet agrees with the intrinsic construction to second order in the sampled radius: eigenvalues at $\mathcal{O}(R^2)$ and dominant eigenspaces at $\mathcal{O}(R^2/\eta)$ under a spectral gap $\eta$ (Prop.~\ref{prop:local_agreement}), in closed form on hyperspheres (Lemma~\ref{lem:sphere_proj_trans}). That agreement does not carry the ambient eigenvectors with it, since projected dominance is not dominant projection (Remark~\ref{rem:noncommute}). This interpretation motivates an important convention: \textit{activity should never be extended by projecting ambient dominance}. Instead, intrinsic frames extend a single central decomposition by parallel transport;
\item direct application to \textit{preshape spaces}---landmark and length-normalized elastic-curve preshapes are hyperspheres, so the theory governs response-driven dimension reduction for functions of shape, illustrated over the $2$-sphere.
\end{itemize}

\subsection{Active subspaces}
\label{sec:AS_intro}

We review active subspaces \cite{Constantine2015} in the form our Riemannian extension generalizes. Let the computational model be a differentiable function $f: \mathcal{X}  \subset \mathbb{R}^n \rightarrow \mathbb{R}$ on a compact $\mathcal{X}$, with parameters $x = (x^1,\dots,x^n)^{\top}$ and square-integrable gradient
\begin{equation} \label{eq:grad}
\overline{\nabla}f\defeq\left(\frac{\partial f}{\partial x^1},\dots,\frac{\partial f}{\partial x^n}\right)^{\top} \in \mathbb{R}^n.
\end{equation}
The compact $\mathcal{X}$ is set by the application---often a hyper-rectangle $[x^{\ell},x^u]^n$ of reasonable parameter intervals, chosen to keep a physics-based model well defined \cite{white2019,grey2019,Constantine2015exploiting,glawsMHD} or to make reasonable perturbations to a parametrized geometry \cite{Lukaczyk2014,Grey2017,SEQUOIA2017}.

Integrating over this compact domain, we define a positive semi-definite matrix and associated (real nonnegative) eigendecomposition,
\begin{equation} \label{eq:C}
C \defeq \int_{\mathcal{X}}\overline{\nabla} f(x) \otimes \overline{\nabla} f(x) \, \rho(x) dx = W\Lambda W^{\top},
\end{equation}
where $W = \left[\boldsymbol{w}_1\, \dots \, \boldsymbol{w}_n\right] \in \mathbb{R}^{n\times n}$ is orthogonal and $\Lambda = \text{diag}(\lambda_1,\dots,\lambda_n) \in \mathbb{R}^{n \times n}$ with eigenvalues $\lambda_i \geq 0$ for all $i=1,\dots,n$. We have defined this integration using an arbitrary measure represented by $\rho:\mathcal{X} \subseteq \mathbb{R}^n \rightarrow
\mathbb{R}_+$ such that $\int_{\mathcal{X}}\rho(x)dx = 1$. Without loss of generality, we always consider centering and rescaling of the domain by affine transformation to achieve $\int_{\mathcal{X}}x\rho(x)dx = 0$ and $\int_{\mathcal{X}}xx^{\top}\rho(x)dx = I_n$. The general integral measure $\rho(x)dx^1\dots dx^n$ provides flexibility to consider arbitrary weighting of the input parameters over the domain although a uniform measure is sufficient for many applications \cite{Grey2017,Lukaczyk2014,glawsMHD,grey2019}.

Following the developments of~\cite{Constantine2015}, the eigenvectors are the \textit{important directions} of $f$. Rearranging \eqref{eq:C}, the non-negative eigenvalues become
\begin{equation} \label{eq:eigvals}
\lambda_i = \mathbb{E}\left[\left(\boldsymbol{w}_i^{\top}\overline{\nabla} f\right)^2\right], \,\, \text{for} \,\, i=1,\dots,n.
\end{equation}
Since the directional derivative $df_{x}[\boldsymbol{w}_i] \defeq \lim_{t \rightarrow 0} (f(x + t\boldsymbol{w}_i) - f(x))/t$ is equivalent to the inner product $\boldsymbol{w}_i^{\top} \nabla f(x)$, we have $\lambda_i = \mathbb{E}[df^2_{x}[\boldsymbol{w}_i]]$: the $i$th eigenvalue is the mean-squared differential along the $i$th eigenvector. The ordering $\lambda_1 \geq \dots \geq \lambda_r > \lambda_{r+1} = \dots =\lambda_n = 0$ ($1\leq r \leq n$) thus ranks the directions $\boldsymbol{w}_i$ by how much they change $f$ on average, down to the trailing $r+1,\dots,n$ directions that \textit{do not change $f$ at all}---the differentials over $\mathrm{span}\left \lbrace \vw_{r+1}, \dots, \vw_{n} \right \rbrace$ vanish:

\begin{proposition} \label{prop:eigs_and_constant_f}
	Let $\mathcal{X}$ be compact and regular closed, $\mathcal{X}=\overline{\mathrm{int}\,\mathcal{X}}$,
	let $f\in C^1(\mathcal{X})$ (so $\overline{\nabla}f$ extends continuously to $\mathcal{X}$), and let
	$\rho$ have full support, $\mathrm{supp}\,\rho=\mathcal{X}$. Write
	$\mathcal{S}\defeq\mathrm{span}\lbrace\boldsymbol{w}_{r+1},\dots,\boldsymbol{w}_n\rbrace$ for the trailing
	eigenspace of $C$. The following are equivalent:
	\begin{enumerate}
		\item[(i)] $\lambda_i=0$ for all $i=r+1,\dots,n$;
		\item[(ii)] $df_x[\boldsymbol{v}]=0$ for every $x\in\mathcal{X}$ and every $\boldsymbol{v}\in\mathcal{S}$;
		\item[(iii)] $f(x+\boldsymbol{v})=f(x)$ for every $x\in\mathcal{X}$ and every $\boldsymbol{v}\in\mathcal{S}$
		with $\lbrace x+t\boldsymbol{v}:t\in[0,1]\rbrace\subseteq\mathcal{X}$.
	\end{enumerate}
\end{proposition}

Here $\mathcal{X}$ is \textit{regular closed}, $\mathcal{X}=\overline{\mathrm{int}\,\mathcal{X}}$. Since $\mathcal{X}$ coincides with the closure of its own interior, it carries no low-dimensional appendages---``whiskers'' or ``spikes''---along which the function could still vary despite a vanishing spectrum~\eqref{eq:eigvals}. Closed hyper-rectangles and balls qualify, as do the more general compact domains of section~\ref{sec:AS_intro}. The distinction anticipates a later choice of domain: the star-shaped supports natural on a manifold (Def.~\ref{def:riemannian_view}) may manifest such appendages, whereas the closed geodesic balls we ultimately adopt (Def.~\ref{def:AMG}) are the closure of an open ball and so satisfy regular-closedness automatically.

Motivated by~\ref{prop:eigs_and_constant_f}, instead of changing a function the most on average, we may reinterpret ``importance'' to motivate a different set of $r$-directions for a particular application \cite{hokanson2017,glaws2017inverse,Li1992,li2018sufficient}. In this work, we seek to emulate an active subspace approximation in an effort to generalize the \textit{explainable} interpretation of an \textit{ordered set} of important directions made precise by \eqref{eq:eigvals}.

\subsection{Ridge functions \& domain partition}
Suppose $C$ is rank deficient such that $\text{rank}(C) = r$ for some $1\leq r < n$. Per Prop. \ref{prop:eigs_and_constant_f}, the $(n-r)$ trailing eigenvalues are equal to zero if and only if the function is not changing along the $\left\lbrace \boldsymbol{w}_{r+1},\dots,\boldsymbol{w}_{n} \right\rbrace$ trailing directions. In this case, we can exactly represent $f(x)$ as a new function composed with fewer $r$-parameters, $h\left(W_r^{\top}x\right)$, where $h:\mathbb{R}^r \rightarrow \mathbb{R}$ such that
\begin{equation} \label{eq:ridge}
f(x) = h(W_r^{\top}x), \quad \text{for all} \quad x \in \mathcal{X}.
\end{equation}
In this case, $f$ is called a \textit{ridge function} and the span of the first $r$-columns of $W$ is called the \textit{active subspace} \cite{Constantine2015}. We use $W_r \defeq \left[\boldsymbol{w}_1 \, \dots\, \boldsymbol{w}_r\right]$ to denote a partition of the first $r$-columns of $W$ associated with the decaying order of eigenvalues. In contrast, the trailing $(n-r)$ orthogonal directions from the remaining columns $W_{\perp} \defeq \left[\boldsymbol{w}_{r+1} \, \dots \, \boldsymbol{w}_n\right]$ constitute a basis for the \textit{inactive subspace}. 

With $W = [W_r \,\, W_{\perp}]$ a new orthonormal basis for $\mathbb{R}^n$, the ambient space splits into the direct sum of the active and inactive subspaces,
\begin{equation} \label{eq:param_partition}
\left\lbrace x \in \mathbb{R}^n \,:\, x = W_r y + W_{\perp}z \right\rbrace,
\end{equation}
for all $y \in \mathbb{R}^r$ and $z \in \mathbb{R}^{n-r}$. This is related to our introductory motivation in section \ref{sec:hypothesis} such that $\mathcal{M} \defeq \text{Range}(W_r)$ and $\mathcal{M}^{\perp} \defeq \text{Range}(W_{\perp})$ informs a useful partition $\mathcal{M} \oplus \mathcal{M}^{\perp} \subseteq \mathbb{R}^n$ of the function's domain\footnote{Care must be taken in practice to appropriately restrict new parameter combinations to the original domain \cite{stinson2016, Constantine2015}}---two transverse linear submanifolds, one which changes $f$ the most on average, $\mathcal{M}$, and another which does not, $\mathcal{M}^{\perp}$. 


When the trailing eigenvalues are small but nonzero, a general $f$ (not necessarily a ridge function) is still well approximated over the leading directions, $f(x) \approx h(W_r^{\top}x)$ for some $h:\mathbb{R}^r \rightarrow \mathbb{R}$; \cite{Constantine2015} makes this precise,
\begin{equation} \label{eq:err_apprx}
\Vert f(x) - h(W_r^{\top}x) \Vert_{L^2_\rho(\mathcal{X})} \leq c_{\rho}\left(\lambda_{r+1}+\dots+\lambda_n\right)^{\frac{1}{2}}
\end{equation}
with $c_{\rho}>0$ depending on $\rho$ and $h$ the conditional average over the inactive coordinates $z \defeq W_{\perp}^{\top}x$. This ties the dimensionality $r$ to Hypothesis \ref{hypothesis} through $\epsilon = c_{\rho}(\lambda_{r+1}+ \dots + \lambda_n)^{1/2}$, and the resulting \textit{ridge approximation} over the active subspace is a near-stationary point for ridge recovery \cite{constantine2017near}. 

Given an explicit or approximate gradient, we seek an intrinsic form of \eqref{eq:C} which is generalizable to smooth manifolds and obtained at a similar computational cost to Algorithm 3.1 of \cite{Constantine2015} (random sampling to estimate active subspaces). Additional work formalizing the Monte Carlo approximation of the subspace is presented in \cite{constantine2014computing}. In section \ref{sec:intro-Riemannian-geo} we introduce an analogous algorithm for approximating the extension of active subspaces over Riemannian manifolds.

\subsection{Generalized problem statement}
We now seek structure analogous to the eigenspaces of \eqref{eq:C} when the domain of $f:\mathcal{X} \subseteq \mathcal{M} \rightarrow \mathbb{R}$ is a Riemannian $n$-manifold $(\mathcal{M},g)$. In direct analogy to active subspaces, consider:

\begin{problem} \label{problem}
	Given Riemannian $n$-manifold $(\mathcal{M},g)$, approximate ordered directions $\boldsymbol{v}_i \in T_{p_0}\mathcal{M}$ in a central tangent space at $p_0 \in \mathcal{M}$ for $i=1,\dots,r$ and $r \leq n$ which change $f:\mathcal{X} \subseteq \mathcal{M}  \rightarrow \mathbb{R}$ with square-integrable gradient more, by some analogous globalizing notion of the average.
\end{problem}


\section{A Riemannian view}
\label{sec:intro-Riemannian-geo}
For a quantity of interest $f:\mathcal{X} \subseteq \mathcal{M} \rightarrow \mathbb{R}$ on a Riemannian $n$-manifold $(\mathcal{M},g)$ we require a principled calculus generalizing active subspaces. We assume either an explicitly known domain---matrix or shape manifolds \cite{Absil2008,Schulz2014}---or an approximated manifold \cite{fletcher2004principal,alain2014regularized}, and develop the \textit{intrinsic} Riemannian view, the eigenspaces of an analog of \eqref{eq:C}, contrasting it with an \textit{extrinsic} perspective \cite{mukherjee2010learning,wu2010learning}. The underlying differential-geometry definitions and helpful constructions are collected in Appendix~\ref{app:geo}.

\subsection{Intrinsic vs. extrinsic perspective}
The intrinsic and extrinsic distinction is computational. Many treatments are motivated by smooth isometric embeddings $\iota:\mathcal{M} \hookrightarrow \mathbb{R}^{m}$, which exist for every Riemannian manifold by the popular theorem of~\cite{nash1956imbedding}, representing the metric of $\mathcal{M}$ as that induced in an ambient Euclidean space \cite{mukherjee2010learning,wu2010learning}. In practice all computation is performed on such ambient realizations and is therefore extrinsic: a manifold-learning chart is inferred from samples observed in ambient coordinates \cite{mukherjee2010learning,wu2010learning}, and Grassmannian computations use a representative Stiefel matrix in $\mathbb{R}^{n\times r}$ \cite{edelman1998geometry}. We use the inclusion map $\iota$, and the hat accent for elements of its image, to mark when an abstract object is viewed in ambient coordinates---e.g.\ $\widehat{x} \in \iota(S^2) \subset \mathbb{R}^3$ for the unit sphere. However, as it has been discovered, the two perspectives can yield genuinely different objects---the intrinsic Karcher mean versus the extrinsic mean \cite{srivastava2002monte,fletcher2004principal,bhattacharya2003large})---so we also pose the extension of \eqref{eq:C} intrinsically. This mitigates the convenience of a particular choice of embedding while recognizing the ways in which the two may coincide under conditions made precise below.

\subsection{Differential, tangent vectors, and the metric}
We take from Appendix~\ref{app:geo} the tangent space $T_p\mathcal{M}\cong\mathbb{R}^n$, the differential $df_p$ as the coordinate directional derivative \eqref{eq:eigvals}, and the Riemannian metric $g$, represented at a coordinate $x$ by a symmetric positive-definite matrix $G_x$ ($=I_n$ in the Euclidean case of section~\ref{sec:AS_intro}). Two features of the differential drive the development.
\begin{remark}
	The differential is local: there is, in general, no intrinsic interpretation of linear combinations (integration) of tangent vectors---or higher-rank tensors---drawn from distinct tangent spaces.
\end{remark}
\begin{remark}
	When the manifold is a Euclidean space, such linear combinations remain in that space. This is exactly the canonical isomorphism that the analysis of section~\ref{sec:AS_intro} (and Algorithm 3.1 of \cite{Constantine2015}) relies on.
\end{remark}
Denoting the \textit{tangent bundle} $T\mathcal{M} \defeq \coprod_{p\in \mathcal{M}}T_p\mathcal{M}$, the central difficulty is therefore to globalize the average gradient tensor product of \eqref{eq:C}---a Monte Carlo sum of tensors from distinct tangent spaces---once that canonical isomorphism is lost.

\subsection{Riemannian gradient}
The metric promotes the differential (a covector) to a vector field, the Riemannian gradient.
\begin{definition}[Riemannian gradient, \cite{Absil2008}, Ch. 3]
	\label{def:riemann_grad}
	Given a smooth scalar-valued function $f \in C^{\infty}(\mathcal{M})$ on a Riemannian manifold $(\mathcal{M},g)$, the \textit{gradient} of $f$ at $p \in \mathcal{M}$, denoted $\text{grad} f(p)$ or $\nabla_{\mathcal{M}} f(p)$, is the unique element of $T_p\mathcal{M}$ satisfying
	$$
	 g_p(\text{grad} f(p),\vv_p) = \langle \nabla_{\mathcal{M}}f(p), \vv_p \rangle_p = df_p[\vv],
	$$
	for all $\vv \in T_p\mathcal{M}$.
\end{definition}
In coordinates $\nabla_{\mathcal{M}}\widehat{f}(x) = G_x^{-1}d\widehat{f}^{\top}_x$, which coincides with the differential's components only when $G_x = I_n$ \cite{mukherjee2010learning,wu2010learning}; in general the covector $d\widehat f_x$ and the vector $\nabla_{\mathcal M}\widehat f(x)$ are distinct objects related by the metric. 


\subsection{Riemannian connections}
An affine connection $\nabla$ (Appendix~\ref{app:geo}) generalizes differentiating one vector field along another. It induces an \textit{isometry}, \textit{parallel transport} $\mathcal{P}_{t,t_0}:T_{\gamma(t)}\mathcal{M} \rightarrow T_{\gamma(t_0)}\mathcal{M}$ with $\gamma(t)=p$, $\gamma(t_0)=p_0$, carrying vectors uniquely along a smooth curve (\cite{Lee1997}, Thm. 4.11), and it determines the zero-acceleration \textit{geodesics}, the distance-minimizing solutions of the geodesic equations \cite{Lee1997}. We write $\text{exp}_{p_0}:T_{p_0}\mathcal{M} \cong \mathbb{R}^n \rightarrow \mathcal{M}$ for the parametrization solving the geodesic initial-value problem and $\text{exp}^{-1}_{p_0}$ for its inverse. In other words, over an appropriate restriction, $\text{exp}_{p_0}$ is a diffeomorphism inducing \textit{normal coordinate charts} \cite{Lee1997}, $\varphi^{-1} \defeq \text{exp}^{-1}_{p_0}$, $\varphi \defeq \text{exp}_{p_0}$.

\subsection{Normal coordinates \& neighborhoods}
Intrinsic quantities and scaled domains over Riemannian manifolds have a natural representation in normal coordinates.\footnote{A rescaling of geodesic submanifolds used for nondimensionalization is deferred to Appendix~\ref{app:geo}.} By Lemma 5.10 of \cite{Lee1997}, $\text{exp}_p:\mathcal{V} \subseteq T_p\mathcal{M} \rightarrow \mathcal{U} \subseteq \mathcal{M}$ is a diffeomorphism on a normal neighborhood $\mathcal{U}$ centered at $p$. Choosing an orthonormal basis $\lbrace \ve_i \rbrace$ of $T_p\mathcal{M}$, so $E \defeq [\ve_1,\dots,\ve_n]\in \mathbb{R}^{n\times n}$ is orthogonal, gives the chart
\begin{equation} \label{eq:nml_coords}
\varphi^{-1} \defeq E^{\top} \circ \text{exp}^{-1}_p:\mathcal{U} \rightarrow T_p\mathcal{M} \cong \mathbb{R}^n
\end{equation}
and parametrization
\begin{equation} \label{eq:nml_params}
\varphi \defeq \text{exp}_p \circ E:T_p\mathcal{M} \cong\mathbb{R}^n \rightarrow \mathcal{U}.
\end{equation}
By naturality of the exponential \cite{Lee1997}, $\text{exp}^{-1}_p(q)$ is identified with a vector in ambient coordinates, so composition in \eqref{eq:nml_coords} is simple matrix-vector multiplication. In a data-driven setting $E$ may be given by an approximated basis \cite{fletcher2004principal} or some other basis of interest. Figure \ref{fig:nml_neighborhood} shows an elliptical normal neighborhood on the $2$-sphere.

\begin{figure}[t]
	\centering
	\includegraphics[width=0.4\textwidth]{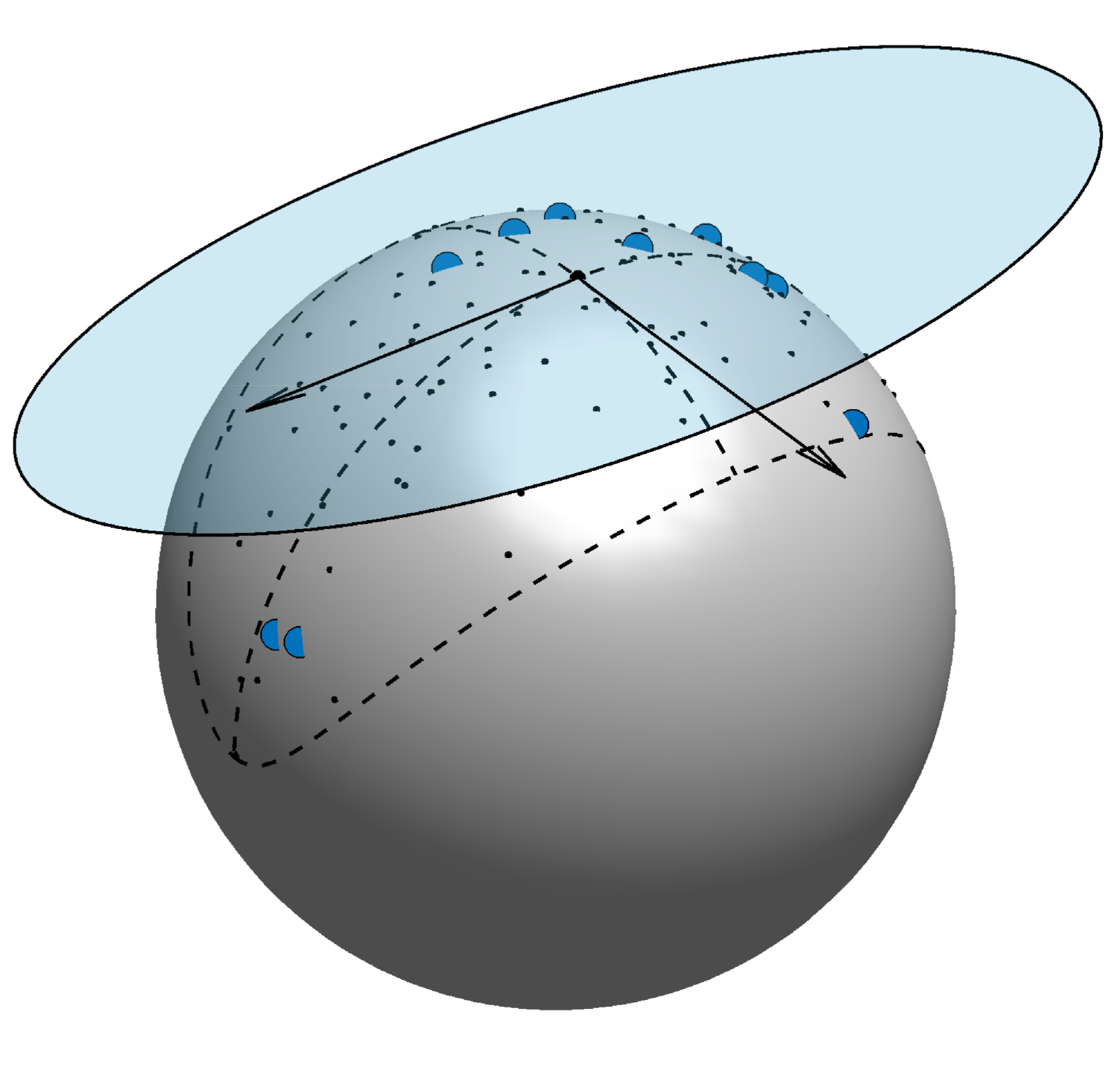}
	\caption{A sphere, $\iota(S^2) \subset \mathbb{R}^3$, shown with elliptical neighborhood and associated image under the exponential map.}
	\label{fig:nml_neighborhood}
\end{figure}

The normal coordinate chart \eqref{eq:nml_coords} has three properties which make it appealing:
\begin{proposition}[\cite{Lee1997}, some of Prop. 5.11] \label{prop:nml_props}
 (Some properties of normal coordinates) Let $\left( \mathcal{U}, x^i \right)$ be any normal coordinate chart [defined by any invertible $E \in \mathbb{R}^{n\times n}$] centered at $p \in \mathcal{U}$.
\begin{enumerate}
	\item For any $\vv_i \in T_p\mathcal{M}$, the geodesic $\gamma_{\vv_i}$ starting at $p$ with initial velocity vector $\vv_i$, i.e., $\text{exp}_p(\vv_i)$, is represented in normal coordinates by the radial line segment (compact subset of a $1$-dimensional subspace)
	$$
	\gamma_{\vv_i} = t\vv_i = \left(t\vv^1,\dots,t\vv^n\right)_i^{\top} \in \mathbb{R}^n
	$$
	as long as $\gamma_{\vv_i}$ stays within $\mathcal{U}$.
	\item The coordinates of $p$ are $\boldsymbol{0} \defeq (0,\dots,0)^{\top} \in \mathbb{R}^n$.
	\item The components of the metric at $p$ are $g_{ij} = \delta_{ij}$.
\end{enumerate}
\end{proposition}

By these properties, the parallel transport to the origin used in \eqref{eq:C_par_trans} reduces to a choice of center $p_0 \in \mathcal{M}$, whose normal coordinates are $\boldsymbol{0}$---the intrinsic analog of centering a Euclidean domain on the origin. In a data-driven setting $p_0$ is the Karcher mean, the basis construction motivated by \cite{fletcher2004principal} offers a choice of $E$ in this central tangent space. Otherwise, we could select a nominal design with an arbitrary orthogonal $E$ around a point of interest, $p_0$. The option could be motivated by the application of interest but computing centers is a challenging inference in and of itself.

\subsection{Euclidean AMG}
When $\mathcal{M} \defeq \mathbb{R}^n$, parallel transport $\overline{P}_{t,t_0}$ is the identity (the Euclidean connection leaves vector entries constant \cite{Lee1997}), which is precisely what lets a tangent vector $\vv_x \in T_x\mathbb{R}^n$ be transported to the origin and treated as a free vector. Our component-wise integration \eqref{eq:C} is then reinterpreted as
\begin{equation} \label{eq:C_par_trans}
C\defeq \int_{\mathcal{X}} \overline{\mathcal{P}}^{-1}_{t_0,t}\left[\overline{\nabla}f(x)\right] \otimes \overline{\mathcal{P}}^{-1}_{t_0,t}\left[\overline{\nabla}f(x)\right]\rho(x)dx^1\dots dx^n
\end{equation}
over arbitrary $\gamma:\mathcal{I} \rightarrow \mathbb{R}^n$ with $\gamma(t_0) = 0$, $\gamma(t) = x$, and $\overline{\mathcal{P}}^{-1}_{t_0,t} = \overline{\mathcal{P}}_{t,t_0} = I_n$.
\begin{remark}
	Although the parallel vector field along any $\gamma$ is unique, parallel transport \textit{depends on the choice of $\gamma$}, and on the metric, in general. The exception is a Euclidean space, with $\ell_2$ inner product, where unique $\overline{\mathcal{P}}_{t,t_0} = I_n$ is independent of the path.
\end{remark}
The integrand of \eqref{eq:C_par_trans} is independent of $\gamma$ here only because the Euclidean connection has zero Christoffel symbols (zero curvature) \cite{Lee1997}. To obtain a unique analog on a curved $(\mathcal{M},g)$, we compensate for this path dependence by transporting along \textit{unique geodesics}. We'll refer to $\gamma$ as the unique geodesic and $\mathcal{P}_{t,t_0}$ parallel transport along it from $\gamma(t)$ to $\gamma(t_0)$, with computations carried out in normal coordinates through composition with $\text{exp}_{p_0}$.

\subsection{Active manifold-geodesics}
\label{sec:AMG}

Stemming from the development of \eqref{eq:C_par_trans}, we can define an extension serving as the analogous average tensor (outer) product of the gradient. To retain the interpretation that the eigenvalues of \eqref{eq:C} correspond to an ordering of mean squared directional derivatives \eqref{eq:eigvals}, compose with an isometry to preserve inner products inherent to the definition of directional derivatives---alluded to in the form of \eqref{eq:C_par_trans}. A choice of isometry retains the original intuition associated with the definition \eqref{eq:C}. That is, the directional derivative at any point in Def. \ref{def:riemann_grad} is unmodified after applying the isometry and thus eigenvalues are analogously interpreted as mean squared differentials (directional derivatives) \eqref{eq:eigvals}. This leads to Def. \ref{def:riemannian_view} as \textit{a Riemannian view on active subspaces}:

\begin{definition}[Riemannian average tensor product of the gradient]
	\label{def:riemannian_view}
	Given a Riemannian $n$-manifold $(\mathcal{M},g)$ with the linear connection $\nabla$ as the Riemannian (Levi-Civita) connection and a central point $p_0 \in \mathcal{M}$, let $\mu$ be the wrapped probability measure on $\mathcal{M}$ obtained by pushing forward, under $\text{exp}_{p_0}$, a Borel probability measure on $T_{p_0}\mathcal{M}$ whose star-shaped support lies within the cut locus of $p_0$, and write $\mathcal{X} \defeq \operatorname{supp}\mu \subseteq \mathcal{M}$. Each $p \in \mathcal{X}$ is joined to $p_0$ by a unique geodesic, so
		$$
		\mathcal{P}_{p}:T_{p}\mathcal{M} \rightarrow  T_{p_0}\mathcal{M}
		$$
	is radial parallel transport along the geodesic---a radial abbreviation of general $\mathcal{P}_{t,t_0}$ in section~\ref{sec:intro-Riemannian-geo}. Then, for sufficiently smooth $f:\mathcal{X} \rightarrow \mathbb{R}$ admitting vector field $\nabla_{\mathcal{M}} f:\mathcal{X} \rightarrow \mathcal{T}(\mathcal{M})$, we define $G_0$ at the central tangent space $T_{p_0}\mathcal{M}$ as
	\begin{equation} \label{eq:riemannian_view}
	G_0 = \int_{\mathcal{X}} \mathcal{P}_{p}\left[\nabla_{\mathcal{M}} f(p)\right]\otimes \mathcal{P}_{p}\left[\nabla_{\mathcal{M}} f(p)\right]\,d\mu(p).
	\end{equation}
\end{definition}

Points $p \in \mathcal{M}$ remain abstract until a computation demands coordinates: $x \defeq \varphi^{-1}(p) \in T_{p_0}\mathcal{M} \cong \mathbb{R}^n$ denotes the normal-coordinate representation of $p$ (so $\varphi^{-1}(p_0) = \boldsymbol{0}$), and $\widehat{x} \defeq \iota(p)$ its ambient representation under an isometric embedding. Each transported gradient $\mathcal{P}_{p}[\nabla_{\mathcal{M}} f(p)]$ lies in the \textit{single} central tangent space $T_{p_0}\mathcal{M}$, where the metric in the orthonormal frame is the identity, $g_{p_0}=I_n$ (Prop.~\ref{prop:nml_props}, property~3)---a statement at the central point alone; away from $p_0$ the metric departs from $I_n$ at second order through curvature \cite{doCarmo2017}. 

It is worth clarifying why a metric that is the identity \textit{only} at $p_0$ suffices to
realize $G_0$ as a matrix on a flat space. Parallel transport carries each Riemannian gradient $\nabla_{\mathcal{M}}f(p)\in T_p\mathcal{M}$ along its radial geodesic into the central tangent space $T_{p_0}\mathcal{M}$, an isometry with respect to $g$. As such, every transported vector $\mathcal{P}_p[\nabla_{\mathcal{M}}f(p)]$ is thereby anchored at the origin, where $g_{p_0}=I_n$. Then, we flatten the normal neighborhood by identifying its image under $\varphi^{-1}=\text{exp}^{-1}_{p_0}$ with the coordinate space $T_{p_0}\mathcal{M}\cong\mathbb{R}^n$, and endow that coordinate space with the constant Euclidean metric $I_n$ at \textit{every} point, not merely at the origin. 

This \textit{constant extension} is canonical rather than a convenience: as noted in section~\ref{sec:intro-Riemannian-geo}, the Euclidean connection has unique path-independent parallel transport $\overline{\mathcal{P}}_{t,t_0}=I_n$ agnostic to the path along which a vector is carried. So a vector rooted at the origin extends to a well-defined free vector rooted at each coordinate point $x=\varphi^{-1}(p)$, and $I_n$ is unambiguously the induced inner product structure at every such root. Without this second step the average \eqref{eq:riemannian_view} would read as an integral of vectors all pinned to the single point $p_0$---a set of $\mu$-measure zero---rather than as a second-moment over the flattened domain.

The curvature of $\mathcal{M}$ does not disappear under this flattening---it is carried entirely by
the transport $\mathcal{P}_p$, which is where all the geometry resides. What the construction avoids
is any explicit encounter with the Riemannian volume element. The Riemannian-volume average is the
special case $d\mu=\sqrt{\det G_x}\,dx$ in normal coordinates, but $\mu$ need not be this measure:
in practice one selects any full-support distribution on the closed geodesic ball
$\overline{B}_\delta(\boldsymbol{0})\subset T_{p_0}\mathcal{M}$---a bump or indicator density---samples it in the flat coordinate space, and estimates \eqref{eq:riemannian_view} by Monte Carlo
(Algorithm~\ref{alg:riemannian_MC_AS}), never forming $\sqrt{\det G_x}$. In this sense $\mu$ plays the role of the weight $\rho$ in \eqref{eq:C}, and the wrapped construction lets us integrate in flat coordinates while remaining agnostic to the volume form --- curvature is represented through the transported integrand.

The rank-one integrands are therefore symmetric positive semidefinite tensors on one fixed Euclidean space, and what is \textit{not} Euclidean is the measure: $\mu$ is, by construction, a \textit{wrapped distribution} \cite{mallasto2018wrapped,pennec2006intrinsic}---drawn in the tangent space and pushed onto $\mathcal{M}$ by $\text{exp}_{p_0}$, the simplest natural device when sampling against the Riemannian volume on $\mathcal{M}$ becomes intractable or unknown. Rescaling of the tangent-space factor keeps $G_0$ independent of the chosen orthonormal frame when assumptions of zero mean and identity covariance are useful~\cite{Constantine2015exploiting}.

Identify $\mathcal{P}_{p}\left[\nabla_{\mathcal{M}} f(p)\right] \in T_{p_0}\mathcal{M}$ as column vectors in $\mathbb{R}^n$, so $\text{Range}(G_0)$ is at most $n$-dimensional. Within the injectivity radius the geodesic $p_0\!\to\!p$, and hence its parallel transport, is unique, so the integrand---and therefore $G_0(\mu,p_0)$---is well-defined; the Monte Carlo estimate of \eqref{eq:riemannian_view} is integration against the empirical measure $\widehat\mu_N=\tfrac1N\sum_i\delta_{p_i}$, realized by sampling only in the tangent space (Algorithm~\ref{alg:riemannian_MC_AS}). The support constraint is explicit for many manifolds of interest---injectivity radius $\pi$ for $S^2$, $\pi/2$ for the Grassmannian \cite{edelman1998geometry}, and infinite for the nonpositively curved SPD manifold---so for the separable shape tensors \cite{grey2023separable, grey2025explainable} (a product of a Grassmannian and an SPD factor) the binding radius is $\pi/2$, and samples wrapped around a Fr\'echet mean within it satisfy this constraint by construction.

This restriction is not a shortcoming of the construction but an intrinsic phenomenon: the cut locus bounds any single normal neighborhood. It suggests a natural partitioning of the domain---covering $\mathcal{M}$ with normal neighborhoods anchored at several central points, each carrying its own $G_0$---so that a function is studied globally through an atlas of local analyses. This constitutes an explicit, intrinsic regularization scheme for finding important nonlinear paths through an embedding to be studied further in future work.

Equation \eqref{eq:riemannian_view} admits a non-negative eigendecomposition $G_0 = W_0\Lambda_0W_0^{\top}$ in the central tangent space $T_{p_0}\mathcal{M}$, which clearly depends on the choice of central (origin) point $p_0 \in \mathcal{M}$. However, the interpretation \eqref{eq:eigvals} of the eigenvalues generalizes \textit{exactly}---not merely by analogy. Extend each eigenvector $\vw_i$ over the support by inverse transport along the radial geodesics,
\begin{equation} \label{eq:AMG_field}
\vw_i(p) \defeq \mathcal{P}^{-1}_{p}[\vw_i] \in T_p\mathcal{M},
\end{equation}
so that $\vw_i(p_0)=\vw_i$. Then:
\begin{lemma}[Exact spectral interpretation] \label{lem:exact_eigvals}
	Let $G_0 = W_0\Lambda_0W_0^{\top}$ per Def.~\ref{def:riemannian_view}. For $i=1,\dots,n$,
	$$
	\lambda_i = \int_{\mathcal{X}} \left(df_p\left[\vw_i(p)\right]\right)^2 d\mu(p)
	$$
	are ordered mean-squared directional derivative of $f$ along the transported eigenvector field \eqref{eq:AMG_field}.
\end{lemma}
Lemma~\ref{lem:exact_eigvals} is the justification for composing with parallel transport: the eigenvalues of $G_0$ are mean-squared differentials of $f$ measured \textit{on the manifold}, pointwise along a canonical extension of the central directions. It is also intrinsic and subsequently independent of any inherent embedding of the manifold. Computationally these intrinsic operations may still be \textit{evaluated} from an ambient realization, by naturality of the Riemannian connection and exponential \cite{Lee1997}, as in section~\ref{sec:eg_2sphere}. Properties of the fields \eqref{eq:AMG_field} are developed in section~\ref{sec:AMG_field}.
 \begin{remark}
    In pursuit of an exact generalization of \eqref{eq:C} such that we preserve the interpretable property \eqref{eq:eigvals}, we require composition with an isometry and there is no more natural (intrinsic) choice than parallel transport---a choice made exact by Lemma~\ref{lem:exact_eigvals}.
\end{remark}

We interpret \eqref{eq:riemannian_view} as a natural extension of \eqref{eq:C} over Riemannian $n$-manifold $(\mathcal{M},g)$ with an appropriate domain restricted to a normal neighborhood but recognize that the choices inherent to Def. \ref{def:riemannian_view} may not extend well in a general tensor integration setting. We define \textit{active manifold-geodesics} (AMG) as

\begin{definition}[Active Manifold-Geodesics] \label{def:AMG}
	Given the eigendecomposition $G_0 = W_0\Lambda_0 W_0^{\top}$ such that $W_0 \in \mathbb{R}^{n\times n}$ orthogonal and $\Lambda_0 = \text{diag}(\lambda_1,\dots,\lambda_n)$, the first $r$-columns of $W_0$, $W_r \defeq [\vw_1,\dots,\vw_r] \in \mathbb{R}^{n\times r}$, ordered according to $\lambda_1 \geq \dots \geq \lambda_r >\lambda_{r+1} = \dots = \lambda_n = 0$ induce an ordering of active manifold-geodesics (AMG),
	\begin{equation} \label{eq:AMG}
	\mathcal{W}_{1,i} \defeq \lbrace \text{exp}_{p_0}\left(\mathcal{V} \cap \mathcal{A}_{1,i}\right) \,:\, \mathcal{A}_{1,i} = \text{span}[\vw_i] \rbrace,
	\end{equation}
	serving as an orthogonal frame for the active $r$-submanifold of the normal neighborhood,
	\begin{equation} \label{eq:r_AMGsubmanifold}
	\mathcal{W}_r \defeq \lbrace \text{exp}_{p_0}\left(\mathcal{V} \cap \mathcal{A}_r\right) \,:\, \mathcal{A}_r = \text{span}[W_r]\rbrace,
	\end{equation}
	as the image under $\text{exp}_{p_0}$ diffeomorphic to $\mathcal{V} \cap \mathcal{A}_r \subseteq T_{p_0}\mathcal{M}$ for open $\mathcal{V} \subset T_{p_0}\mathcal{M}$ containing the origin.
\end{definition}

Given apparent choices about the design of $\mathcal{X}$ in applications, we emphasize the utility of taking $\mathcal{X}$ as a closed geodesic ball under $\overline{B}_{\delta}(\boldsymbol{0}) \subset \mathcal{V} \subset T_{p_0}\mathcal{M}$ with radius less than or equal to the injectivity radius for any particular computational model---e.g., choosing $\mu$ supported on $\overline{B}_{\delta}(\boldsymbol{0})$ with a (normalized) bump or indicator density for some radius $\delta >0$ such that $\mathcal{X} \subseteq \text{exp}_{p_0}\left(\overline{B}_{\delta}(\boldsymbol{0})\right)$ remains a diffeomorphism. 

Def. \ref{def:riemannian_view} and Def. \ref{def:AMG} are still ``locally'' defined with respect to the manifold since uniqueness is only retained by restricting to local geodesic balls and thus $\mathcal{X}$ is necessarily star-shaped \cite{Lee1997}. If a Riemannian manifold is geodesically complete \cite{Lee1997,Absil2008}, we have---at a minimum---introduced a globalizing notion for the expectation (integral) of a $(0,2)$-tensor by virtue of parallel transport over $\mathcal{T}(\mathcal{M})$. However, to retain an analogous unique integrand, $\mathcal{X}$ is at most a compact subset of the $n$-manifold so the function is restricted to a normal neighborhood of maximal geodesics. The apparent locality of $\mathcal{X} \subseteq \text{exp}_{p_0}(\mathcal{V}) \subseteq \mathcal{M}$ is only as limiting as the data: whenever the sampled domain lies within the injectivity radius of $p_0$---intrinsically local in this precise sense---the restriction to a single normal neighborhood captures the response faithfully, and the reduction is meaningful for exactly that data.

\subsection{The extrinsic perspective}
\label{sec:extrinsic}
Similarly, choosing isometric embedding $\iota:\mathcal{M} \hookrightarrow \mathbb{R}^{m}$, and integrating against the common wrapped measure $\mu$ so that the two constructions are directly comparable (Figure~\ref{fig:intrinsic_extrinsic}), we can also write
\begin{equation} \label{eq:emb_opg}
C_{\iota} \defeq \int_{\mathcal{X}} d\iota\left[\nabla_{\mathcal{M}} f(p)\right]\otimes d\iota\left[\nabla_{\mathcal{M}} f(p)\right] d\mu(p)
\end{equation}
such that $C_{\iota} = W_{\iota}\Lambda_{\iota} W_{\iota}^{\top}$ is the associated non-negative eigendecomposition explicitly dependent on the choice of $\iota$. This definition, which also reduces to \eqref{eq:C} in the case of a Euclidean space by selecting $\iota$ such that $d\iota \defeq I_n$, involves a \textit{choice} of non-unique isometric embedding and thus non-unique integrand in light of questions around rigidity \cite{hang2009rigidity,berger1981some}---i.e., certain rigid motion reparametrizing the isometric embedding could serve as an advantage or disadvantage in the numerical approximation of \eqref{eq:emb_opg}. The gradients of \eqref{eq:emb_opg} are unique by naturality of the Riemannian connection but the integrand of \eqref{eq:emb_opg} is not given a choice of $\iota$.

\begin{figure*}[t]
	\centering
	\includegraphics[width=0.8\textwidth,trim={0 100pt 0 0},clip]{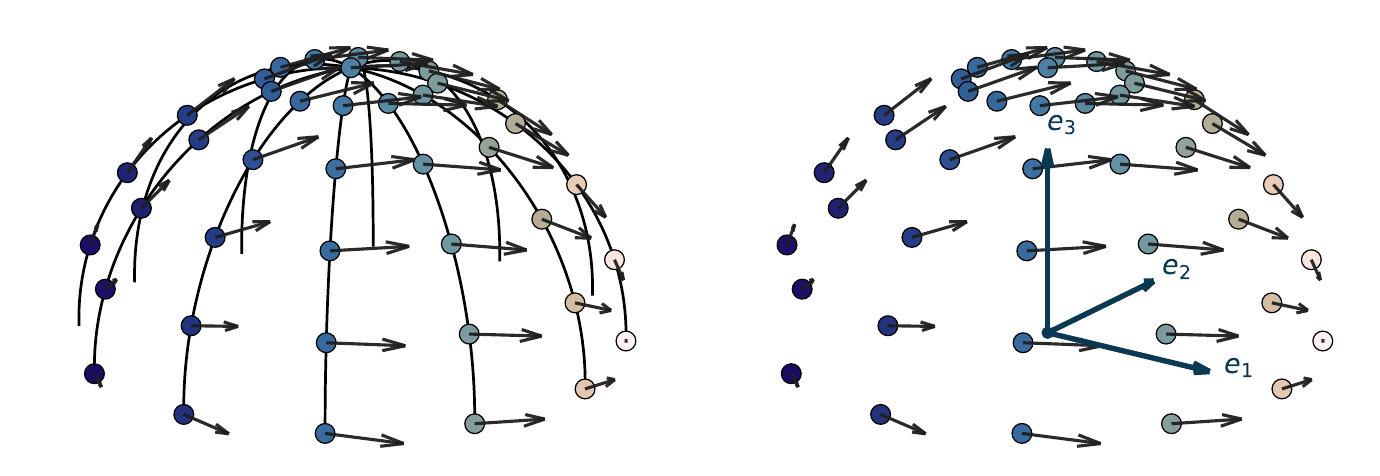}
	\caption{The two fields of the \textit{same} tangential gradients over a hemisphere of $\iota(S^2)$. Colors correspond to function values registered at the same samples across panels. \textbf{Left (intrinsic)}: the construction \eqref{eq:riemannian_view} carries each Riemannian gradient to the central tangent space by parallel transport along its unique radial geodesic (black curves through the central point). \textbf{Right (extrinsic)}: the construction \eqref{eq:emb_opg} reads the same gradients as vectors in the ambient coordinates $(\ve_1,\ve_2,\ve_3)$ fixed by the embedding $\iota$. An extrinsic perspective subsequently requires compression to the central tangent space (Thm.~\ref{thm:normal_eigvals}).}
	\label{fig:intrinsic_extrinsic}
\end{figure*}

Equation~\eqref{eq:emb_opg} is precisely the manifold gradient outer product of \cite{wu2010learning,mukherjee2010learning}: the ambient average of the \textit{tangential} gradients $d\iota[\nabla_{\mathcal{M}}f]$. Its dominant eigenvectors recover the ambient dimension-reduction directions (\cite{wu2010learning}, Prop.~11)---an \textit{ambient} basis that must be projected to a central tangent space to compare with the intrinsic $G_0$ of \eqref{eq:riemannian_view}. The following theorem characterizes that projection:

\begin{theorem} \label{thm:normal_eigvals}
    Let $\iota:\mathcal{M} \hookrightarrow \mathbb{R}^m$ ($m>n$) be an isometric embedding, $\widehat{x}_0$ a central point, $E_0\in\mathbb{R}^{m\times n}$ an orthonormal basis of the central tangent space $T_{\widehat{x}_0}\iota(\mathcal{M})$ in ambient coordinates, and $\pi_0\defeq E_0E_0^{\top}$ the orthogonal projection onto that tangent space. Then the representation of $C_{\iota}$ on the central tangent space is the second moment of the \textit{centrally-projected} gradients,
    $$
    E_0^{\top} C_{\iota} E_0 = \int_{\mathcal{X}}\left(E_0^{\top} d\iota[\nabla_{\mathcal{M}} f]\right)\otimes\left(E_0^{\top} d\iota[\nabla_{\mathcal{M}} f]\right) d\mu(p);
    $$
    in particular, for every central direction $\widehat{\vw}=E_0\vw\in T_{\widehat{x}_0}\iota(\mathcal{M})$,
    $$
    \vw^{\top}\!\left(E_0^{\top} C_{\iota} E_0\right)\!\vw = \int_{\mathcal{X}}\left(\widehat{\vw}^{\top} \pi_0\,d\iota[\nabla_{\mathcal{M}} f]\right)^2 d\mu(p).
    $$
\end{theorem}
The eigenvalues of $E_0^{\top} C_{\iota} E_0$ are therefore the mean-squared directional derivatives of the \textit{centrally-projected gradient} $\pi_0\,d\iota[\nabla_{\mathcal{M}} f]$---projections between tangent spaces which are identity mappings in the Euclidean setting of section \ref{sec:AS_intro}. Since $E_0^{\top}$ is a partial isometry, this projection is an isometry on $T_{\widehat{x}_0}\iota(\mathcal{M})$ but annihilates its orthogonal complement: no gradient component orthogonal to the central tangent space contributes, and an ambient direction $\widehat{\vb}$ with $\pi_0\widehat{\vb}=\boldsymbol{0}$ carries no weight in the spectrum. The conflation depends on both the selected $\iota$ and the orientation of the central tangent space in the ambient Euclidean space. By contrast the intrinsic $G_0$ of \eqref{eq:riemannian_view} carries each gradient to the central tangent space by parallel transport---an isometry---and so loses no directional derivative (Lemma~\ref{lem:exact_eigvals}).
\begin{remark}
    The utility in the interpretation \eqref{eq:riemannian_view} is \textit{retaining an ordering} of important directions in a central tangent space devoid of a partial isometry.
\end{remark}

\subsection{Local analysis: projection versus transport}
\label{sec:local_analysis}
The two constructions assemble the \textit{same} tangential gradients and differ only in the map carrying each gradient to the central tangent space: the intrinsic \eqref{eq:riemannian_view} applies the isometry $\mathcal{P}_p$, while the central representation of the extrinsic \eqref{eq:emb_opg} applies the restriction of $\pi_0$ to $T_{\widehat{x}}\iota(\mathcal{M})$ (Thm.~\ref{thm:normal_eigvals}). On hyperspheres the discrepancy between the two maps admits a closed form. Identities invoked below are collected with proofs in Appendix~\ref{app:sphere}.

\begin{lemma}[Projection versus transport on hyperspheres] \label{lem:sphere_proj_trans}
	Let $\iota(S^n)\subset\mathbb{R}^{n+1}$ be the unit hypersphere with central point $\widehat{x}_0$, and let $\widehat{x}\in\iota(S^n)$ have geodesic distance $d\defeq d(\widehat{x},\widehat{x}_0)<\pi$. If $\vu\in T_{\widehat{x}}\iota(S^n)$ the unit radial direction (tangent at $\widehat{x}$ to the geodesic from $\widehat{x}_0$) then, for every $\vv \in T_{\widehat{x}}\iota(S^n)$,
	$$
	\pi_0\,\vv = \mathcal{P}_{\widehat{x}}[\vv] - (1-\cos d)\,\langle \vv,\vu\rangle\, \mathcal{P}_{\widehat{x}}[\vu],
	$$
	central projection is parallel transport composed with a $\cos d$ contraction of the radial component. Transverse components are transported exactly.
\end{lemma}
The identity implies a generic, second-order agreement of the two perspectives---with hypersphere constants explicit and, on general manifolds, constants set by how the embedding curves:
\begin{proposition}[Local agreement of the two perspectives] \label{prop:local_agreement}
	Let $\mathcal{X}$ lie within the geodesic ball of radius $R$ about $p_0$, with $R$ less than the injectivity radius, and let $G_0$ and $C_\iota$ be defined against the same $\mu$. Then

	\begin{align*}
		\left\Vert E_0^{\top}C_{\iota}E_0 - G_0 \right\Vert_2 &\leq C \int_{\mathcal{X}} d^2(p,p_0)\, \Vert \nabla_{\mathcal{M}}f(p)\Vert^2_g \, d\mu(p) \\
		& \leq C\, R^2 \sup_{\mathcal{X}}\Vert\nabla_{\mathcal{M}}f\Vert_g^2,
	\end{align*}

	where $C$ depends only on the second fundamental form of $\iota(\mathcal{M})$ and its curvature over the support; for the unit hypersphere one may take $C=1+R^2/4$, via Lemma~\ref{lem:sphere_proj_trans}. Consequently the eigenvalues of the two central representations agree to $\mathcal{O}(R^2)$ (Weyl), and whenever $G_0$ has a spectral gap $\eta \defeq \lambda_r - \lambda_{r+1} > 0$ the dominant eigenspaces agree to $\mathcal{O}(R^2/\eta)$ (Davis--Kahan \cite{Golub1996}).
\end{proposition}

\begin{remark}[When the perspectives are equivalent] \label{rem:equivalence}
	Proposition~\ref{prop:local_agreement} separates three statements. (i) As \textit{matrices on the central tangent space}, the extrinsic representation $E_0^{\top}C_{\iota}E_0$ and the intrinsic $G_0$ always agree to $\mathcal{O}(R^2)$---with no exceptional set---and exactly when $\iota(\mathcal{M})$ is an affine subspace (the Euclidean case $d\iota = I_n$, recovering \eqref{eq:C}). (ii) Identifying a \textit{dominant direction} from either matrix additionally requires a spectral gap exceeding the $\mathcal{O}(R^2)$ discrepancy. (iii) The common extrinsic \textit{practice} \cite{mukherjee2010learning,wu2010learning}---eigendecompose the ambient $m\times m$ matrix $C_\iota$ first, then project its dominant eigenvector $\widehat{\vb}$ to the central tangent space---additionally requires $\pi_0\widehat{\vb}\neq\boldsymbol{0}$: the ambient eigen-ordering weighs gradient components that $\pi_0$ annihilates (Thm.~\ref{thm:normal_eigvals}), so a direction dominant in ambient space can be tangentially invisible. It is (iii), not (i), that can fail: where $\pi_0\widehat{\vb}=\boldsymbol{0}$, the projection annihilates the ambient importance carried in its nullspace. This is the measure-zero aligned set for the sphere ridge, but the failure set is geometry-dependent and need not have measure zero in general; the central-representation agreement (i) holds regardless. The examples of section~\ref{sec:eg_2sphere} exhibit each consideration.
\end{remark}

\begin{remark}[Projected dominance versus dominant projection] \label{rem:noncommute}
	Approximating activity over a manifold asks two things in turn. Activity must first be identified, by integrating a gradient second moment---either \eqref{eq:riemannian_view} or \eqref{eq:emb_opg}---and eigendecomposing the result; it must then be extended over the manifold, either by recomputing a decomposition on each new tangent space or by parallel transport of the central frame (section~\ref{sec:AMG_field}). For the first task the two perspectives agree quadratically. The constant $E_0^{\top}$ commutes with integration (Thm.~\ref{thm:normal_eigvals}), so $E_0^{\top}C_{\iota}E_0$ is itself a second moment of centrally-projected gradients, and its leading eigenvector---the \textit{dominant projection}---recovers the intrinsic eigenvalues to $\mathcal{O}(R^2)$ and the dominant eigenspace to $\mathcal{O}(R^2/\eta)$ (Prop.~\ref{prop:local_agreement}). What that agreement does not license is \textit{projected dominance}: eigendecomposing the ambient $C_{\iota} = \sum_{i=1}^m \lambda_i\,\vw_i\vw_i^{\top}$ first and projecting its dominant eigenvector afterward. On the central tangent space,
	$$
	E_0^{\top}C_{\iota}E_0 = \sum_{i=1}^m \lambda_i\,\vu_i\vu_i^{\top}, \qquad \vu_i \defeq E_0^{\top}\vw_i,
	$$
	is a sum over $m > n$ directions that are linearly dependent and non-orthogonal, each scaled by $\Vert\vu_i\Vert = \cos\theta_i$ for $\theta_i$ the angle between $\vw_i$ and the central tangent space. The weight carried by direction $i$ is therefore $\lambda_i\cos^2\theta_i$, so ambient dominance leaning toward the normal space is overtaken, and at $\theta_1 = \pi/2$ it is annihilated outright (Thm.~\ref{thm:normal_eigvals}). Eigenvalues survive this compression in the weak sense of interlacing, $\lambda_i(C_{\iota}) \geq \lambda_i(E_0^{\top}C_{\iota}E_0) \geq \lambda_{i+m-n}(C_{\iota})$ \cite{Golub1996}, but the eigenvectors do not transfer at all. Projected dominance is not dominant projection, and activity should never be \textit{extended} by it.
\end{remark}

\begin{remark}
    The extrinsic collapse of the partial isometry---wherever the dominant ambient direction meets $\text{Null}(\pi_0)$---is by no means catastrophic. Vanishing components annihilated by the projection may simply require a rigid action on the embedding---i.e., a different choice of $\iota$.
\end{remark}

The same second-order locality governs how the Euclidean theory of section~\ref{sec:AS_intro} transfers to normal coordinates. Applied to the pullback $f\circ\text{exp}_{p_0}$ over the normal-coordinate domain $\text{exp}^{-1}_{p_0}(\mathcal{X})$, the ridge-approximation bound \eqref{eq:err_apprx} employs the eigenpairs of the coordinate-gradient average, which agree with those of $G_0$ to $\mathcal{O}(R^2)$ by the identical frame-comparison argument. The same bound~\eqref{eq:err_apprx} therefore holds on the manifold up to $\mathcal{O}(R^2)$ corrections, exact in the limit $R \to 0$. The same rate governs the local recovery of ambient ridge structure in normal coordinates, demonstrated numerically in section~\ref{sec:eg_2sphere}.

\subsection{The AMG frame field}
\label{sec:AMG_field}
The extensions \eqref{eq:AMG_field} deserve study in their own right. Since geodesics within the injectivity radius depend smoothly on their endpoint and parallel transport preserves inner products, $\{\vw_1(p),\dots,\vw_n(p)\}$ is a smooth orthonormal frame field over the normal neighborhood of $p_0$: every point receives a full set of ordered directions, with $\lambda_i$ the mean-squared differential along the $i$th field (Lemma~\ref{lem:exact_eigvals}). Moreover, because a geodesic parallel-transports its own velocity, $\vw_i(\text{exp}_{p_0}(t\vw_i)) = \tfrac{d}{dt}\,\text{exp}_{p_0}(t\vw_i)$: each active manifold-geodesic \eqref{eq:AMG} is the integral curve of its own eigenvector field through $p_0$. Integral curves through other points are generally \textit{not} geodesics, and we deliberately retain the exponential images of Def.~\ref{def:AMG}---parametrized by straight lines in a fixed tangent space---rather than flowing the fields (cf.\ the degeneracy risk of fully nonlinear reductions, section~\ref{sec:lit_review}).

\begin{remark}[Ill-posedness of unconstrained curve-based reduction] \label{rem:dichotomy}
	Fully nonlinear methods---active manifolds \cite{bridges2019active} and neural variants \cite{zanoni2025neural}---search the domain for a single curve $\gamma$ such that $f \approx h \circ \pi_{\gamma}$, with $\pi_{\gamma}$ a projection onto the curve. Without an explicit constraint on the complexity of $\gamma$, this search is ill-posed. For any continuous $f$ and any tube radius $\tau>0$ a space-filling sweep of the domain reproduces $f$ within its modulus of continuity $\omega_f(\tau)$, so the accuracy $\epsilon\geq \omega_f(\tau)$ of Hypothesis~\ref{hypothesis} is met for \textit{every} $f$ by taking $\tau$ small enough; yet a curve whose $\tau$-tube covers an $n$-dimensional domain must have length of order $\tau^{1-n}$ (Appendix~\ref{app:reach}), so minimizing sequences degenerate (explode in scale) and the vanishing error certifies no genuine one-dimensional structure. Well-posedness requires regularizing the admissible curves, and geodesics are the natural extreme of such regularization: zero covariant acceleration, with extents restricted to the geodesic ball. Def.~\ref{def:AMG} restricts the search accordingly---to exponential images of straight lines in a single tangent space---trading expressiveness for identifiability.
\end{remark}
Given a Riemannian manifold, intrinsic perspectives offer a natural complementary rigidity. If the trailing eigenvalues vanish, $\lambda_{r+1}=\dots=\lambda_n=0$ (with $f\in C^1$ and $\mu$ of full support), then Lemma~\ref{lem:exact_eigvals} gives $\int_{\mathcal{X}} (df_p[\vw_i(p)])^2\,d\mu(p)=0$ for each $i>r$. The squared-differential over the frame, $(df_p[\vw_i(p)])^2$, is nonnegative and continuous ($f\in C^1$ and $\vw_i(p)$ smooth) and $\mu$ has full support. By the argument of Prop.~\ref{prop:eigs_and_constant_f}, full support offers pointwise vanishing $df_p[\vw_i(p)]=0$ for all $p\in\mathcal{X}$. Thus, $f$ is unchanged along every inactive field on $\mathcal{X}$. Assuming $f$ smooth, these invariances close under Lie brackets ($Xf = Yf = 0$ implies $[X,Y]f = X(Yf) - Y(Xf) = 0$), so $df$ is annihilated by the whole distribution, $\overline{\mathcal{F}}$, that the inactive fields Lie-generate. Assuming $\overline{\mathcal{F}}$ has constant rank, only two outcomes are possible. (i) Either the inactive directions and their brackets ``knit'' together into a family of lower-dimensional surfaces filling $\mathcal{X}$---a \textit{foliation}, by the Frobenius theorem \cite{Lee2003}---and $f$ is a function of the leaves (subsets of the foliation): a genuine manifold analog of the ridge \eqref{eq:ridge}, with \textit{leaves} playing the role of the inactive subspace. (ii) Or the inactive directions bracket-generate, reaching every point of a connected $\mathcal{X}$ (Chow--Rashevskii theorem \cite{chow1939systeme,agrachev2020comprehensive}), and $f$ is constant outright.

In this geometric setting, \textit{there is no space-filling middle ground}: a non-constant response with $r$ active directions carries surface-like inactive structure. Additional details are available in Appendix~\ref{app:reach}. The foliation is the exact, global object, non-linearly reducing the domain dimension based on changes in $f$. However, just as Euclidean active subspaces summarize a ridge rather than reconstructing level sets, the AMGs of Def.~\ref{def:AMG} are tractable, interpretable, variance-ordered surrogates over domains which often contain all the input data of interest. AMGs reproduce the leaves exactly only in the flat, transport-invariant case in which the trailing eigenvalues vanish identically, and to $\mathcal{O}(R^2)$ otherwise (Prop.~\ref{prop:local_agreement}). 

The dichotomy of nonlinear dimension reduction versus activity over Riemannian manifolds thus motivates the geodesic regularization---low-complexity, identifiable directions in a single tangent space---rather than offering a rival global construction. The contrast is also problem-dependent. In the setting of nonlinear dimension reduction, we may not have knowledge of an underlying Riemannian manifold topology to leverage as the domain. But when we do, an AMG is a natural regularized restriction and several applications, like those posed over spaces of curves, offer precisely this type of domain~\cite{srivastava2016functional}.

The frame field, $\vw_i(p)$, is parallel along radial geodesics by construction, but it is not a parallel frame: transporting around a closed loop generally rotates it (holonomy; on $S^2$ the rotation is the enclosed area \cite{Lee1997}), so radial transport is part of the definition---canonical given $(\mu,p_0)$ and matching the wrapped construction of Def.~\ref{def:riemannian_view}. The extrinsic analog extends an ambient direction $\widehat{\vb}$ by pointwise orthogonal projection, $\widehat{x} \mapsto \pi_{\widehat{x}}\,\widehat{\vb}$ with $\pi_{\widehat{x}}$ the projection onto $T_{\widehat{x}}\iota(\mathcal{M})$ \cite{wu2010learning}; that field is smooth but degenerates---its magnitude is the cosine of the angle between $\widehat{\vb}$ and the tangent space, vanishing wherever $\widehat{\vb}$ is normal to the manifold (on the sphere, at $\widehat{x}=\pm\widehat{\vb}$: precisely the aligned centers of section~\ref{sec:eg_2sphere})---whereas the intrinsic field \eqref{eq:AMG_field} is unit-norm everywhere. Note also the converse caution inherited from \eqref{eq:eigvals}: an ambient ridge \eqref{eq:ridge} does not, in general, restrict to a manifold ridge---the transported inactive fields need not annihilate $df$ even when the ambient trailing eigenvalues vanish. Regardless, ambient ridge structure is still recovered locally at the $\mathcal{O}(R^2)$ rate of the same second-order locality (Prop.~\ref{prop:local_agreement}; Figures~\ref{fig:amg_ridge_recovery} and~\ref{fig:amg_ridge_recovery_nl}). In applications, the dominant field $\vw_1(p)$ assigns to every configuration $p$ its most influential perturbation direction: a sensitivity vector field over $\mathcal{X}$.

Consider the extension task of Remark~\ref{rem:noncommute} with both perspectives informed by a \textit{single} integration. Intrinsically, one central eigendecomposition extends canonically: the transported frame \eqref{eq:AMG_field} is orthonormal, unit-norm, and consistently ordered over the support, and each $\lambda_i$ retains its exact interpretation (Lemma~\ref{lem:exact_eigvals}). The extrinsic counterpart of that same integration is the projected-dominance field $\widehat{x} \mapsto \pi_{\widehat{x}}\,\widehat{\vb}$, which is inexpensive but degenerates (Figure~\ref{fig:frame_field_compare}) and carries an ambient ordering that need not survive restriction (Remark~\ref{rem:noncommute}). \textit{Activity should never be extended by the partial isometry}. A faithful extrinsic extension must instead recompute the dominant projection, forming a fresh block $E_{\widehat{x}}^{\top}C_{\iota}E_{\widehat{x}}$ and its eigendecomposition at each new tangent space. The comparison is then fairly drawn against a re-centered intrinsic construction---transports and all, at each new point---rather than against the constant transported frame of \eqref{eq:AMG_field}. The latter extensions requiring repeated eigendecompositions in either perspective inherit the same computational budget and define interesting flows for nonlinear extensions.

\subsection{Approximating AMG}
We have not yet addressed the computational challenges of approximating
\eqref{eq:riemannian_view} and the local diffeomorphism $\text{exp}_{p_0}$. The exponential and
inverse-exponential maps solve initial- and boundary-value problems for the geodesic system, a set of $n$ coupled second-order O.D.E.'s whose coefficients---the Christoffel symbols of $g$---are intrinsic. The \textit{number} of solves is fixed by the Monte Carlo sample count, one geodesic and one transport along it per sample, so the slow $\mathcal{O}(N^{-1/2})$ rate governs how many solves buy a target accuracy independent of dimension. The more challenging computational burden is the boundary-value problems---the inverse exponential for normal coordinates and the Karcher mean for centering---whose cost and conditioning degrade in the high-dimensional embeddings typical of shape-space applications. It is approximating these intrinsic maps, far more than the Monte Carlo rate, that bound tractability in practice. Nonetheless, we introduce an algorithm for the approximation of \eqref{eq:riemannian_view} as Algorithm \ref{alg:riemannian_MC_AS} and call on efficient procedures to battle the curse of dimensionality for approximating the intrinsic maps more generally. Many algebraic routines exist for computing or approximating intrinsic maps for manifolds of interest like hyperspheres and matrix manifolds.
\begin{algorithm}[t]
\small
	Given the wrapped measure $\mu$ of Def. \ref{def:riemannian_view} and the map $\nabla_{\mathcal{M}} \widehat{f}_i:\mathbb{R}^m \rightarrow \mathbb{R}^m$ in ambient coordinates,
	\begin{enumerate}
		\item Draw $N$ independent samples $\left\lbrace x_i \right\rbrace_{i=1}^N$ from $\mu$ in the central tangent space and push forward, $\widehat{x}_i = \text{exp}_{\widehat{x}_0}(x_i)$ (or accept input samples $\lbrace\widehat{x}_i\rbrace$ in the data-driven setting).
		\item For each sample, compute $\nabla_{\mathcal{M}} \widehat{f}_i \defeq \nabla_{\mathcal{M}} \widehat{f}(\widehat{x}_i) \in \mathbb{R}^m$ and the quantity of interest $f_i = f(\widehat{x}_i)$.
		\item Compute or approximate parallel transport to the central tangent space,
		$$
		\mathcal{P}_{\widehat{x}_i}\left[\nabla_{\mathcal{M}} \widehat{f}_i\right],
		$$
		along the unique geodesic joining $\widehat{x}_i$ to $\widehat{x}_0$.
		\item Compute the average tensor product of the parallel transported gradients and the associated eigenvalue decomposition,
		\begin{equation*}
		\frac{1}{N}\sum_{i=1}^{N} \mathcal{P}_{\widehat{x}_i}\left[\nabla_{\mathcal{M}} \widehat{f}_i\right] \otimes \mathcal{P}_{\widehat{x}_i}\left[\nabla_{\mathcal{M}} \widehat{f}_i\right] = \tilde{W}_0\tilde{\Lambda}_0\tilde{W}_0^{\top},
		\end{equation*}
		where $\tilde{W}$ is the orthogonal matrix of eigenvectors, and $\tilde{\Lambda} = \text{diag}(\tilde{\lambda}_1,\dots,\tilde{\lambda}_m)$ is the diagonal matrix of eigenvalues in descending order. Equivalently, take the singular value decomposition
		$$
		\frac{1}{\sqrt{N}}\left[\mathcal{P}_{\widehat{x}_1}\left[\nabla_{\mathcal{M}} \widehat{f}_1\right],\dots,\mathcal{P}_{\widehat{x}_N}\left[\nabla_{\mathcal{M}} \widehat{f}_N\right]\right] = \tilde{W}_0\sqrt{\tilde{\Lambda}_0}\tilde{V}_0^{\top}.
		$$
	\end{enumerate}
	\caption{Monte Carlo approximation of $G_0$}
	\label{alg:riemannian_MC_AS}
\end{algorithm}
\normalsize

The average formed in Algorithm~\ref{alg:riemannian_MC_AS} estimates the ambient realization $\widehat{G}_0 \defeq E_0 G_0 E_0^{\top} \in \mathbb{R}^{m\times m}$, not the intrinsic $G_0$ of \eqref{eq:riemannian_view}: each transported gradient is stored in ambient coordinates and lies in $\text{Range}(E_0) = T_{\widehat{x}_0}\iota(\mathcal{M})$, so $\widehat{G}_0$ has rank at most $n$ and shares the nonzero spectrum of the intrinsic $n\times n$ matrix $G_0$ \textit{exactly}, the remaining $m-n$ eigenvalues vanishing up to machine precision (the $\tilde{\lambda}_3$ of section~\ref{sec:eg_2sphere}). Active directions are read from the leading eigenvectors either way; passing to the intrinsic $n\times n$ matrix is the coordinate change $E_0^{\top}(\cdot)\,E_0$.

A variety of supplementary approximations may be necessary to execute Algorithm \ref{alg:riemannian_MC_AS} in practice. The authors \cite{Absil2008} introduce the concept of a \textit{retraction map} which constitutes a local approximation of $\text{exp}_{p_0}$ through projection to the manifold given an isometric embedding in an ambient Euclidean space. A retraction could be used to supplement approximations of the Karcher mean to define the central tangent space and subsequent normal coordinates. The use of such an approximation, if it can be shown to converge to the true map, does not interfere with the uniqueness of the defined integrand \eqref{eq:riemannian_view} given the naturality of the exponential and Riemannian connection \cite{Lee1997}. Additionally, \cite{Absil2008} also introduce the notion of a general vector transport which serve to approximate parallel transport. 
Efforts should be taken to inform appropriate choices for an approximation of the diffeomorphism $\text{exp}_{p_0}$ and parallel transport $\mathcal{P}_{p}$ to balance required accuracy and computational complexity enabling Algorithm \ref{alg:riemannian_MC_AS}.

There is also the possibility that we require some approximation of the Riemannian gradient in local coordinates $\nabla_{\mathcal{M}} \widehat{f}_i$. Fortunately, the local Euclidean nature of the manifold suggests Taylor series may provide useful insights \cite{Lee2003}. The work \cite{mukherjee2010learning} and \cite{wu2010learning} offer approximations of $\nabla_{\mathcal{M}} \widehat{f}_i$ which converge based on the intrinsic dimension of the manifold---as opposed to the potentially inflated dimension of the ambient Euclidean space. Efficient approximations of the gradient are also essential in the case that gradient evaluations of a model are not readily available. 

\section{AMG over the $2$-sphere}
\label{sec:eg_2sphere}

Embedding the sphere as $\iota(S^2) \subset \reals^3$, we compute the intrinsic objects (geodesics, normal coordinates, parallel transport) from the ambient representation via naturality of the exponential map \cite{Absil2008,edelman1998geometry}. Consistent with data being observed extrinsically \cite{wu2010learning,mukherjee2010learning}, we build a \textit{data-driven} approximation of \eqref{eq:riemannian_view} over the $2$-sphere under the Riemannian (tangential) connection. Beyond serving as a visualizable example, the sphere is the geometry of \textit{preshape} spaces: Kendall's landmark preshapes are hyperspheres \cite{kendall1984}, as are length-normalized elastic open curves \cite{mio2007shape,Joshi2007}, so the constructions demonstrated here transfer directly to functions of shape.

\subsection{Exponential map \& projection}
Utilizing the exponential map, we move along geodesics from a point $\widehat{x}\in \iota(S^2)$ in the tangent direction $\widehat{\vv}_{\widehat{x}} \in T_{\widehat{x}}S^2$ at speed $\Vert \widehat{\vv}_{\widehat{x}} \Vert_2$ represented by great circles as points in a subspace spanned by $\widehat{x}$ and $\widehat{\vv}_{\widehat{x}}$. In other words, for a point $\widehat{x} = (\widehat{x}^1,\widehat{x}^2,\widehat{x}^3)^T \in \iota(S^2) \subset \reals^3$ and tangent vector $\widehat{\vv}_{\widehat{x}} = (\widehat{v}^1,\widehat{v}^2,\widehat{v}^3)^T \in T_{\widehat{x}}\iota(S^2) \subset \reals^3$, the exponential map for the unit $2$-sphere is parametrized by 
\begin{equation} \label{eq:sphere_exp}
\text{exp}_{\widehat{x}}(t\widehat{\vv}_{\widehat{x}}) \defeq \widehat{x}\cos(t\Vert \widehat{\vv}_{\widehat{x}} \Vert_2) + \frac{\widehat{\vv}_{\widehat{x}}}{\Vert \widehat{\vv}_{\widehat{x}} \Vert_2}\sin(t\Vert \widehat{\vv}_{\widehat{x}} \Vert_2),
\end{equation}
with $t \in \mathcal{I} \subset \reals$ for some interval $\mathcal{I}$ and the standard Euclidean metric $\Vert \widehat{\vv}_{\widehat{x}} \Vert^2_2 = (\widehat{v}^1)^2 + (\widehat{v}^2)^2 + (\widehat{v}^3)^2$. Again, \eqref{eq:sphere_exp} is simply a (non-unique) parametrization of a circle in the subspace spanned by $\widehat{x}$ and $\widehat{\vv}_{\widehat{x}}$ given the embedding of the unit sphere in $\reals^3$. Note, the expression~\ref{eq:sphere_exp} is dimension-free, holding verbatim on any unit hypersphere $\iota(S^n) \subset \mathbb{R}^{n+1}$. This, together with the closed-form parallel transport and the tangent-space decomposition it induces, is collected in Appendix~\ref{app:sphere}. 

Additionally, we must introduce the inverse of the exponential map in ambient coordinates denoted, $\text{exp}^{-1}_{\widehat{x}}:\widehat{\mathcal{M}} \rightarrow T_{\widehat{x}}\widehat{\mathcal{M}}$, to map geodesic points to a \textit{central tangent space} for subsequent averaging. For the special case $\widehat{\mathcal{M}} \defeq \iota(S^2)$ we can write the inverse exponential for distinct $\widehat{x},\widehat{y} \in \iota(S^2)$ as
\begin{equation} \label{eq:sphere_log}
\text{exp}^{-1}_{\widehat{x}}(\widehat{y}) = \cos^{-1}(\left\langle \widehat{x},\widehat{y}\right\rangle)\frac{\widehat{y} - \left\langle \widehat{x},\widehat{y}\right\rangle \widehat{x}}{\Vert \widehat{y} - \left\langle \widehat{x},\widehat{y}\right\rangle \widehat{x} \Vert_2} \in T_{\widehat{x}}\iota(S^2)
\end{equation}
with $\left\langle \widehat{x},\widehat{y} \right\rangle \defeq \widehat{x}^T\widehat{y}$ on $\reals^3$. Its magnitude is the geodesic distance $d(\widehat{x},\widehat{y}) = \cos^{-1}\langle \widehat{x},\widehat{y}\rangle$ (the arc-length angle, since $\Vert\widehat{x}\Vert_2=\Vert\widehat{y}\Vert_2=1$) and its direction is the normalized orthogonal projection at $\widehat{x}$, $\pi_{\widehat{x}}(\widehat{y}) \defeq (I_3 - \widehat{x}\widehat{x}^T)\widehat{y}$. This is the pointwise tangent projection of section~\ref{sec:AMG_field} and, at the center, it is the $\pi_0$ of Thm.~\ref{thm:normal_eigvals} conceived as an ambient representation of a tangent vector \cite{Absil2008,edelman1998geometry}. As throughout, these intrinsic objects are computed from extrinsic data.

In a \textit{data-driven} setting we are given samples $\lbrace \widehat{x}_i \rbrace_{i=1}^N \sim \mu$ of $\widehat{\mathcal{M}}$---random points on the $2$-sphere. The inverse exponential yields their Karcher mean \cite{fletcher2004principal,absil2004riemannian,srivastava2002monte}, which we take as the center $\widehat{x}_0$, so $\varphi^{-1}(\widehat{x}_0) = \boldsymbol{0}$ in normal coordinates---the intrinsic analog of centering Euclidean samples at the origin \cite{Jolliffe2002,Chatterjee2000} as in \eqref{eq:C_par_trans}. The resulting active manifold-geodesics \eqref{eq:AMG} are then defined relative to the tangent space at this Karcher mean, as required by Def. \ref{def:riemannian_view} and Algorithm \ref{alg:riemannian_MC_AS}.

Additionally, \cite{fletcher2004principal} discuss the projection to submanifolds so that we may construct analogous shadow plots. In our case, given $\mathcal{S}_{1,j} \defeq \lbrace \gamma(t) \in \mathcal{M} \, :\, \gamma(t) = \text{exp}_{p_0}(t\vv_j),\,\, t \in \mathcal{I}_j \subseteq \mathbb{R}\rbrace$, we define projection to the submanifold with fixed origin $p_0 = \text{exp}_{p_0}(\boldsymbol{0})$ as
\begin{equation} \label{eq:proj_nml_submanifold}
\pi_{\mathcal{S}_{1,j}}(p)\defeq \underset{t\in \mathcal{I}_j}{\text{argmin}}\,\,d^2(p,\text{exp}_{p_0}(t\vv_j))
\end{equation}
for all $p \in \mathcal{X}$ and $\mathcal{I}_j \subseteq \mathbb{R}$ containing zero. Additionally, note that the embedded coordinate defining the central tangent space as a $2$-dimensional affine space tangent to the sphere in $\mathbb{R}^3$ is defined at the point $\widehat{x}_0 \defeq \iota(\varphi(\boldsymbol{0})) \in \iota(S^2)$ representing our center in ambient coordinates.

\subsection{Ambient ridge function}
Suppose we have an ambient linear function $f(\widehat{x}) \defeq \va^T\widehat{x}$ for all $\widehat{x} \in \mathcal{X} \subset \iota(S^2) \subset \mathbb{R}^3$ and some $\va \in \mathbb{R}^3$ such that $\Vert \va\Vert_2 = 1$. Evidently the ambient gradient is a constant $\overline{\nabla} f = \va$. The ambient function is a ridge function \eqref{eq:ridge} along the direction $\va$ although this property is lost when restricting to the nonlinear surface of the sphere. Consequently, $G_0$ is rank-$2$ (this is consistent throughout each example) but still admits an \textit{ordering} of important directions per $\lambda_1 \geq \lambda_2 > \lambda_3 = 0$. Specifically, observing Figure \ref{fig:lin_AMG}, it is evident that this ambient linear function---when restricted to the sphere---changes more along a particular $1$-dimensional submanifold. 
\begin{figure*}[t] 
	\centering
	\subfloat[Data-driven extrinsic observations of a ridge function and corresponding tangent gradients restricted to the sphere.]{
		\includegraphics[width=0.45\textwidth]{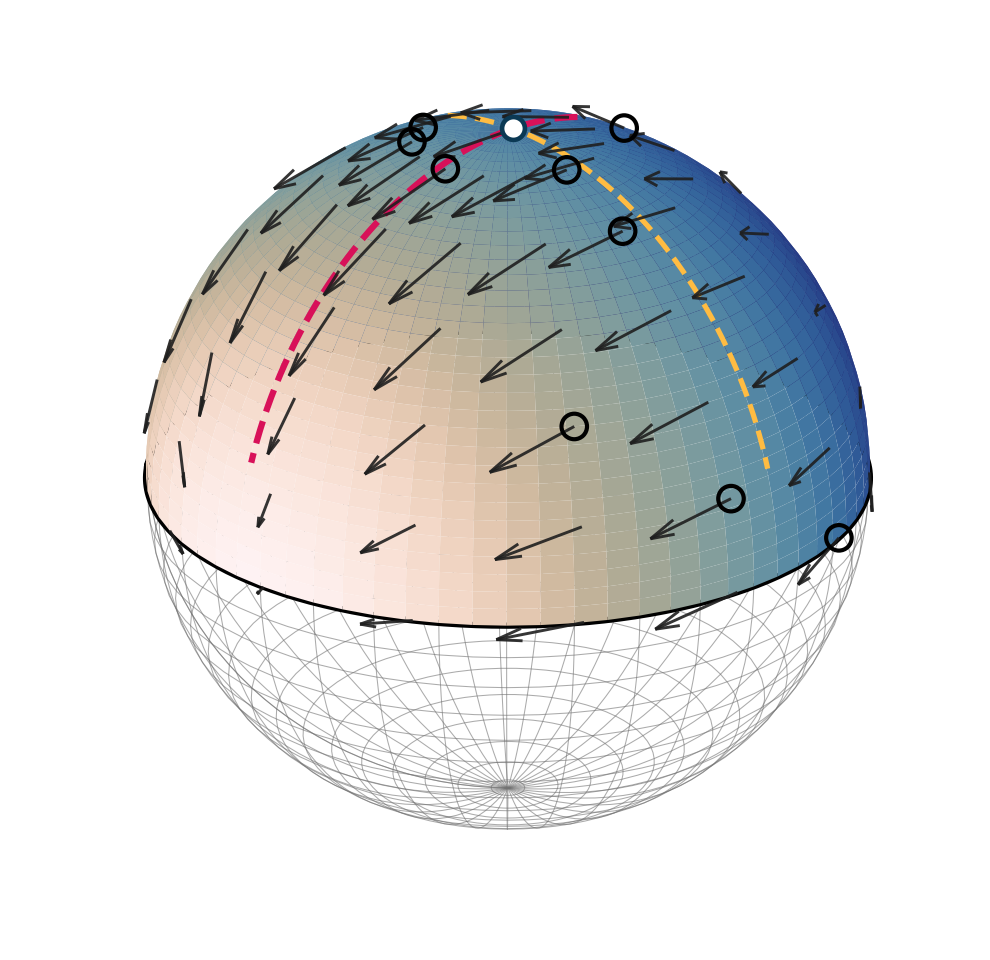}
	}
	\hfil
	\subfloat[Analogous shadow plot $\lbrace (\pi_{\tilde{\mathcal{W}}_1}(\widehat{x}_i), f(\widehat{x}_i))\rbrace$ over normal coordinates (colored scatter) and function responses over one dimensional submanifolds.]{
		\includegraphics[width=0.5\textwidth]{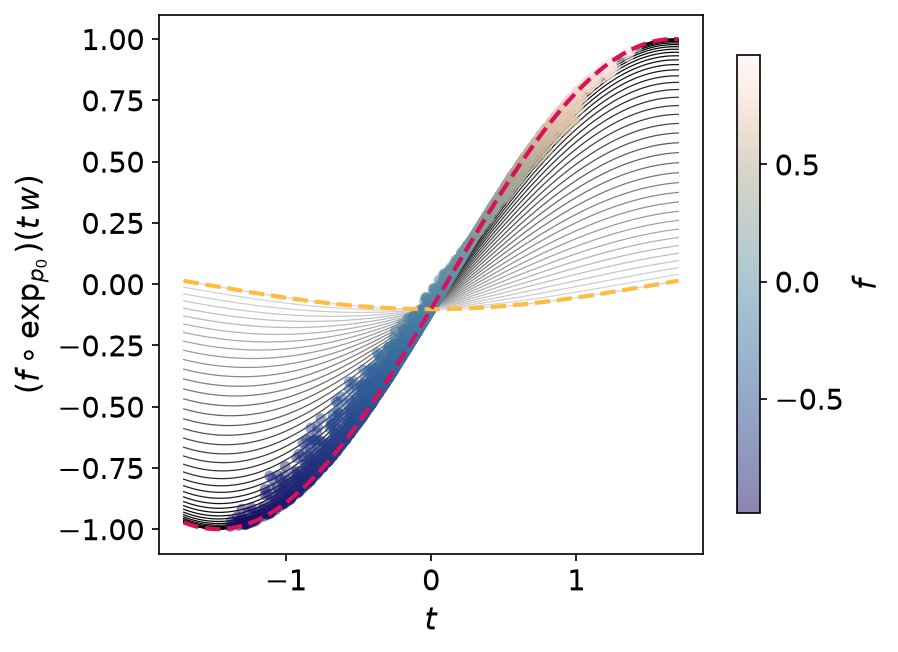}
	}
	
	\caption{AMG of an ambient linear function restricted to $\iota(S^2) \subset \mathbb{R}^3$. \textbf{Left (a)}: Ambient linear function over the sphere with a subset of random points and corresponding gradients. The AMG $\mathcal{W}_1$ (rose great circle) is shown; the projection to $T_{\widehat{x}_0}$ of the dominant eigenvector of \eqref{eq:emb_opg} (the extrinsic direction) coincides with the AMG here ($d\approx4\times10^{-4}$) and is therefore not drawn as a separate trace. \textbf{Right (b)}: Function evaluations restricted to a variety of submanifolds. The response over the AMG (rose) is shown against responses over geodesics defined by rotations of the dominant eigenvector in the tangent space (thin gray lines). Additionally, the function response over the AMG is visually consistent with the response resulting from the extrinsic perspective (coincident with the AMG). An analogous shadow plot utilizing \eqref{eq:proj_nml_submanifold} is shown over $\tilde{\mathcal{W}}_{1,1}$ with corresponding function evaluations (colored-dots).}
	\label{fig:lin_AMG}
\end{figure*}

\begin{figure*}[t]
	\centering
	\includegraphics[width=0.85\textwidth,trim={0 30pt 0 0},clip]{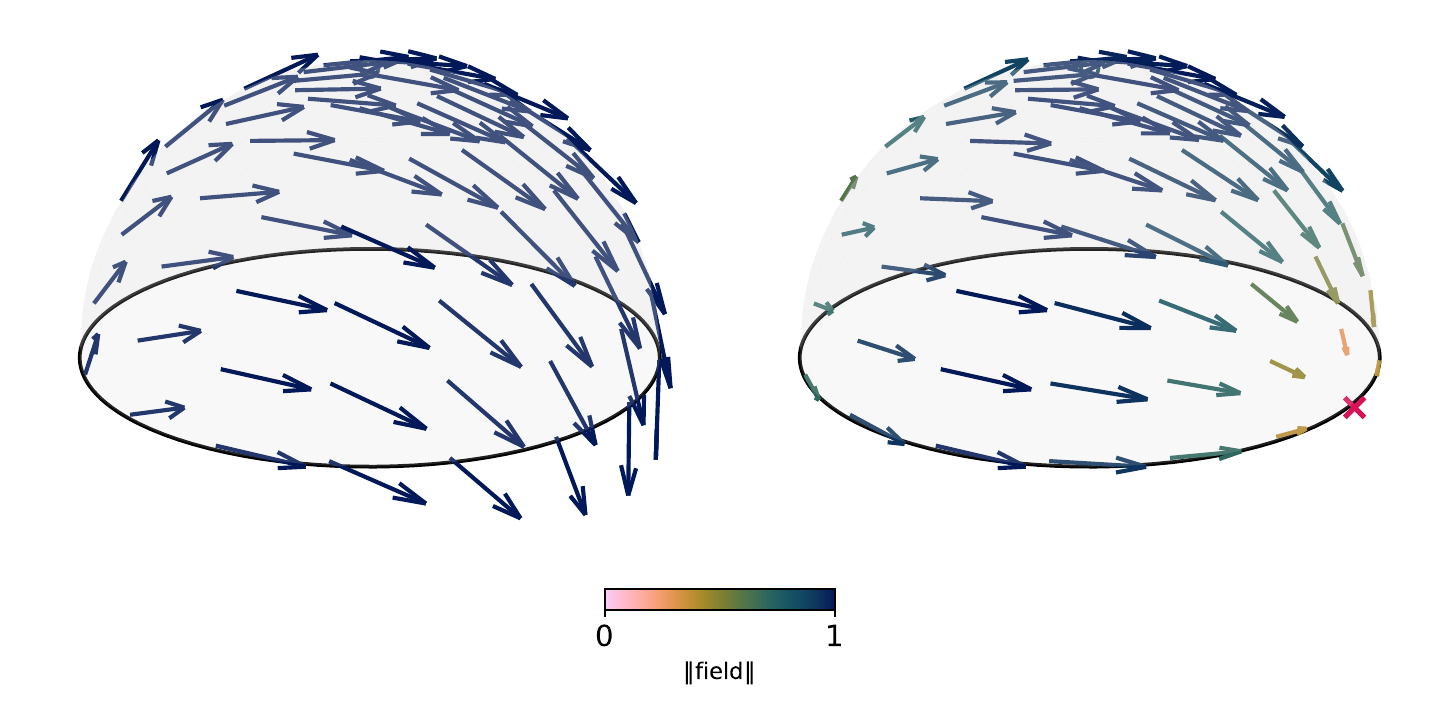}
	\caption{Transported versus projected extension of a dominant direction over a hemisphere of $\iota(S^2)$, colored by field magnitude. \textbf{Left}: the intrinsic frame field $\vw_1(p) = \mathcal{P}^{-1}_{p}[\vw_1]$ of \eqref{eq:AMG_field}---one central eigendecomposition propagated by the isometric transport---is unit-norm everywhere. \textbf{Right}: the \textit{projected-dominance} extension $\widehat{x} \mapsto \pi_{\widehat{x}}\,\widehat{\vb}$ of the fixed ambient direction $\widehat{\vb}$ degenerates: its magnitude is the cosine of the angle between $\widehat{\vb}$ and the tangent space, vanishing where $\widehat{\vb}$ is normal to the manifold (rose $\times$)---the pointwise collapse of the partial isometry (Thm.~\ref{thm:normal_eigvals}). The extrinsic ordering it carries need not survive and activity should never be extended this way (Remark~\ref{rem:noncommute}). The natural extension of any fixed frame is always that of parallel transport.}
	\label{fig:frame_field_compare}
\end{figure*}

The precise sense in which the ambient ridge survives---and is \textit{recovered}---in normal coordinates is Prop.~\ref{prop:local_agreement}. Figure~\ref{fig:amg_nml_coords} shows the ridge in normal coordinates over a geodesic ball, and Figures~\ref{fig:amg_ridge_recovery} and~\ref{fig:amg_ridge_recovery_nl} confirm the $\mathcal{O}(R^2)$ recovery numerically for linear and nonlinear ridges and a reduced rate against a non-ridge.
\begin{figure}[t]
	\centering
	\includegraphics[width=\columnwidth]{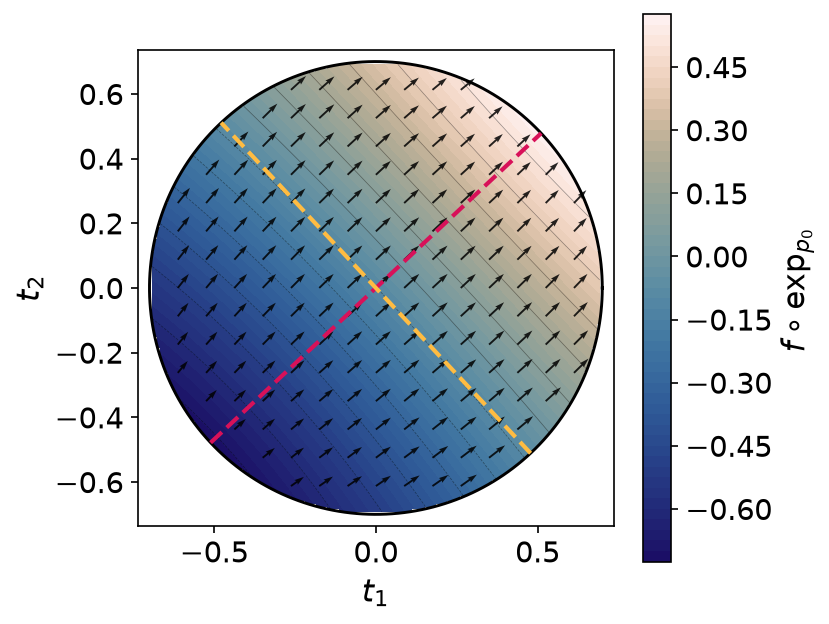}
	\caption{The ambient ridge $f=\va^{\top}\widehat{x}$ in normal coordinates over a geodesic ball at the Karcher mean $\widehat{x}_0$: color field $f\circ\text{exp}_{\widehat{x}_0}$, black arrows the tangential gradient parallel-transported to $\widehat{x}_0$ (the $G_0$ integrand). The intrinsic active axis $w_1$ (solid rose) aligns with steepest ascent and the inactive axis $w_2$ (dashed marigold) with the level set. The near-straight contours illustrate the $\mathcal{O}(R^2)$ ridge recovery of the second-order locality (Prop.~\ref{prop:local_agreement}).
    }
	\label{fig:amg_nml_coords}
\end{figure}
\begin{figure*}[t]
	\centering
	\includegraphics[width=0.85\textwidth]{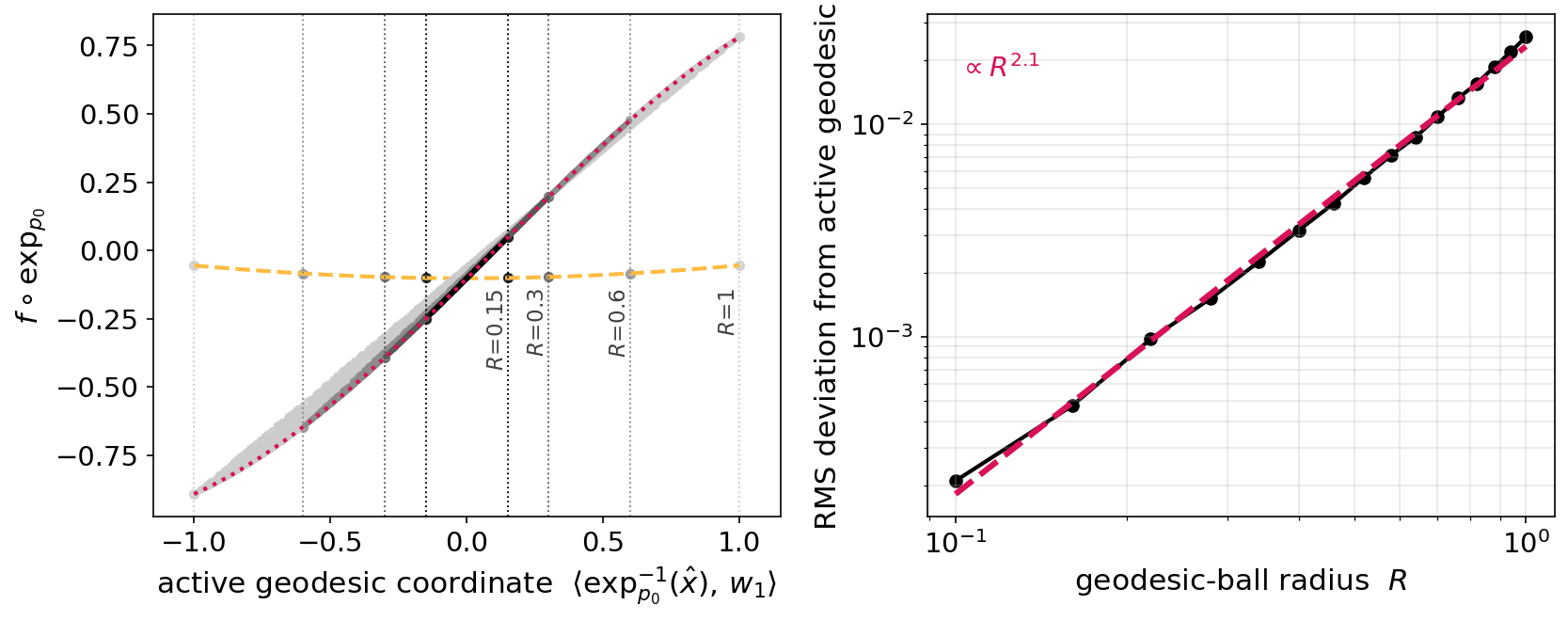}
    \includegraphics[width=0.85\textwidth]{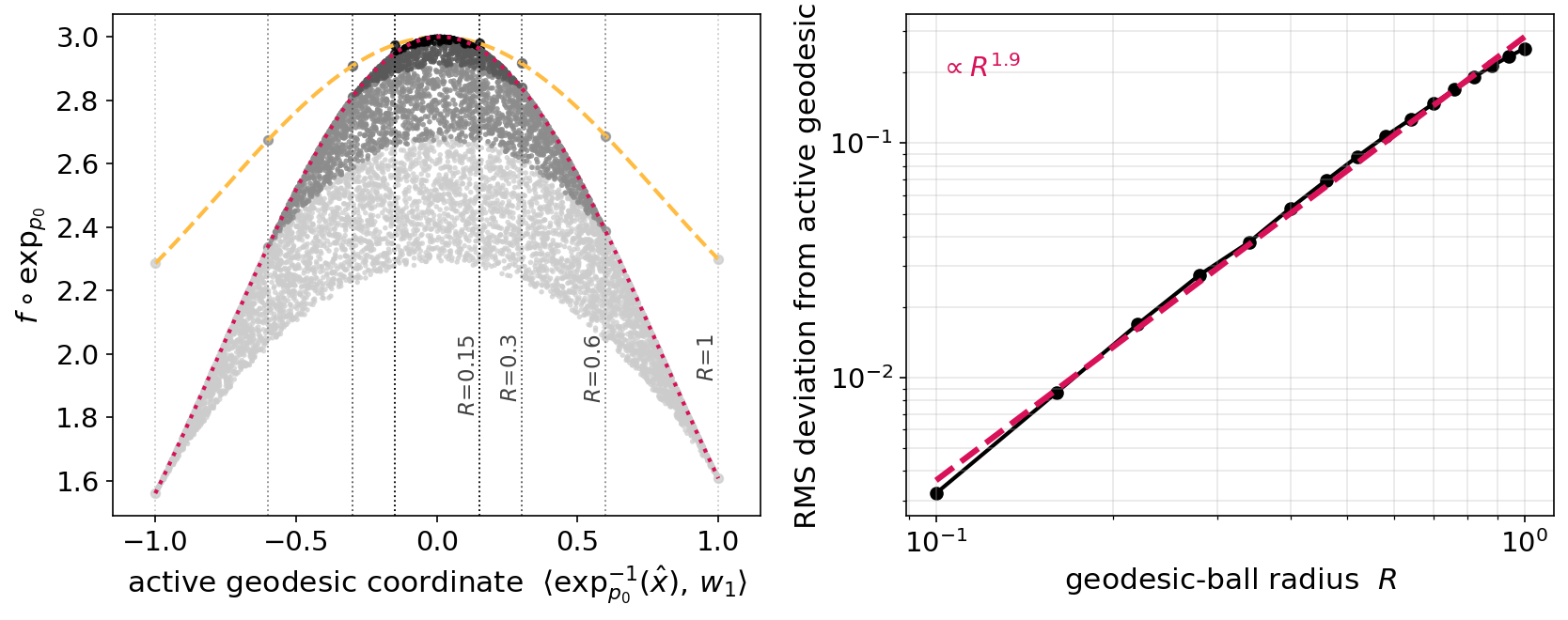}
	\caption{Numerical demonstration of $\mathcal{O}(R^2)$ ridge recovery over a \textit{shrinking} geodesic ball---the second-order locality of Prop.~\ref{prop:local_agreement}. \textbf{Top}: An ambient linear function. \textbf{Bottom}: An ambient quadratic with preferential directions (non-ridge). \textbf{Left}: $f\circ\text{exp}_{\widehat{x}_0}$ against the active geodesic coordinate $s_1=\langle\text{exp}^{-1}_{\widehat{x}_0}(\widehat{x}),\,w_1\rangle$, with the sample scatter in nested shells for decreasing ball radius $R$ (lightening with $R$; dotted lines mark each ball's extent $\pm R$). As $R\to 0$ the extent contracts and the scatter collapses onto the $1$-dimensional active-geodesic sweep (rose). \textbf{Right}: the RMS deviation of the samples from that active-geodesic sweep as a function of $R$, decaying as $\mathcal{O}(R^2)$---the curvature-limited recovery rate (dashed guide).}
	\label{fig:amg_ridge_recovery}
\end{figure*}
\begin{figure*}[t]
	\centering
    \includegraphics[width=0.825\textwidth]{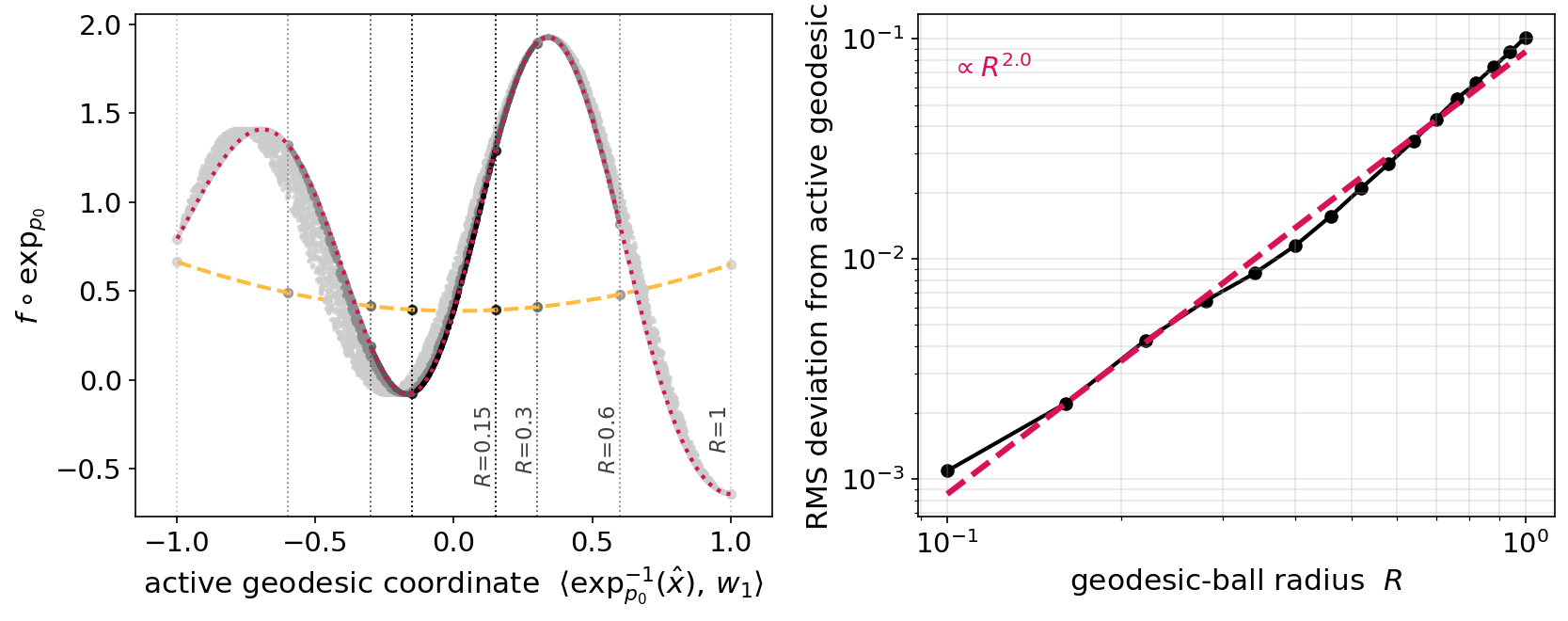}
    \includegraphics[width=0.825\textwidth]{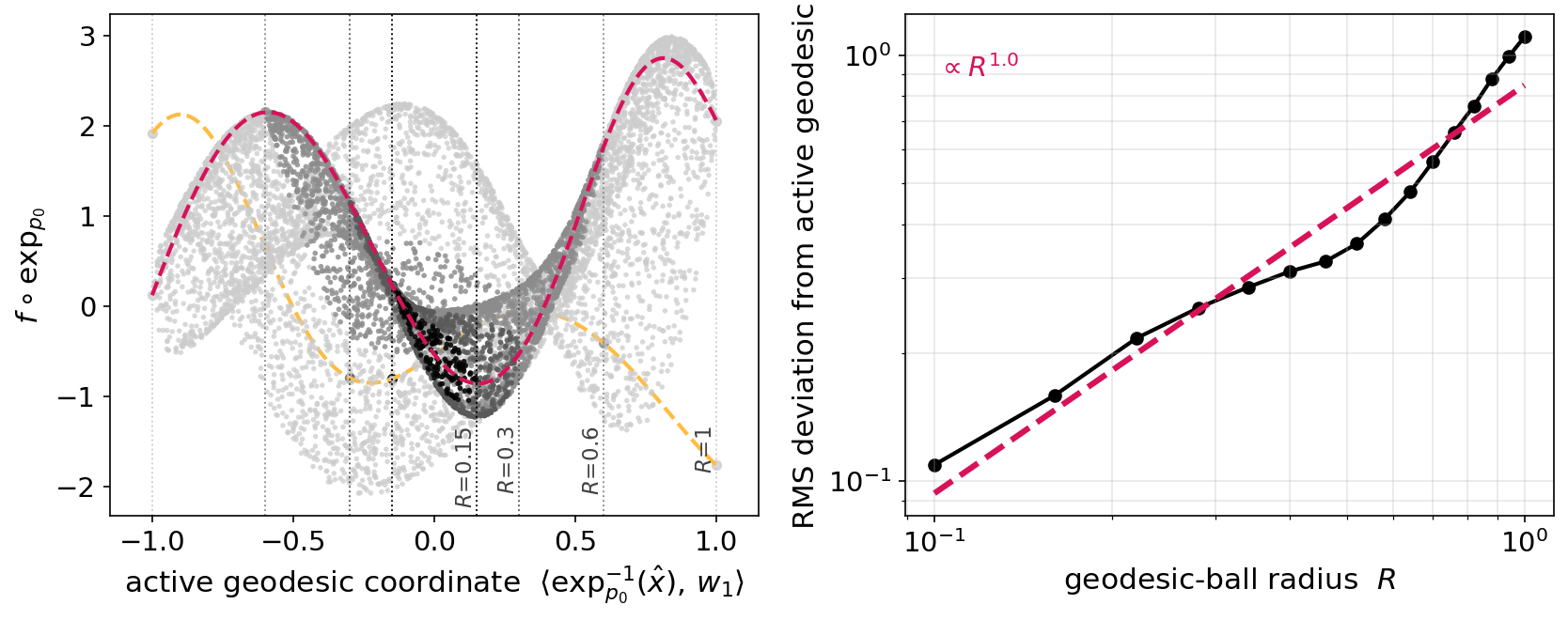}
	\caption{Numerical demonstration of $\mathcal{O}(R^2)$ ridge recovery over a \textit{shrinking} geodesic ball---the second-order locality of Prop.~\ref{prop:local_agreement}. \textbf{Top}: A nonlinear ridge. \textbf{Bottom}: A nonlinear non-ridge. \textbf{Left}: $f\circ\text{exp}_{\widehat{x}_0}$ against the active geodesic coordinate $s_1=\langle\text{exp}^{-1}_{\widehat{x}_0}(\widehat{x}),\,w_1\rangle$, with the sample scatter in nested shells for decreasing ball radius $R$ (lightening with $R$; dotted lines mark each ball's extent $\pm R$). As $R\to 0$ the extent contracts and the scatter collapses onto the $1$-dimensional active-geodesic sweep (rose). \textbf{Right}: the RMS deviation of the samples from that active-geodesic sweep as a function of $R$, decaying as $\mathcal{O}(R^2)$---the curvature-limited recovery rate (dashed guide). A genuine $2$-dimensional (non-ridge) function instead retains a vertical spread, and its RMS does not match the expected rate absent the assumed spectral gap.}
	\label{fig:amg_ridge_recovery_nl}
\end{figure*}
Consistent with the data-driven setting, assume we are given triples $\lbrace (\widehat{x}_i,\nabla_{\mathcal{M}}\widehat{f}(\widehat{x}_i), f(\widehat{x}_i))\rbrace_{i=1}^N$ as random samples from unknown $\mu$ supported over the upper hemisphere $\mathcal{X}\subset \iota(S^2)$. We proceed by centering normal coordinates on the Karcher mean of $\lbrace \widehat{x}_i \rbrace$ and approximate eigenspaces of $G_0$ using $\lbrace \nabla_{\mathcal{M}}\widehat{f}(\widehat{x}_i) \rbrace$ in Algorithm \ref{alg:riemannian_MC_AS} (omitting step one since samples are provided). In this context, parallel transport is computed over the upper hemisphere using the form presented in \cite{edelman1998geometry} for the $2$-dimensional Grassmannian, $\text{G}_{3,1}$\footnote{The upper hemisphere of $S^2$ is locally isometric to $\text{Gr}(1,3)=\mathbb{RP}^2$, its double cover; the Grassmann formulae apply because the hemisphere contains no antipodal pairs.}. The approximated dominant eigenvector $\tilde{\vw}_1$ is computed as the first column of $\tilde{W}_0 \in O(3)$ ordered according to decaying eigenvalues.\footnote{The third eigenvalue is order machine precision given an intrinsic dimension of $n=2$---i.e., each parallel transported extrinsic vector is an element of a $2$-dimensional affine space admitting, at most, two important directions.} The direction $\tilde{\vw}_1$ is used to define a $1$-dimensional submanifold as the active manifold-geodesic $\tilde{\mathcal{W}}_{1,1}$ per \eqref{eq:AMG} such that $\mathcal{V}$ is diffeomorphic to the upper hemisphere. An example of the computations involving $N=1000$ random samples of the upper hemisphere and Karcher mean $\widehat{x}_0 \defeq (0,0,1)^T$ are depicted in Figure \ref{fig:lin_AMG}---for visual acuity, extrinsic visualizations correspond to a subset of the total random samples.

Examining Figure \ref{fig:lin_AMG}(a), we visualize the active manifold-geodesic (rose) and notice $\tilde{\mathcal{W}}_{1,1}$, according to $\tilde{\vw}_1$ such that $\tilde{\lambda}_1 > \tilde{\lambda}_2$, maximizes the function's range restricted to the $1$-dimensional submanifold. In contrast, the transverse submanifold $\tilde{\mathcal{W}}_{1,2}$ (marigold) defined using $\tilde{\vw}_2$ with smaller paired eigenvalue minimizes the function's range restricted to a $1$-dimensional submanifold emanating from the same central tangent space. This is further emphasized by examining Figure \ref{fig:lin_AMG}(b) which depicts function evaluations in normal coordinates restricted to the corresponding color and line-style submanifolds in Figure \ref{fig:lin_AMG}(a). 

The intermediate-range function responses (thin gray lines) in Figure \ref{fig:lin_AMG}(b) correspond to evaluations restricted to alternative submanifolds obtained by rotating $\tilde{\vw}_1$ in the central tangent space. The function response over each alternative submanifold is bounded above by the range over $\tilde{\mathcal{W}}_{1,1}$ and below by that over $\tilde{\mathcal{W}}_{1,2}$. The responses over alternative submanifolds in Figure \ref{fig:lin_AMG}(b) are also shown with an analogous \textit{shadow plot} of pairs $\lbrace (\pi_{\tilde{\mathcal{W}}_{1,1}}(\widehat{x}_i), f(\widehat{x}_i))\rbrace$ (colored scatter plot).

Lastly, in this case, the orthogonal projection $\pi_0$ of the approximated dominant eigenspace of \eqref{eq:emb_opg} using Algorithm 3.1 of \cite{Constantine2015} also admits a submanifold which is coincident with the active manifold-geodesic (subspace distance $\approx4\times10^{-4}$) and therefore not drawn separately---the generic regime (i)--(iii) of Remark~\ref{rem:equivalence}. The intrinsic and extrinsic perspectives differ fundamentally by definition of their generating eigenspaces \eqref{eq:riemannian_view} and \eqref{eq:emb_opg}, respectively. However, in this example involving a known ambient ridge function and a specific central tangent space, we observe the consistency quantified by Prop.~\ref{prop:local_agreement}.
\begin{figure}[t] 
	\centering
	\includegraphics[width=0.45\textwidth]{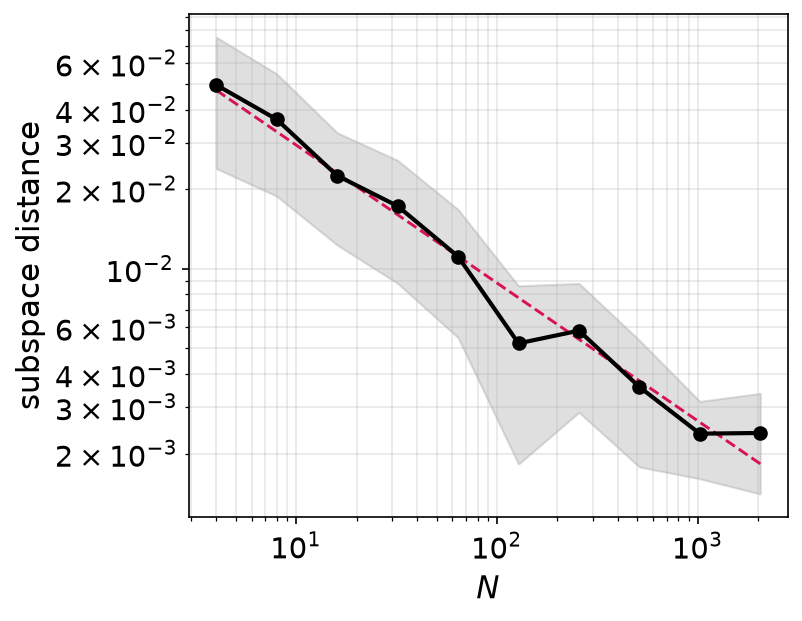}
	\caption{Convergence of approximated dominant $G_0$ eigenspace using Algorithm \ref{alg:riemannian_MC_AS} to the orthogonal projection of the known ambient ridge direction, $\pi_0(\va)$. Depicted is the average error-metric over $500$ re-sampled tangent gradients for each of the $N$ total samples used in Algorithm \ref{alg:riemannian_MC_AS} (black circle-line). The gray filled intervals represent $3$ and $6$ times the standard error plus and minus the average error-metric at the corresponding $N$ samples.}
	\label{fig:AMG_sphere_conv}
\end{figure}

To study this apparent consistency numerically, consider an error metric as the \textit{subspace distance}  \cite{Golub1996} between the dominant eigenspace obtained from Algorithm \ref{alg:riemannian_MC_AS} and the orthogonal projection of the known ambient ridge direction, $W_1 \defeq \pi_0(\va)$, as
$
\text{err}_{N} \defeq\Vert \tilde{W}_1\tilde{W}^T_1 - W_1W_1^T \Vert_2.
$
The matrices $\tilde{W}_1$ and $W_1$ represent the separate bases in $T_{\widehat{x}_0}\iota(S^2)$ defining the one-dimensional submanifolds for the true submanifold $\mathcal{W}_1$ and approximate submanifold $\tilde{\mathcal{W}}_1$. Figure \ref{fig:AMG_sphere_conv} depicts the convergence rate for increasing $N$ random gradients over the upper hemisphere. This process is then repeated $500$ times with a new random seed to offer evidence of stability in the computation. The approximation of the dominant eigenspace of \eqref{eq:emb_opg} using Monte Carlo converges to $\va$ as $\mathcal{O}(N^{-1/2})$ and the dominant eigenspace of \eqref{eq:riemannian_view} approximated by Algorithm \ref{alg:riemannian_MC_AS} converges to $\pi_0(\va)$ at the same rate; this subspace rate presumes a nonzero eigenvalue gap (Davis--Kahan \cite{Golub1996}), which closes precisely in the degenerate case $\widehat{x}_0=\pm\va$ examined next. 

Evidently, we observe consistency in the numerical experiment given the two fundamentally different definitions of important eigenspaces of \eqref{eq:riemannian_view} and \eqref{eq:emb_opg}. However, let's study what happens when we choose $\widehat{x}_0 = \pm \va$---more generally $\widehat{x}_0 \in \text{Range}(\overline{\nabla}f) \cap \iota(S^2)$. Analogously, rotate the random subset $\lbrace \widehat{x_i} \rbrace_{i=1} ^N$ so that the central tangent space has representative Karcher mean at one of the two $\widehat{x}_0 = \pm \va$ and re-evaluate the function depicted in Figure \ref{fig:lin_AMG} at the new points. The result is depicted in Figure \ref{fig:lin_AMG_noproj}.

\begin{figure*}[t] 
	\centering
	\subfloat[Data-driven extrinsic observations of a ridge function and corresponding tangent gradients restricted to the sphere.]{
		\includegraphics[width=0.45\textwidth]{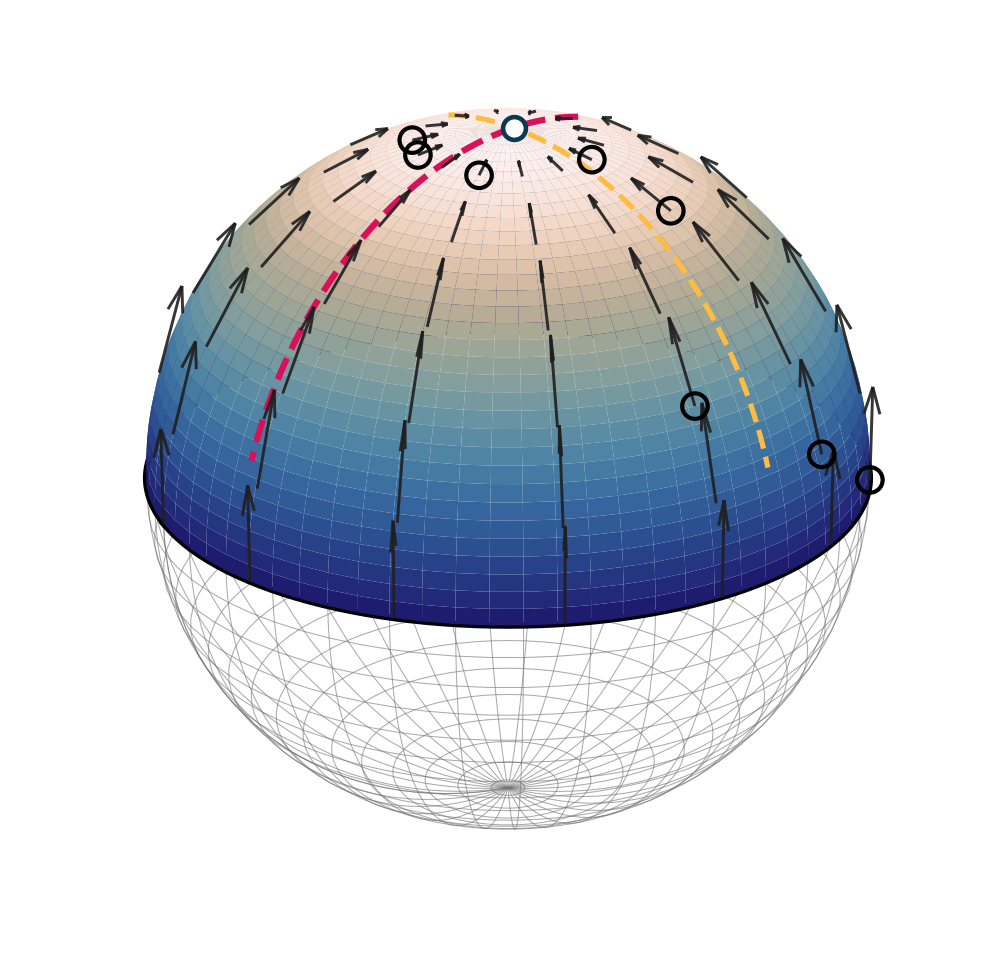}
	}
	\hfil
	\subfloat[Analogous shadow plot $\lbrace (\pi_{\tilde{\mathcal{W}}_1}(\widehat{x}_i), f(\widehat{x}_i))\rbrace$ over normal coordinates (colored scatter) and function responses over one dimensional submanifolds.]{
		\includegraphics[width=0.5\textwidth]{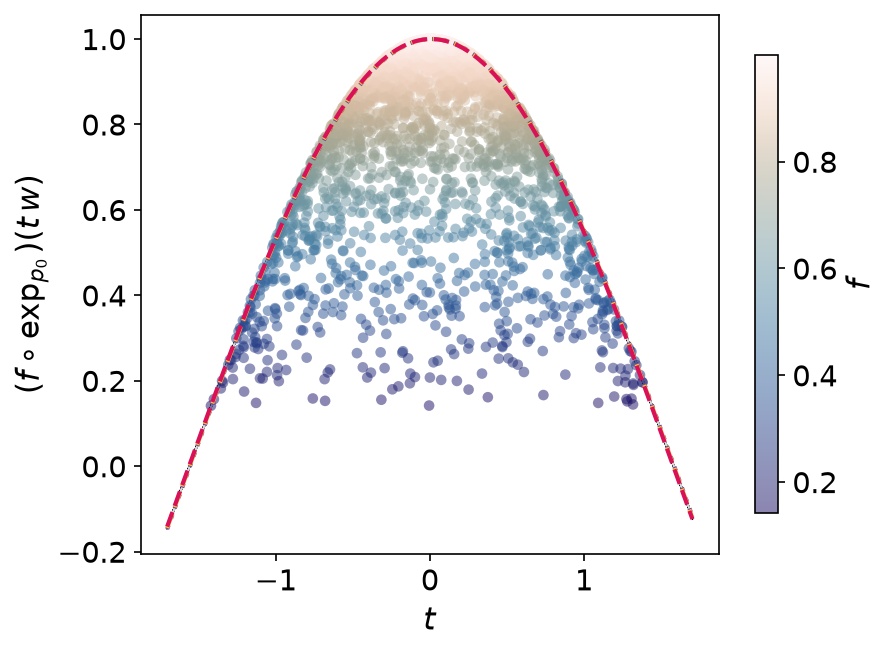}
	}
	
	\caption{Active manifold-geodesics of an ambient linear function restricted to $\iota(S^2) \subset \mathbb{R}^3$. \textbf{Left (a)}: Ambient linear function over the sphere with a subset of random points and corresponding gradients. The AMG (rose great circle) is shown with the transverse (inactive) direction in the central tangent space (marigold great circle). Here the \textit{projected-ambient-eigenvector} estimate from \eqref{eq:emb_opg} vanishes: at the aligned center $\widehat{x}_0=\pm\va$ the dominant ambient eigenvector of $C_\iota$ is normal to $T_{\widehat{x}_0}\iota(S^2)$, so its projection is annihilated---regime (iii) of Remark~\ref{rem:equivalence}. The central representation $E_0^{\top}C_{\iota}E_0$ itself remains isotropic, consistent with the two equal eigenvalues of $G_0$ (regime (i) holds; the spectral gap of regime (ii) is zero for both). \textbf{Right (b)}: Function evaluations restricted to two submanifolds. The response over the AMG (rose) is shown against the response over the orthogonal direction in the central tangent space (marigold). An analogous shadow plot according to \eqref{eq:proj_nml_submanifold} is shown over $\tilde{\mathcal{W}}_{1,1}$ with corresponding function evaluations (colored-dots).}
	\label{fig:lin_AMG_noproj}
\end{figure*}

In Figure \ref{fig:lin_AMG_noproj}, the projected dominant eigenspace of \eqref{eq:emb_opg} converges to the trivial tangent vector, $\lim_{N\rightarrow \infty}\pi_0(\tilde{\va}_N) = \boldsymbol{0} \in T_{\widehat{x}_0}\iota(S^2)$, for the Monte Carlo estimate $\tilde{\va}_N$. If the samples of the gradient are not appropriately projected to local tangent spaces $T_{\widehat{x}}\iota(S^2)$, we obtain a single (dominant) nonzero eigenvalue per Prop. \ref{prop:eigs_and_constant_f} and a meaningless interpretation of ridge structure which is annihilated by projected-dominance. In contrast, the approximated eigendecomposition of \eqref{eq:riemannian_view} using Algorithm \ref{alg:riemannian_MC_AS} only admits two non-zero approximately equal magnitude eigenvalues indicating no preferential direction over the $2$-dimensional tangent space of the manifold. The intrinsic result suggests the function does not change more along either submanifold given by the approximated eigenspaces of \eqref{eq:riemannian_view} \textit{without subsequent interpretation involving projections}. In fact, the function response is equal over each one dimensional submanifold defined by arbitrary $\widehat{\vv} \in T_{\widehat{x}_0}\iota(S^2)$. Figure \ref{fig:lin_AMG_noproj}(b) shows the overlapping function responses along $\tilde{\mathcal{W}}_{1,1}$ and transverse $\tilde{\mathcal{W}}_{1,2}$ defined using orthonormal $\tilde{\vw}_1, \tilde{\vw}_2 \in T_{\widehat{x}_0}\iota(S^2)$ resulting from Algorithm \ref{alg:riemannian_MC_AS}. 

\begin{remark}
	The numerical examples depicted in Figure \ref{fig:lin_AMG} and Figure \ref{fig:lin_AMG_noproj}, utilizing the same ambient ridge structure, emphasize that this generalized perspective of low-dimensional structure depends on the choice of central tangent space---made explicit by Thm. \ref{thm:normal_eigvals} and Def. \ref{def:riemannian_view}.
\end{remark}
Figure \ref{fig:lin_AMG_noproj} emphasizes the failure mode of projected dominance. An approximation of \eqref{eq:emb_opg} projected to the tangent space may only result in small cancellations as opposed to a vanishing direction. This raises the question addressed in Remark~\ref{rem:noncommute}: \textit{if the projection of dominant directions from an approximation of \eqref{eq:emb_opg} to the central tangent space are non-zero but small, (in very high dimensions) how should we re-weight our interpretation of importance inferred by an ordering of eigenvalues?}

Theorem~\ref{thm:normal_eigvals} supplies the principled re-weighting: eigendecompose the central representation $E_0^{\top}C_{\iota}E_0$---project \textit{then} order---rather than projecting the leading ambient eigenvector. Here that representation is isotropic, matching the two equal intrinsic eigenvalues, and in general it remains $\mathcal{O}(R^2)$-consistent with $G_0$ (Prop.~\ref{prop:local_agreement}); it is the ambient eigen-ordering, which weighs components the projection annihilates, that cannot be repaired by any re-weighting after the fact. Alternatively, Algorithm \ref{alg:riemannian_MC_AS} only produced two intrinsic directions paired with two approximately equal eigenvalues indicating the anticipated interpretation of equal importance of either direction. 

Precisely, given fixed $E_0$, the central projection \textit{commutes with the average}---the first identity of Thm.~\ref{thm:normal_eigvals}---so projecting each tangential gradient and then averaging yields the same matrix as averaging and then compressing to the central tangent space. What does \textit{not} commute with projection is the eigendecomposition, and with it the eigen-\textit{ordering} (Remark~\ref{rem:noncommute}): \cite{wu2010learning} order in the ambient space and then projecting the dominant eigenvector, whereas Thm.~\ref{thm:normal_eigvals} orders the centrally-projected representation. The improvement is ordering \textit{after} projection, (dominant projections) which is immune to ambient dominance carried by components the projection annihilates.

\subsection{Ambient non-ridge function}
Building on the distinction between the two eigenspaces, consider the case that some ambient function is not a ridge function but retains an ordering of directional derivatives. For example, an ambient quadratic $f(\widehat{x}) \defeq \widehat{x}^TA\widehat{x}$ for all $\widehat{x} \in \iota(S^2) \subset \mathbb{R}^3$ such that $A \defeq \text{diag}(3,2,1) \in \mathbb{R}^{3 \times 3}$. This quadratic admits an ordering of directional derivatives but it is not an ambient ridge function \eqref{eq:ridge}. Figure \ref{fig:quad_AMG} depicts the resulting computations using the previously defined conditions, $N = 1000$ and $\widehat{x}_0 \defeq (0,0,1)^T$.
\begin{figure*}[t] 
	\centering
	\subfloat[Data-driven extrinsic observations of a ridge function and corresponding tangent gradients restricted to the sphere.]{
		\includegraphics[width=0.45\textwidth]{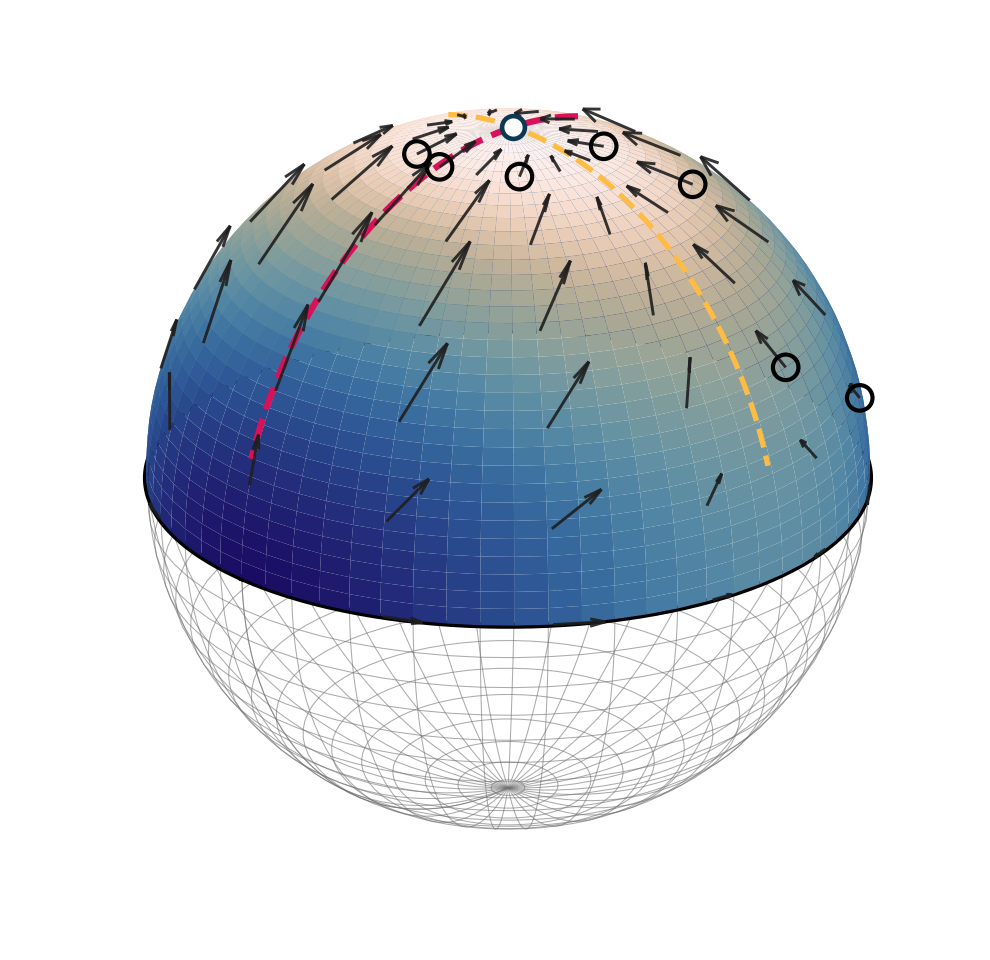}
	}
	\hfil
	\subfloat[Analogous shadow plot $\lbrace (\pi_{\tilde{\mathcal{W}}_1}(\widehat{x}_i), f(\widehat{x}_i))\rbrace$ over normal coordinates (colored scatter) and function responses over one dimensional submanifolds.]{
		\includegraphics[width=0.5\textwidth]{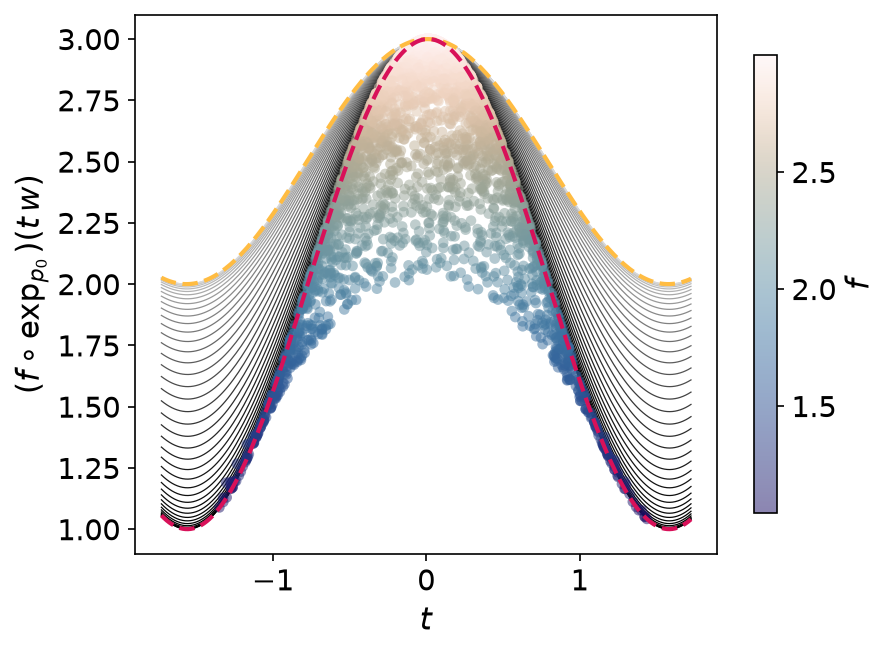}
	}
	
	\caption{Active manifold-geodesics of an ambient quadratic function with ordered directional derivatives on $\iota(S^2) \subset \mathbb{R}^3$. \textbf{Left (a)}: Ambient quadratic function over the sphere with subset of random points and corresponding gradients. The active manifold-geodesic (rose line) is shown; under the extrinsic construction the extrinsic direction from \eqref{eq:emb_opg} coincides with the AMG here ($d\approx5\times10^{-3}$). \textbf{Right (b)}: Function evaluations restricted to a variety of submanifolds. The response over the active manifold-geodesic (rose) is shown against responses over a subset of other possible rotations in the $2$-dimensional tangent space. The extrinsic response coincides with the AMG. An analogous shadow plot according to \eqref{eq:proj_nml_submanifold} is shown over $\tilde{\mathcal{W}}_{1,1}$ with corresponding function evaluations (colored-dots).}
	\label{fig:quad_AMG}
\end{figure*}

For this quadratic the dominant dimension-reduction direction is genuinely tangential at $\widehat{x}_0$, so the extrinsic estimate---the projection to $T_{\widehat{x}_0}\iota(S^2)$ of the dominant eigenvector of \eqref{eq:emb_opg}---survives that projection ($\Vert\pi_0\widehat{\vb}\Vert\approx0.99$) and \textit{coincides} with the intrinsic AMG of Algorithm~\ref{alg:riemannian_MC_AS}, exactly as in the ridge case of Figure~\ref{fig:lin_AMG}. Both eigenspaces capture the direction of maximum range of the function restricted to a one-dimensional submanifold, and the ordering of directional derivatives is recovered identically by the two perspectives. This is the generic situation of Remark~\ref{rem:equivalence}: all three regimes hold---the central representations agree to $\mathcal{O}(R^2)$, the spectral gap is well separated from that discrepancy, and the dominant ambient eigenvector survives the projection.

The coincidence observed with the preferential quadratic directions is also robust. Scanning many central points shows that this quadratic's ambient ordering is overturned only where its two leading ambient eigenvalues are nearly degenerate---exactly where a \textit{dominant} direction ceases to be well-defined---so the ordered ambient quadratic cannot honestly exhibit the reordering of Remark~\ref{rem:noncommute}. For the nonlinear non-ridge of Appendix~\ref{app:gallery} at sufficiently large support, there are central points where the $\cos^2\theta_i$ reweighting overturns a robust ambient dominance restricted to the sphere, both the ambient and central spectra remaining well separated. The two candidate directions of Remark~\ref{rem:noncommute}, projected dominance and dominant projection, then depart by a visible angle.

The nonlinear non-ridge example deserves emphasis, because it is where the choice of perspective genuinely matters. No ridge hypothesis enters Lemma~\ref{lem:exact_eigvals}: for any sufficiently smooth $f$, the intrinsic eigenvalues are exact mean-squared directional derivatives along the unit-norm transported fields \eqref{eq:AMG_field}. The eigen-ordering therefore remains an honest, variance-ordered ranking of directions even when there is no low-dimensional structure to certify. Without vanishing trailing eigenvalues, only the certification of ridge structure \eqref{eq:err_apprx} is lost but ordering is retained. The extrinsic representation orders exactly what Thm.~\ref{thm:normal_eigvals} says it orders, but that is a fundamentally different quantity. For a central direction $\widehat{\vw} = E_0\vw$, the quadratic form $\vw^{\top}(E_0^{\top}C_{\iota}E_0)\vw = \int_{\mathcal{X}} \left(df_p[\pi_{\widehat{x}}\,\widehat{\vw}]\right)^2 d\mu(p)$ is the mean-squared derivative along the \textit{projected extension} of $\widehat{\vw}$, the field of section~\ref{sec:AMG_field} whose magnitude decays with the angle between $\widehat{\vw}$ and the tangent space. Extrinsic eigenvalues thus entangle the variation of $f$ with the tilt of the embedding, measuring it along a comparison field that geometry alters in both direction and magnitude, while projected dominance can conflate ordering outright (Remark~\ref{rem:noncommute}).

Absent certifiable ridge structure, i.e., the more general ridge approximation setting, the intrinsic perspective is the more precise analog for identifying important directions: \textit{$G_0$ ranks directions by the variation of the function alone, along unit-norm fields, whereas either extrinsic reading weighs that variation by the geometry of a chosen embedding}. Identifying a single \textit{dominant} intrinsic direction still requires the spectral gap of regime (ii) in Remark~\ref{rem:equivalence}, and inactivity by virtue of small eigenvalues remains geodesically local, a consequence of the restrictions that make the integrand unique.

Figure~\ref{fig:amg_panel} in Appendix~\ref{app:gallery} gathers all of these toy examples---the ridge, degenerate, and quadratic cases above together with nonlinear ridge and non-ridge functions---across the normal-coordinate, sphere, and shadow views, so that the presence or absence of low-dimensional structure can be read down each column.


\section{Conclusions}
Both intrinsic and extrinsic eigenspaces generalize Euclidean active subspaces. The extrinsic formulation avoids parallel-transport approximations, and Remark~\ref{rem:noncommute} organizes what remains. Activity is first identified by integration, where the dominant projection agrees with the intrinsic construction quadratically, and is then extended over the manifold, either by recomputing a decomposition on each tangent space or by parallel transport of a single central frame. One route should be avoided throughout: projected dominance is not dominant projection, and activity should never be extended by projecting the dominant ambient eigenvector. The intrinsic view also represents the geometry more faithfully, free of a chosen embedding and of its projection to the central tangent space.

The relationship between the two intrinsic and extrinsic perspectives is now quantitative, and they are \textit{not} identical: on the central tangent space the eigenvalues agree to second order in the geodesic radius of the sampled domain, and under a spectral gap $\eta$ the dominant eigenspaces agree to $\mathcal{O}(R^2/\eta)$---exactly characterized on hyperspheres by the projection--transport identity (Lemma~\ref{lem:sphere_proj_trans}) and in general by Prop.~\ref{prop:local_agreement}---while the eigenvalues of the intrinsic $G_0$ retain an exact mean-squared-differential interpretation along the transported frame field (Lemma~\ref{lem:exact_eigvals}) at any radius. What can fail is the extrinsic practice of projecting the dominant \textit{ambient} eigenvector of $C_\iota$ (Thm.~\ref{thm:normal_eigvals}, Remark~\ref{rem:equivalence}), wherever that direction lies in $\text{Null}(\pi_0)$---a measure-zero set for the $2$-sphere examples, though not necessarily so in general. Since the eigenspaces are fundamentally different and can support different interpretations, neither is uniformly ``better'' than the other. Contrasting the two is often the most informative use, balancing computational expense against interpretability---all the more so because fully intrinsic formulations are often out of reach without approximations of intrinsic maps.

The limitation of the intrinsic AMG construction is that it remains \textit{geodesically local}. The exact identity of Lemma~\ref{lem:exact_eigvals} holds at any radius, but its reading---that a small $\lambda_i$ marks a genuinely inactive submanifold---is faithful only on a moderate geodesic ball about $p_0$. The active manifold-geodesics are therefore a mean-centered, geodesically local reduction over localized domain data rather than a global one. Capturing ridge geometry that turns across a compact domain calls for a genuinely foliated reduction, which lies in the scope of future work.

\textit{Outlook.} The constructions developed here extend beyond the hypersphere setting emphasized in this work. Since preshape spaces of landmarks and of length-normalized open curves are hyperspheres \cite{kendall1984,mio2007shape}, the present theory already governs response-driven dimension reduction for functions of shape along preshape geodesics. Forthcoming applications carry the same pipeline to the landmark-affine Grassmannian of separable shape tensors \cite{grey2023separable}, where gradient-faithful sensitivities (e.g., adjoint-based aerodynamic drag) and gradients learned from shape ensembles \cite{mukherjee2010learning,wu2010learning} instantiate the extrinsic--intrinsic comparison of Prop.~\ref{prop:local_agreement} at a measured support radius. A companion effort on numerically determining the injectivity radius of data-driven manifolds would make the support hypothesis of Def.~\ref{def:riemannian_view} auditable in practice.
\appendix
\section{Riemannian-geometry background}
\label{app:geo}
This appendix collects the standard differential-geometry definitions invoked in section~\ref{sec:intro-Riemannian-geo}; full developments are in \cite{Lee1997,Lee2003,Absil2008,doCarmo2017}.

\subsection{Tangent vectors and the differential}
The differential, viewed as a directional derivative, motivates the definition of a tangent vector.
\begin{definition}[Tangent Vector, \cite{Absil2008}, Ch. 3, Def. 3.5.1] \label{def:tan_vec}
	A \textit{tangent vector} $\vv_p \defeq \vv\vert_p$ to a smooth manifold $\mathcal{M}$ at a point $p$ is a map $\vv_p:C^{\infty}(\mathcal{M}) \rightarrow \mathbb{R}$ such that there exists a curve $\gamma:\mathcal{I} \subseteq \mathbb{R} \rightarrow \mathcal{M}$ with $\gamma(0) = p$ satisfying
	$$
	\vv_p(f) = \dot{\gamma}(0)f \defeq \frac{d}{dt}(f\circ \gamma)\biggl\vert_{t=0}
	$$
	for all $f \in C^{\infty}(\mathcal{M})$. Such a curve $\gamma$ is said to \textit{realize} the tangent vector $\vv_p$.
\end{definition}
By this definition $\vv_p$ is linear and satisfies the product rule, i.e.\ it is a \textit{derivation} \cite{Lee2003}; the $n$ derivations at a fixed $p$ form a basis for the tangent space $T_p\mathcal{M} \cong \mathbb{R}^n$ (\cite{Lee2003}, Prop. 3.15), visualized as a tangent plane \cite{Walker2015,Absil2008}. These tangent vectors are \textit{locally defined operators} that can change from one neighborhood to the next.
\begin{definition}[Differential, \cite{Lee2003}, Ch. 3]
	Let $F:\mathcal{M} \rightarrow \mathcal{N}$ be a smooth map between smooth manifolds and $\vv_p \in T_pM$. For all $p \in \mathcal{M}$ the \textit{differential of $F$ at $p$} is the map $dF_p:T_p\mathcal{M} \rightarrow T_{F(p)}\mathcal{N}$ for which $dF_p[\vv]$ is a derivation at $F(p)$ acting on $f \in C^{\infty}(\mathcal{N})$ by $dF_p[\vv](f) = \vv_p(f \circ F)$.
\end{definition}
In a chart $(U,\psi)$ with components $v^i = \vv_p(x^i) = dx^i$, this yields the directional derivative
$$
df_p:T_p\mathcal{M} \rightarrow \mathbb{R}:\vv_p \mapsto df_p[\vv] \defeq \sum_{i=1}^n\frac{\partial f}{\partial x^i}(p)\,dx^i,
$$
the coordinate form of the differential of a scalar-valued function; it generalizes the Jacobian in a coordinate-independent way, but any computation in components (such as \eqref{eq:eigvals}) implies a choice of local coordinates \cite{Walker2015}.

\subsection{Riemannian metric}
Every smooth manifold admits a Riemannian metric $g$, making it a Riemannian manifold $(\mathcal{M},g)$ \cite{Lee1997,Absil2008}; $g$ endows each tangent space with a smooth inner product, written interchangeably \cite{Absil2008,Lee1997,Lee2003,Walker2015,doCarmo2017}
$$
g(\vv_p,\vu_p) = \left\langle \vv,\vu \right\rangle_p, \qquad \vv,\vu \in T_p\mathcal{M}.
$$
Computationally the metric is a symmetric positive-definite matrix $G_x$ at each coordinate $x$, reducing to $I_n$ in the Euclidean case (section~\ref{sec:AS_intro}).

\subsection{Linear connection}
\begin{definition}[Linear Connection]
	Let $\mathcal{T}(\mathcal{M})$ denote the smooth vector fields on $M$. An \textit{affine connection} is a map $\nabla:\mathcal{T}(\mathcal{M}) \times \mathcal{T}(\mathcal{M}) \rightarrow \mathcal{T}(\mathcal{M}):(V,W) \mapsto \nabla_V W$ satisfying
	\begin{enumerate}
		\item $C^{\infty}(\mathcal{M})$-linearity in $V$: $\nabla_{fV_1 + gV_2}W = f\nabla_{V_1}W + g\nabla_{V_2}W$;
		\item $\mathbb{R}$-linearity in $W$: $\nabla_V(aW_1 + bW_2) = a\nabla_VW_1 + b\nabla_VW_2$;
		\item product rule in $W$: $\nabla_V(fW) = V(f)W + f\nabla_VW$.
	\end{enumerate}
\end{definition}

\subsection{Geodesic rescaling of submanifolds}
For each $V \in T\mathcal{M}$ the geodesic is
\begin{equation} \label{eq:geo_param}
\gamma_V(t) = \text{exp}(tV),
\end{equation}
defined while $\gamma(t)\in\mathcal{X}\subset\mathcal{M}$. As active-subspace computations rescale the domain for nondimensionalization \cite{Constantine2015}, we may rescale the $1$-dimensional geodesic submanifolds: for $\vv_i \in T_p\mathcal{M}$ and compact $\mathcal{I}_i \subset \mathbb{R}$,
\begin{equation} \label{eq:submanifold}
\mathcal{S}_{1,i} \defeq \lbrace \gamma(t) \in \mathcal{M} \, :\, \gamma(t) = \text{exp}_{p}(t\vv_i),\,\, t \in \mathcal{I}_i\rbrace,
\end{equation}
and, by rescaling (\cite{Lee1997}, Lemma 5.8), the normalized
\begin{equation} \label{eq:std_submanifold}
\mathcal{S}_{1,i} \defeq \lbrace \gamma(t) \in \mathcal{M} \, :\, \gamma(t) = \text{exp}_{p}(t\vv_i),\,\, t \in [-1,1]\rbrace,
\end{equation}
the two related by an affine transformation chosen for a given numerical study.


\section{Curve complexity and reachability}
\label{app:reach}
The remarks of section~\ref{sec:AMG_field} rest on two standard facts, recorded here informally.

\textit{Space-filling curves are long.} Realize $\mathcal{X}$ through an isometric embedding, so distance is the ambient $\Vert\cdot\Vert_2$, and let $\mathcal{X}$ have intrinsic dimension $n\geq2$ and $n$-dimensional volume $V$. Sample a length-$L$ curve $\gamma$ at the $\lfloor L/\tau\rfloor+1$ arc-length nodes $\gamma_k\defeq\gamma(k\tau)$. Because arc length dominates Euclidean distance ($\Vert\gamma_k-\gamma(t)\Vert_2\leq\vert k\tau-t\vert$ for unit-speed $\gamma$), a triangle inequality places every point of the $\tau$-tube $T_{\tau}\defeq\lbrace x:\Vert x-\gamma\Vert_2\leq\tau\rbrace$ within $2\tau$ of a node: if $\Vert x-\gamma(t')\Vert_2\leq\tau$ and $\gamma(t')$ lies within arc length $\tau$ of $\gamma_k$, then $\Vert x-\gamma_k\Vert_2\leq\Vert x-\gamma(t')\Vert_2+\Vert\gamma(t')-\gamma_k\Vert_2\leq2\tau$. Thus $\mathcal{X}\subseteq T_{\tau}$ is covered by the $\lfloor L/\tau\rfloor+1$ balls $B(\gamma_k,2\tau)$, and comparing $n$-dimensional volumes (a constant $c_n$ independent of $\tau$---absorbing the domain's bounded curvature and the ball-packing overlap---aside), $V\leq c_n\,(L/\tau+1)\,V_B(n)(2\tau)^n$ with $V_B(n)$ the unit-ball volume, so
\begin{equation}\label{eq:tube_length}
	L \;\geq\; \frac{V}{c_n\,V_B(n)\,2^n}\,\tau^{1-n} - \tau \;\xrightarrow[\;\tau\to0\;]{}\; \infty .
\end{equation}
The exponent is the \textit{intrinsic} dimension of $\mathcal{X}$, and it drives two distinct blow-ups. At fixed $n$ the length diverges as $\tau\to0$---the space-filling limit, in which even a curve confined to a curved embedded domain must grow without bound to remain $\tau$-dense. At a fixed scale $\tau$ below the domain's extent it instead grows \textit{exponentially in the dimension} $n$, $\tau^{1-n} = \tau/\tau^n$, since $\tau$-covering an $n$-dimensional domain already requires $\sim\tau^{-n}$ balls. This second modality is a curse of dimensionality bearing directly on the manifold-learning critique: when the latent (intrinsic) dimension is ambiguous and over-specified, the unconstrained single-curve reduction degenerates exponentially for the $\tau$-sized cover. Conversely, dense curves exist: a boustrophedon sweep of a $\tau$-grid passes within $\mathcal{O}(\tau)$ of every point of $\mathcal{X}$, and assigning each point the value of a continuous $f$ at its nearest curve point reproduces $f$ within its modulus of continuity $\omega_f(\tau)$. The tube radius therefore controls the achievable accuracy directly: any $\epsilon$ of Hypothesis~\ref{hypothesis} is met once $\omega_f(\tau)\leq\epsilon$, so the reduction error over unconstrained curves can be driven to zero for \textit{every} $f$ while the length \eqref{eq:tube_length} required to do so diverges---the ill-posedness of Remark~\ref{rem:dichotomy}.

\textit{Foliations and reachability.} Let $\mathcal{F}\subseteq T\mathcal{X}$ be a smooth constant-rank distribution. \textit{Frobenius} \cite{Lee2003}: if $\mathcal{F}$ is involutive---$[X,Y]$ a section of $\mathcal{F}$ whenever $X,Y$ are---then $\mathcal{X}$ is partitioned into immersed submanifolds (leaves) everywhere tangent to $\mathcal{F}$, and any $f$ with $df|_{\mathcal{F}}=0$ is constant on each connected leaf. \textit{Chow--Rashevskii} \cite{chow1939systeme,agrachev2020comprehensive}: if instead $\mathcal{F}$ is \textit{bracket-generating}---iterated Lie brackets of its sections span $T_x\mathcal{X}$ at every $x$ (H{\"o}rmander's condition)---then on a connected $\mathcal{X}$ any two points are joined by a path almost-everywhere tangent to $\mathcal{F}$, so any $f$ with $df|_{\mathcal{F}}=0$ is constant. Applied to section~\ref{sec:AMG_field}: the inactive fields and their iterated brackets annihilate $df$ and span the involutive closure $\overline{\mathcal{F}}$; when $\overline{\mathcal{F}}$ has constant rank it is either a proper subbundle---Frobenius yields a foliation and a manifold ridge---or all of $T\mathcal{X}$, whereupon Chow--Rashevskii forces $f$ constant.

\section{Hypersphere identities}
\label{app:sphere}
This appendix collects, with proofs, the elementary identities for the unit hypersphere $\iota(S^n) = \lbrace \widehat{x} \in \mathbb{R}^{n+1} \,:\, \Vert\widehat{x}\Vert_2 = 1\rbrace$, with metric induced by the ambient Euclidean inner product, used by Thm.~\ref{thm:normal_eigvals}, Lemma~\ref{lem:sphere_proj_trans}, Prop.~\ref{prop:local_agreement}, and section~\ref{sec:eg_2sphere}.

\subsection{Tangent spaces, projections, and the tangential connection}
Differentiating $\Vert c(t)\Vert_2^2 = 1$ along any curve $c$ on the sphere gives $\langle\dot{c}(t), c(t)\rangle = 0$, so the tangent space is the orthogonal hyperplane
\begin{equation}\label{eq:app_tan}
T_{\widehat{x}}\iota(S^n) = \lbrace \widehat{\vv} \in \mathbb{R}^{n+1} \,:\, \widehat{x}^{\top}\widehat{\vv} = 0 \rbrace, \quad \pi_{\widehat{x}} = I_{n+1} - \widehat{x}\widehat{x}^{\top},
\end{equation}
with $\pi_{\widehat{x}}$ the (symmetric, idempotent) orthogonal projection onto it. At a central point $\widehat{x}_0$ with orthonormal tangent basis $E_0 \in \mathbb{R}^{(n+1)\times n}$ the two representations agree: $\pi_0 \defeq E_0E_0^{\top} = I_{n+1} - \widehat{x}_0\widehat{x}_0^{\top}$, since both are the orthogonal projections onto the same hyperplane. The covariant derivative of the induced (tangential) connection is the tangential projection of the ambient derivative: for a vector field $V(t)$ along $c(t)$, $\tfrac{D}{dt}V = \pi_{c(t)}\dot{V}$ (\cite{Lee1997}; \cite{Absil2008}, Ch.~5). In particular, $V$ is \textit{parallel} along $c$ if and only if $\dot{V}(t)$ is normal to the sphere at $c(t)$ for every $t$.

\subsection{$S^n$ geodesics, exponential, and logarithm}
For unit $\widehat{\vv} \in T_{\widehat{x}}\iota(S^n)$ define
\begin{equation}\label{eq:app_geo}
\gamma(t) \defeq \cos(t)\,\widehat{x} + \sin(t)\,\widehat{\vv}, \qquad
\dot\gamma(t) = -\sin(t)\,\widehat{x} + \cos(t)\,\widehat{\vv}.
\end{equation}
Then: (i) $\Vert\gamma(t)\Vert_2 = 1$, since $\widehat{x}^{\top}\widehat{\vv} = 0$ kills the cross term; (ii) $\ddot\gamma(t) = -\gamma(t)$ is normal to the sphere at $\gamma(t)$, so $\tfrac{D}{dt}\dot\gamma = \pi_{\gamma(t)}\ddot\gamma = \boldsymbol{0}$ and $\gamma$ is a geodesic of the tangential connection; (iii) $\gamma(0) = \widehat{x}$, $\dot\gamma(0) = \widehat{\vv}$, and $\Vert\dot\gamma(t)\Vert_2 = 1$, so $\gamma$ is unit-speed and $d(\widehat{x},\gamma(t)) = t$ for $t \in [0,\pi)$. Hence $\text{exp}_{\widehat{x}}(t\widehat{\vv}) = \gamma(t)$, which is \eqref{eq:sphere_exp} in any dimension. Conversely, for non-antipodal $\widehat{x},\widehat{y}$, setting $d \defeq \cos^{-1}(\widehat{x}^{\top}\widehat{y})$ and $\widehat{\ve} \defeq \pi_{\widehat{x}}(\widehat{y})/\Vert\pi_{\widehat{x}}(\widehat{y})\Vert_2$ gives $\widehat{y} = \cos(d)\,\widehat{x} + \sin(d)\,\widehat{\ve}$, whence $\text{exp}^{-1}_{\widehat{x}}(\widehat{y}) = d\,\widehat{\ve}$, which is \eqref{eq:sphere_log}.

\subsection{Parallel transport along a radial geodesic}
Let $\gamma(t) = \cos(t)\,\widehat{x}_0 + \sin(t)\,\widehat{\ve}$ be the unit-speed geodesic from $\widehat{x}_0$ with initial direction $\widehat{\ve}$, and $\widehat{x} = \gamma(d)$. Two families of parallel fields along $\gamma$ suffice to transport any tangent vector:
\begin{enumerate}
	\item \textit{the velocity field} $\dot\gamma(t)$, parallel because $\gamma$ is a geodesic ($\tfrac{D}{dt}\dot\gamma = \boldsymbol{0}$);
	\item \textit{constant transverse fields}: for any fixed $\widehat{\vw} \in \mathbb{R}^{n+1}$ with $\widehat{\vw} \perp \widehat{x}_0$ and $\widehat{\vw} \perp \widehat{\ve}$, consider the vector field \textit{along $\gamma$} defined by $W(t) \defeq \widehat{\vw}$ for all $t$---constant as an element of the ambient $\mathbb{R}^{n+1}$. It satisfies $W(t)^{\top}\gamma(t) = \widehat{\vw}^{\top}\gamma(t) = 0$ for all $t$ (tangency along the whole geodesic), and its ambient derivative is $\dot{W}(t) = \boldsymbol{0}$ since the assignment does not vary with $t$; hence $W$ is parallel by the criterion of \eqref{eq:app_tan}.
\end{enumerate}
Since parallel transport is the unique linear isometry propagating parallel fields along a curve (\cite{Lee1997}, Thm.~4.11), transport from $\widehat{x} = \gamma(d)$ back to $\widehat{x}_0 = \gamma(0)$ acts as
\begin{equation}\label{eq:app_pt}
\begin{aligned}
\mathcal{P}_{\widehat{x}}\left[\dot\gamma(d)\right] &= \dot\gamma(0) = \widehat{\ve}, \\
\mathcal{P}_{\widehat{x}}\left[\widehat{\vw}\right] &= \widehat{\vw} \quad \text{for } \widehat{\vw} \perp \text{span}\{\widehat{x}_0,\widehat{\ve}\}.
\end{aligned}
\end{equation}
This is the $\mathrm{Gr}(1,n+1)$ specialization of the closed-form transport of \cite{edelman1998geometry}.

\subsection{Radial--transverse decomposition}
Retaining the notation of the preceding subsection: $\gamma$ is the unit-speed geodesic from $\widehat{x}_0$ with initial direction $\widehat{\ve} = \dot\gamma(0) \in T_{\widehat{x}_0}\iota(S^n)$, and $\widehat{x} = \gamma(d)$. With $\vu \defeq \dot\gamma(d) \in T_{\widehat{x}}\iota(S^n)$ the unit radial direction in the tangent space at $\widehat{x}$,
\begin{equation}\label{eq:app_split}
\begin{aligned}
T_{\widehat{x}}\iota(S^n) &= \text{span}\{\vu\} \oplus \mathcal{W}_{\perp}, \\
\mathcal{W}_{\perp} &\defeq \text{span}\{\widehat{x}_0,\widehat{\ve}\}^{\perp}.
\end{aligned}
\end{equation}
The decomposition lifts from the ambient space. Write $\mathcal{W} \defeq \text{span}\{\widehat{x}_0,\widehat{\ve}\}$ for the plane of the geodesic, so $\mathbb{R}^{n+1} = \mathcal{W} \oplus \mathcal{W}_{\perp}$ with $\mathcal{W}_{\perp}$ defined by the geodesic alone. Every $\widehat{\vw} \in \mathcal{W}_{\perp}$ is orthogonal to both $\widehat{x}_0$ and $\widehat{\ve}$, hence to $\widehat{x} = \cos(d)\,\widehat{x}_0 + \sin(d)\,\widehat{\ve}$, so all of $\mathcal{W}_{\perp}$ is tangent at $\widehat{x}$: $\mathcal{W}_{\perp} \subset T_{\widehat{x}}\iota(S^n)$. Within the plane, $\mathcal{W} \cap T_{\widehat{x}}\iota(S^n)$ is one-dimensional (the plane contains the normal direction $\widehat{x}$ itself) and $\vu \in \mathcal{W}$ satisfies $\widehat{x}^{\top}\vu = 0$, so that intersection is $\text{span}\{\vu\}$. The dimensions confirm the direct sum exhausts the tangent space: $\dim\text{span}\{\vu\}=1$ and $\dim\mathcal{W}_{\perp}=(n+1)-2=n-1$ (the orthogonal complement in $\mathbb{R}^{n+1}$ of the $2$-plane $\mathcal{W}$), summing to $n=\dim T_{\widehat{x}}\iota(S^n)$. Restricting the ambient splitting to the tangent hyperplane \eqref{eq:app_tan} implies every $\vv \in T_{\widehat{x}}\iota(S^n)$ decomposes uniquely as $\vv = \langle\vv,\vu\rangle\,\vu + \vv_{\perp}$ with $\vv_{\perp} \in \mathcal{W}_{\perp}$.

\subsection{Maclaurin bounds}
For all $t \in \mathbb{R}$,
\begin{equation}\label{eq:app_maclaurin}
0 \,\leq\, 1 - \cos t \,\leq\, \tfrac{t^2}{2}
\qquad\text{and}\qquad
\left\vert \tfrac{\sin t}{t} - 1 \right\vert \,\leq\, \tfrac{t^2}{6},
\end{equation}
the truncation errors of the Maclaurin (Taylor at zero) expansions of the cosine and sinc functions. Both follow by iterated integration of $\vert\sin s\vert \leq \vert s\vert$: first $\vert 1 - \cos t\vert = \vert\int_0^t \sin s\, ds\vert \leq \tfrac{t^2}{2}$ (nonnegativity of $1-\cos t$ is immediate), then $\vert\sin t - t\vert = \vert\int_0^t(\cos s - 1)\, ds\vert \leq \tfrac{\vert t\vert^3}{6}$.

\section{Proofs}
\label{app:proofs}
Proofs of the numbered results of the main text, in order of appearance.

\begin{proof}[Proof of Prop.~\ref{prop:eigs_and_constant_f}]
	The differential is the linear form $df_x[\boldsymbol{v}]=\boldsymbol{v}^{\top}\overline{\nabla}f(x)$,
	so \eqref{eq:C} gives $\mathbb{E}[df_x^2[\boldsymbol{v}]]=\boldsymbol{v}^{\top}C\boldsymbol{v}$ for every
	fixed $\boldsymbol{v}\in\mathbb{R}^n$. We prove (i)$\Rightarrow$(ii)$\Rightarrow$(iii)$\Rightarrow$(i).

	\textit{(i)$\Rightarrow$(ii).} Let $\boldsymbol{v}=\sum_{k=r+1}^n c_k\boldsymbol{w}_k\in\mathcal{S}$.
	The eigenvector relations $C\boldsymbol{w}_k=\lambda_k\boldsymbol{w}_k$, orthonormality
	$\boldsymbol{w}_k^{\top}\boldsymbol{w}_l=\delta_{kl}$, and $\lambda_k=0$ for $k>r$ give the integral
	identity
	\begin{equation*}
		\mathbb{E}[df_x^2[\boldsymbol{v}]]=\boldsymbol{v}^{\top}C\boldsymbol{v}
		=\sum_{k,l=r+1}^n c_k c_l\,\boldsymbol{w}_k^{\top}C\boldsymbol{w}_l
		=\sum_{k=r+1}^n c_k^2\lambda_k=0.
	\end{equation*}
	This becomes a pointwise statement assuming full support. The function
	$g\defeq(df_x[\boldsymbol{v}])^2$ is nonnegative and continuous, since $f\in C^1$, so its superlevel set 	$\lbrace g>0\rbrace$ is relatively open in $\mathcal{X}$. Were it nonempty, full support $\mathrm{supp}\,\rho=\mathcal{X}$---every nonempty relatively-open subset of $\mathcal{X}$ has positive $\rho$-measure---would force $\int_{\mathcal{X}}g\,\rho\,dx\geq\int_{\lbrace g>0\rbrace}
	g\,\rho\,dx>0$, contradicting $\mathbb{E}[df_x^2[\boldsymbol{v}]]=0$. Hence $\lbrace g>0\rbrace=\emptyset$; that is,
	$df_x[\boldsymbol{v}]=0$ for all $x\in\mathcal{X}$.

	\textit{(ii)$\Rightarrow$(iii).} For $x\in\mathcal{X}$ and $\boldsymbol{v}\in\mathcal{S}$ with
	$\lbrace x+t\boldsymbol{v}:t\in[0,1]\rbrace\subseteq\mathcal{X}$, the fundamental theorem of
	calculus and (ii) give
	\begin{equation*}
		f(x+\boldsymbol{v})-f(x)=\int_0^1 df_{x+t\boldsymbol{v}}[\boldsymbol{v}]\,dt=0 .
	\end{equation*}
	No convexity of $\mathcal{X}$ is used, and no claim is made along directions that leave
	$\mathcal{X}$ (e.g.\ an outward-normal $\boldsymbol{v}$ at a boundary point), for which no such
	segment exists.

	\textit{(iii)$\Rightarrow$(i).} Fix $i>r$. For $x\in\mathrm{int}\,\mathcal{X}$ choose $\delta>0$ with
	$B(x,\delta)\subseteq\mathcal{X}$; then for $\vert t\vert<\delta/\Vert\boldsymbol{w}_i\Vert_2$ the
	segment $\lbrace x+s\boldsymbol{w}_i:s\in[0,t]\rbrace\subseteq\mathcal{X}$, so (iii) gives $f(x+t\boldsymbol{w}_i)=f(x)$ on a neighborhood of $t=0$; every difference
	quotient therefore vanishes,
	\begin{equation*}
		df_x[\boldsymbol{w}_i]=\lim_{t\to 0}\frac{f(x+t\boldsymbol{w}_i)-f(x)}{t}
		=\lim_{t\to 0}\frac{0}{t}=0 .
	\end{equation*}
	Thus $df_x[\boldsymbol{w}_i]=0$ on $\mathrm{int}\,\mathcal{X}$. The map
	$x\mapsto df_x[\boldsymbol{w}_i]=\boldsymbol{w}_i^{\top}\overline{\nabla}f(x)$ is continuous on
	$\mathcal{X}$ ($f\in C^1$) and vanishes on the dense subset $\mathrm{int}\,\mathcal{X}$, since
	$\mathcal{X}=\overline{\mathrm{int}\,\mathcal{X}}$; hence $df_x[\boldsymbol{w}_i]\equiv 0$ on
	$\mathcal{X}$. Therefore $\lambda_i=\int_{\mathcal{X}}(df_x[\boldsymbol{w}_i])^2\rho\,dx=0$ by
	\eqref{eq:eigvals}, and since $i>r$ was arbitrary, (i) holds. Similar arguments appear in
	\cite{constantine2017near,constantine2014active}.
\end{proof}

\begin{proof}[Proof of Lemma~\ref{lem:exact_eigvals}]
	Work in normal coordinates, identifying $T_{p_0}\mathcal{M}\cong\mathbb{R}^n$ through the
	orthonormal frame for which $g_{p_0}=I_n$ (Prop.~\ref{prop:nml_props}, property~3). By the
	two-step flattening discussed after Def.~\ref{def:riemannian_view}---transport into
	$T_{p_0}\mathcal{M}$, then constant-$I_n$ extension over the coordinate space---$G_0$ of
	\eqref{eq:riemannian_view} is realized as the second-moment matrix
	$\int_{\mathcal{X}}\mathcal{P}_p[\nabla_{\mathcal{M}}f(p)]\,\mathcal{P}_p[\nabla_{\mathcal{M}}f(p)]^{\top}d\mu(p)$,
	with inner products on $T_{p_0}\mathcal{M}$ given by the canonical $\ell_2$ form.

	From the eigendecomposition $G_0=W_0\Lambda_0W_0^{\top}$ with orthonormal columns
	$\boldsymbol{w}_j$, the $i$th eigenvalue is the quadratic form
	$\lambda_i=\boldsymbol{w}_i^{\top}G_0\boldsymbol{w}_i$. Evaluating the rank-one integrand, for
	each $p$,
	$\boldsymbol{w}_i^{\top}\big(\mathcal{P}_p[\nabla_{\mathcal{M}}f(p)]\,\mathcal{P}_p[\nabla_{\mathcal{M}}f(p)]^{\top}\big)\boldsymbol{w}_i
	=\langle\mathcal{P}_p[\nabla_{\mathcal{M}}f(p)],\boldsymbol{w}_i\rangle_{p_0}^2$, so
	\begin{equation*}
		\lambda_i=\int_{\mathcal{X}}\langle\mathcal{P}_p[\nabla_{\mathcal{M}}f(p)],\boldsymbol{w}_i\rangle_{p_0}^2\,d\mu(p).
	\end{equation*}
	Parallel transport is a linear isometry,
	$\langle\mathcal{P}_p\vu,\mathcal{P}_p\vv\rangle_{p_0}=\langle\vu,\vv\rangle_p$ for all
	$\vu,\vv\in T_p\mathcal{M}$ (\cite{Lee1997}, Thm.~4.11). Taking
	$\vv=\mathcal{P}_p^{-1}[\boldsymbol{w}_i]=\boldsymbol{w}_i(p)$ and
	$\vu=\nabla_{\mathcal{M}}f(p)$, and using $\mathcal{P}_p\mathcal{P}_p^{-1}=\mathrm{id}$,
	\begin{align*}
		\langle\mathcal{P}_p[\nabla_{\mathcal{M}}f(p)],\boldsymbol{w}_i\rangle_{p_0}
		&=\langle\mathcal{P}_p[\nabla_{\mathcal{M}}f(p)],\mathcal{P}_p\mathcal{P}_p^{-1}[\boldsymbol{w}_i]\rangle_{p_0}\\
		&=\langle\mathcal{P}_p[\nabla_{\mathcal{M}}f(p)],\mathcal{P}_p[\boldsymbol{w}_i(p)]\rangle_{p_0}\\
		&=\langle\nabla_{\mathcal{M}}f(p),\boldsymbol{w}_i(p)\rangle_p.
	\end{align*}
	Finally, by the defining property of the Riemannian gradient (Def.~\ref{def:riemann_grad}),
	$\langle\nabla_{\mathcal{M}}f(p),\boldsymbol{w}_i(p)\rangle_p=df_p[\boldsymbol{w}_i(p)]$.
	Substituting,
	\begin{equation*}
		\lambda_i=\int_{\mathcal{X}}\big(df_p[\boldsymbol{w}_i(p)]\big)^2\,d\mu(p).
	\end{equation*}
	Non-negativity follows from positive semidefiniteness of $G_0$ (a $\mu$-average of rank-one
	positive semidefinite tensors); the ordering $\lambda_1\geq\dots\geq\lambda_n\geq0$ is the
	descending convention of Def.~\ref{def:AMG}.
\end{proof}

\begin{proof}[Proof of Thm.~\ref{thm:normal_eigvals}]
	Write $\widehat{\vg}(p) \defeq d\iota[\nabla_{\mathcal{M}} f(p)] \in T_{\widehat{x}}\iota(\mathcal{M}) \subset \mathbb{R}^m$,
	a $p \times 1$ column. Although $\widehat{\vg}(p)$ lies in a \textit{different} tangent subspace of
	$\mathbb{R}^m$ for each $p$, the outer product $\widehat{\vg}\,\widehat{\vg}^{\top}$ is formed in the
	\textit{fixed} ambient $\mathbb{R}^m$, so \eqref{eq:emb_opg} reads
	$C_{\iota} = \int_{\mathcal{X}} \widehat{\vg}\,\widehat{\vg}^{\top}\, d\mu(p) \in \mathbb{R}^{m\times m}$
	as an ordinary average of ambient $m\times m$ matrices---the extrinsic construction sidesteps the
	varying-tangent-space obstruction of section~\ref{sec:intro-Riemannian-geo} by never leaving $\mathbb{R}^m$.

	\textit{Step 1 (representation).} The matrix $E_0 \in \mathbb{R}^{m\times n}$ is constant over the
	integral, so it compresses each ambient outer product:
	\begin{align*}
	    E_0^{\top}C_{\iota}E_0 &= \int_{\mathcal{X}} E_0^{\top}\big(\widehat{\vg}\,\widehat{\vg}^{\top}\big)E_0\, d\mu(p)\\
	&= \int_{\mathcal{X}} \big(E_0^{\top}\widehat{\vg}\big)\big(E_0^{\top}\widehat{\vg}\big)^{\top} d\mu(p),
	\end{align*}
	which is the first identity.

	\textit{Step 2 (reshape the outer product).} For $\vw\in\mathbb{R}^n$, multiplying by $\vw$ on either
	side turns each rank-one integrand into a squared scalar---$\vw^{\top}E_0^{\top}\widehat{\vg}$ equals
	its own transpose---so
	\begin{align*}
		\vw^{\top}\!\left(E_0^{\top}C_{\iota}E_0\right)\!\vw &= \int_{\mathcal{X}} \big(\vw^{\top}E_0^{\top}\widehat{\vg}\big)^2 d\mu(p) \\
		&= \int_{\mathcal{X}} \big(\widehat{\vw}^{\top}\widehat{\vg}\big)^2 d\mu(p),
	\end{align*}
	with $\widehat{\vw} \defeq E_0\vw$.

	\textit{Step 3 (swap in the projection).} The matrix $\pi_0 = E_0E_0^{\top}$ is symmetric and, since
	$E_0^{\top}E_0 = I_n$, idempotent: $\pi_0^2 = E_0(E_0^{\top}E_0)E_0^{\top} = \pi_0$. It fixes central
	directions, $\pi_0\widehat{\vw} = E_0(E_0^{\top}E_0)\vw = E_0\vw = \widehat{\vw}$. Hence, by symmetry
	of $\pi_0$,
	$$
	\widehat{\vw}^{\top}\widehat{\vg} = \big(\pi_0\widehat{\vw}\big)^{\top}\widehat{\vg}
	= \widehat{\vw}^{\top}\big(\pi_0\,\widehat{\vg}\big),
	$$
	moving the projection off the (already central) test direction and onto the gradient. This gives the second identity.

	\textit{Step 4 (partial isometry).} The relation $E_0^{\top}E_0E_0^{\top} = E_0^{\top}$ makes $E_0^{\top}$
	a partial isometry with initial projection $\pi_0$ \cite{hearon1967partially}. Its effect on norms is
	the load-bearing consequence: for every $\widehat{\vv} \in \mathbb{R}^m$,
	$$
	\Vert E_0^{\top}\widehat{\vv}\Vert_2^2 = \widehat{\vv}^{\top}E_0E_0^{\top}\widehat{\vv}
	= \widehat{\vv}^{\top}\pi_0\widehat{\vv} = \widehat{\vv}^{\top}\pi_0^2\widehat{\vv}
	= \Vert\pi_0\widehat{\vv}\Vert_2^2,
	$$
	using $\pi_0 = \pi_0^{\top} = \pi_0^2$. The compression $E_0^{\top}$ thus preserves the
	central-tangent component of each gradient and annihilates its orthogonal complement: any ambient
	$\widehat{\vb}$ with $\pi_0\widehat{\vb} = \boldsymbol{0}$ contributes nothing to $E_0^{\top}C_{\iota}E_0$.
\end{proof}

\begin{proof}[Proof of Lemma~\ref{lem:sphere_proj_trans}]
	Let $\widehat{\ve} \in T_{\widehat{x}_0}\iota(S^n)$ be the unit initial direction of the geodesic $\gamma$ from $\widehat{x}_0$ to $\widehat{x}$, so that by \eqref{eq:app_geo}
	$$
	\widehat{x} = \gamma(d) = \cos(d)\,\widehat{x}_0 + \sin(d)\,\widehat{\ve}
    $$
    and
    $$
    \vu = \dot\gamma(d) = -\sin(d)\,\widehat{x}_0 + \cos(d)\,\widehat{\ve}.
	$$
	\textit{Step 1 (decompose the tangent space).} By the radial--transverse splitting \eqref{eq:app_split}, every $\vv \in T_{\widehat{x}}\iota(S^n)$ writes uniquely as $\vv = \langle\vv,\vu\rangle\,\vu + \vv_{\perp}$ with $\vv_{\perp} \perp \text{span}\{\widehat{x}_0,\widehat{\ve}\}$.
	\textit{Step 2 (transport each piece).} By the closed-form transport \eqref{eq:app_pt}, $\mathcal{P}_{\widehat{x}}[\vu] = \dot\gamma(0) = \widehat{\ve}$ and $\mathcal{P}_{\widehat{x}}[\vv_{\perp}] = \vv_{\perp}$; by linearity,
	$$
	\mathcal{P}_{\widehat{x}}[\vv] = \langle\vv,\vu\rangle\,\widehat{\ve} + \vv_{\perp}.
	$$
	\textit{Step 3 (project each piece).} With $\pi_0 = I_{n+1} - \widehat{x}_0\widehat{x}_0^{\top}$ \eqref{eq:app_tan}: the radial piece has $\widehat{x}_0^{\top}\vu = -\sin(d)$ and, by definition of $\vu$, $\vu + \sin(d)\widehat{x}_0 = \cos(d)\widehat{\ve}$, so
	\begin{align*}
	    \pi_0\,\vu &= \vu - (\widehat{x}_0^{\top}\vu)\,\widehat{x}_0 \\
        &= \vu + \sin(d)\,\widehat{x}_0 \\
        &= \cos(d)\,\widehat{\ve} \\
        &= \cos(d)\,\mathcal{P}_{\widehat{x}}[\vu].
	\end{align*}
	The transverse piece has $\widehat{x}_0^{\top}\vv_{\perp} = 0$, so $\pi_0\,\vv_{\perp} = \vv_{\perp} = \mathcal{P}_{\widehat{x}}[\vv_{\perp}]$---transverse components are projected \textit{and} transported identically.
	\textit{Step 4 (assemble).} Combining Steps 2--3, noting specifically that Step 2 implies $\mathcal{P}_{\widehat{x}}[\vv_{\perp}] = \vv_{\perp} = \mathcal{P}_{\widehat{x}}[\vv] - \langle \vv, \vu \rangle\widehat{\ve} = \mathcal{P}_{\widehat{x}}[\vv] - \langle \vv, \vu \rangle\mathcal{P}_{\widehat{x}}[\vu]$,
	\begin{align*}
	    \pi_0\,\vv &= \pi_0(\langle\vv,\vu\rangle\,\vu + \vv_{\perp}) \\
        &= \langle\vv,\vu\rangle\cos(d)\,\mathcal{P}_{\widehat{x}}[\vu] + \mathcal{P}_{\widehat{x}}[\vv_{\perp}] \\
        & = \langle\vv,\vu\rangle\cos(d)\,\mathcal{P}_{\widehat{x}}[\vu] + \mathcal{P}_{\widehat{x}}[\vv] -  \langle \vv, \vu \rangle \mathcal{P}_{\widehat{x}}[\vu]\\
        &= \mathcal{P}_{\widehat{x}}[\vv] - (1-\cos d)\,\langle\vv,\vu\rangle\,\mathcal{P}_{\widehat{x}}[\vu].
	\end{align*}
\end{proof}

\begin{proof}[Proof of Prop.~\ref{prop:local_agreement}]
	Identify $T_{p_0}\mathcal{M}$ with $T_{\widehat{x}_0}\iota(\mathcal{M}) \subset \mathbb{R}^m$ via $d\iota$, with coordinates in the orthonormal basis $E_0$, and write, for $p \in \mathcal{X}$ at geodesic distance $d \defeq d(p,p_0) \leq R$,
	$$
	\widehat{\vg}(p) \defeq d\iota\left[\nabla_{\mathcal{M}}f(p)\right] \in T_{\widehat{x}}\iota(\mathcal{M})
    $$
    and
    $$
	\tilde{\vg}(p) \defeq d\iota\left[\mathcal{P}_p[\nabla_{\mathcal{M}}f(p)]\right] \in T_{\widehat{x}_0}\iota(\mathcal{M}),
	$$
	the ambient representations of the gradient and of its parallel transport. The coordinates of the transported gradient in the frame $E_0$ are $E_0^{\top}\tilde{\vg}$, so the matrix of $G_0$ in that frame is $G_0 = \int_{\mathcal{X}}(E_0^{\top}\tilde{\vg})(E_0^{\top}\tilde{\vg})^{\top}d\mu(p)$, while Thm.~\ref{thm:normal_eigvals} gives $E_0^{\top}C_{\iota}E_0 = \int_{\mathcal{X}}(E_0^{\top}\widehat{\vg})(E_0^{\top}\widehat{\vg})^{\top}d\mu(p)$. Since $E_0^{\top}\pi_0 = E_0^{\top}$, we may replace $\widehat{\vg}$ by $\vq \defeq \pi_0\widehat{\vg} \in T_{\widehat{x}_0}\iota(\mathcal{M})$ without changing the integrand, and the difference of interest becomes
	\begin{equation} \label{eq:prop_delta}
	\Delta \defeq E_0^{\top}C_{\iota}E_0 - G_0 = \int_{\mathcal{X}} E_0^{\top}\left( \vq\,\vq^{\top} - \tilde{\vg}\,\tilde{\vg}^{\top} \right)E_0\, d\mu(p).
	\end{equation}
	\textit{Step 1 (pointwise comparison of projection and transport).} On the unit hypersphere, Lemma~\ref{lem:sphere_proj_trans} applied to $\vv = \widehat{\vg}(p)$ gives
	$$
	\vq = \tilde{\vg} + \vr, \qquad \vr = -(1-\cos d)\,\langle\widehat{\vg},\vu\rangle\,\mathcal{P}_{\widehat{x}}[\vu],
	$$
	with $\vu$ the radial direction at $\widehat{x} = \iota(p)$. The vector $\mathcal{P}_{\widehat{x}}[\vu]$ is unit, $\vert\langle\widehat{\vg},\vu\rangle\vert \leq \Vert\widehat{\vg}\Vert_2$ by Cauchy--Schwarz, and transport is an isometry, $\Vert\tilde{\vg}\Vert_2 = \Vert\widehat{\vg}\Vert_2 = \Vert\nabla_{\mathcal{M}}f\Vert_g$; hence, by the Maclaurin bound \eqref{eq:app_maclaurin},
	$$
	\Vert\vr\Vert_2 \leq (1-\cos d)\,\Vert\nabla_{\mathcal{M}}f\Vert_g \leq \tfrac{d^2}{2}\,\Vert\nabla_{\mathcal{M}}f\Vert_g.
	$$
	\textit{Step 2 (difference of rank-one products).} For any vectors related by $\vq = \tilde{\vg} + \vr$,
	$$
	\vq\,\vq^{\top} - \tilde{\vg}\,\tilde{\vg}^{\top} = \tilde{\vg}\,\vr^{\top} + \vr\,\tilde{\vg}^{\top} + \vr\,\vr^{\top}.
	$$
	For any unit $\xi$, $\vert\xi^{\top}(\tilde{\vg}\vr^{\top} + \vr\tilde{\vg}^{\top})\xi\vert = 2\vert(\xi^{\top}\tilde{\vg})(\vr^{\top}\xi)\vert \leq 2\Vert\tilde{\vg}\Vert_2\Vert\vr\Vert_2$, and $\Vert\vr\vr^{\top}\Vert_2 = \Vert\vr\Vert_2^2$. With Step 1,
	\begin{align*}
	    \left\Vert \vq\,\vq^{\top} - \tilde{\vg}\,\tilde{\vg}^{\top} \right\Vert_2 &\leq 2\Vert\tilde{\vg}\Vert_2\Vert\vr\Vert_2 + \Vert\vr\Vert_2^2 \\
        & \leq \left(d^2 + \tfrac{d^4}{4}\right)\Vert\nabla_{\mathcal{M}}f\Vert_g^2 \\
        &\leq \left(1 + \tfrac{R^2}{4}\right) d^2\, \Vert\nabla_{\mathcal{M}}f\Vert_g^2.
	\end{align*}
	\textit{Step 3 (integrate).} The spectral norm of a matrix-valued integral is at most the integral of the spectral norms---for unit $\xi,\zeta$, $\xi^{\top}(\int A\, d\mu)\zeta = \int \xi^{\top}A\,\zeta\, d\mu \leq \int\Vert A\Vert_2\, d\mu$---and compression by orthonormal columns does not increase it, $\Vert E_0^{\top}AE_0\Vert_2 \leq \Vert A\Vert_2$ (the quadratic form is evaluated over the subspace $\text{Range}(E_0)$). Applying both to \eqref{eq:prop_delta} and inserting Step 2 yields the displayed bound with $C = 1 + R^2/4$ on the unit hypersphere.
    
	\textit{Step 4 (general isometrically embedded manifolds).} The hypersphere enters only through Step 1; on a general $\iota(\mathcal{M})$ the same conclusion holds with $\Vert\vr\Vert_2 \leq C_0\, d^2\,\Vert\nabla_{\mathcal{M}}f\Vert_g$ for a constant $C_0$ determined by the embedding, as follows. Let $x = \varphi^{-1}(p)$, $\bar{x} \defeq x/\Vert x\Vert_2$, and consider the frames $F(t) \defeq d(\iota\circ\text{exp}_{p_0})_{t\bar{x}}:\mathbb{R}^n \rightarrow \mathbb{R}^m$ along the radial geodesic, a smooth family with $F(0) = E_0$. Taylor's theorem with integral remainder gives $F(t)[\vw] = E_0\vw + t\,\mathrm{II}(\bar{x},\vw) + \mathcal{O}(t^2)$, where the first-order coefficient is the second fundamental form by the Gauss formula (\cite{doCarmo2017}, Ch.~6; \cite{Lee1997}). 
    
    The coefficient is \emph{purely} $\mathrm{II}$---with no tangential contribution---precisely because we work in normal coordinates: there the Christoffel symbols vanish at the center, $\Gamma(0)=\boldsymbol{0}$, so the tangential part of the embedding Hessian $\partial^2(\iota\circ\text{exp}_{p_0})$ vanishes at $\widehat{x}_0$ and only the normal second fundamental form survives.
    
    Since $\mathrm{II}$ is \textit{normal-valued} at $\widehat{x}_0$, $\pi_0\,\mathrm{II}(\bar{x},\vw) = \boldsymbol{0}$ and
	$$
	\pi_0\, F(t)[\vw] = E_0\vw + \mathcal{O}(t^2).
	$$
	Meanwhile the normal-coordinate components $V(t)$ of a parallel field along the same geodesic solve $\dot{V}^k = -\Gamma^k_{ij}(t\bar{x})\,\dot{\gamma}^i V^j$, where the Christoffel symbols vanish at the center and grow linearly, $\Vert\Gamma(t\bar{x})\Vert = \mathcal{O}(t)$ (Prop.~\ref{prop:nml_props}; \cite{doCarmo2017}); integrating and applying Gr\"onwall's inequality gives $V(t) = V(0) + \mathcal{O}(t^2)$, i.e.\ the transported and coordinate frames agree to second order. Combining the two displays: for every unit $\vv \in T_p\mathcal{M}$, $\pi_0\,d\iota[\vv] = d\iota[\mathcal{P}_p\vv] + \mathcal{O}(d^2)$ uniformly over the support, with constants controlled by $\Vert\mathrm{II}\Vert$ and the curvature; Steps 2--3 then apply verbatim. Unlike the explicit hypersphere value $C=1+R^2/4$ of Step~3, this general $C$ is left qualitative---a fixed function of $\sup_{\mathcal{X}}\Vert\mathrm{II}\Vert$ and the sectional curvature over the support, not tracked to a closed form.
    
	\textit{Step 5 (spectral consequences).} For symmetric matrices, Weyl's inequality gives $\vert\lambda_i(E_0^{\top}C_{\iota}E_0) - \lambda_i(G_0)\vert \leq \Vert\Delta\Vert_2$ for each $i$ \cite{Golub1996}. For the eigenspaces: if $\eta \defeq \lambda_r(G_0) - \lambda_{r+1}(G_0) > \Vert\Delta\Vert_2$, then by Weyl the trailing spectrum of $E_0^{\top}C_{\iota}E_0$ is separated from the leading spectrum of $G_0$ by at least $\eta - \Vert\Delta\Vert_2$, and the Davis--Kahan $\sin\Theta$ theorem \cite{Golub1996} bounds the distance between the dominant $r$-dimensional eigenspaces by $\Vert\Delta\Vert_2/(\eta - \Vert\Delta\Vert_2)$---in particular $\mathcal{O}(R^2/\eta)$ as $R \rightarrow 0$.
\end{proof}

\section{Gallery of toy examples}
\label{app:gallery}
Figure~\ref{fig:amg_panel} collects the toy examples of section~\ref{sec:eg_2sphere} across the three complementary views used in the main text: the function in normal coordinates at the Karcher mean $\widehat{x}_0$ (color field $f\circ\text{exp}_{\widehat{x}_0}$, black arrows the parallel-transported gradient, rose/marigold the active/inactive axes); the function on the sphere (sparse tangential-gradient grid with the active/inactive manifold-geodesics and a ringed Monte-Carlo subset); and the shadow plot ($f$ along a sweep of geodesics with the samples projected onto the active coordinate). Reading down each column shows how low-dimensional structure---or its absence---manifests in each representation.
\begin{figure*}[p]
	\centering
	\includegraphics[width=\textwidth,height=0.95\textheight,keepaspectratio]{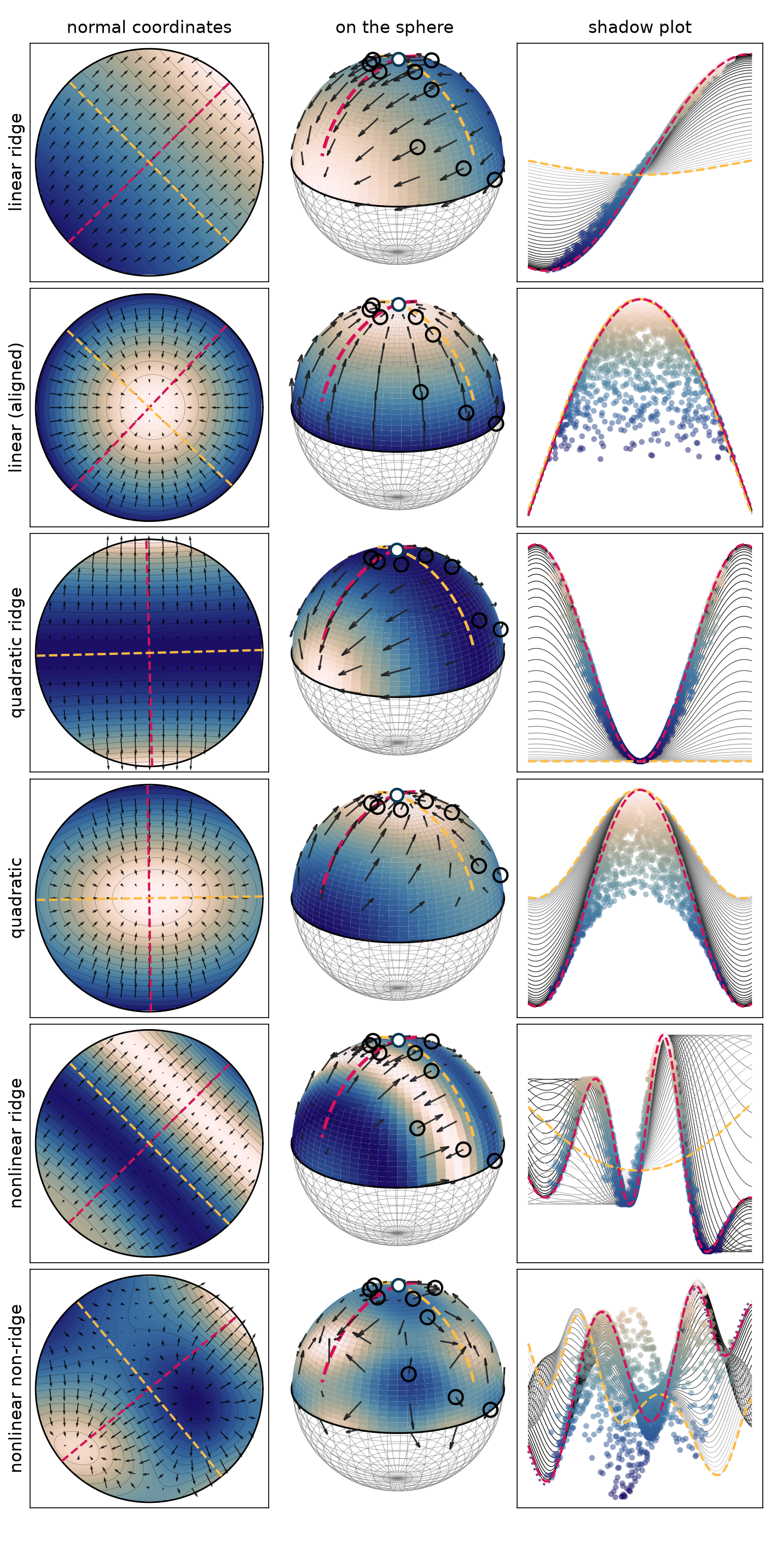}
	\caption{Gallery of toy examples over $\iota(S^2)$. \textbf{Rows}: linear ridge; degenerate (aligned) linear ridge; quadratic ridge; ordered quadratic; nonlinear ridge; nonlinear non-ridge. \textbf{Columns}: normal coordinates at $\widehat{x}_0$; the function on the sphere (gradient arrows scaled per panel to their largest, so magnitudes are relative); and the shadow plot (samples colored by $f$). The active manifold-geodesic is rose, the inactive marigold.}
	\label{fig:amg_panel}
\end{figure*}

\section*{Acknowledgments}
These opinions, recommendations, findings, and conclusions do not necessarily reflect the views or policies of NIST or the United States Government. Official contribution of the National Institute of Standards and Technology; not subject to copyright in the United States. This manuscript was edited with the assistance of various AI tools. AI was used to refine language, improve clarity, and enhance readability in accordance with the author’s instructions. All content, scientific claims, and conclusions have been reviewed and verified by the author to ensure accuracy and originality.

\bibliographystyle{plain}
\bibliography{workflow_v5}

\begin{thebibliography}{10}

\bibitem{absil2004riemannian}
P.-A. Absil, R.~Mahony, and R.~Sepulchre.
\newblock Riemannian geometry of grassmann manifolds with a view on algorithmic
  computation.
\newblock {\em Acta Applicandae Mathematica}, 80(2):199--220, 2004.

\bibitem{Absil2008}
P.-A. Absil, R.~Mahony, and R.~Sepulchre.
\newblock {\em Optimization Algorithms on Matrix Manifolds}.
\newblock Princeton University Press, 3 Market Place, Woodstock, Oxfordshire
  OX20 1SY, 2008.

\bibitem{agrachev2020comprehensive}
Andrei Agrachev, Davide Barilari, and Ugo Boscain.
\newblock {\em A Comprehensive Introduction to Sub-Riemannian Geometry}, volume
  181 of {\em Cambridge Studies in Advanced Mathematics}.
\newblock Cambridge University Press, Cambridge, 2020.

\bibitem{alain2014regularized}
Guillaume Alain and Yoshua Bengio.
\newblock What regularized auto-encoders learn from the data-generating
  distribution.
\newblock {\em The Journal of Machine Learning Research}, 15(1):3563--3593,
  2014.

\bibitem{SEQUOIA2017}
J.~J. Alonso, M.~S. Eldred, P.~Constantine, K.~Duraisamy, C.~Farhat,
  G.~Iaccarino, and G.~Jakeman.
\newblock Scalable environment for quantification of uncertainty and
  optimization in industrial applications {(SEQUOIA)}.
\newblock In {\em 19th AIAA Non-Deterministic Approaches Conference}, pages
  1--19, Grapevine, TX, 2017.

\bibitem{Alonso2017}
J.~J. Alonso, M.~S. Eldred, P.~G. Constantine, K.~Duraisamy, C.~Farhat,
  G.~Iaccarino, and J.~Jakeman.
\newblock Scalable environment for quantification of uncertainty and
  optimization in industrial applications {(SEQUOIA)}.
\newblock {\em 19th AIAA Non-Deterministic Approaches Conference}, 2017.

\bibitem{berger1981some}
E.~Berger, R.~Bryant, and P.~Griffiths.
\newblock Some isometric embedding and rigidity results for riemannian
  manifolds.
\newblock {\em Proceedings of the National Academy of Sciences},
  78(8):4657--4660, 1981.

\bibitem{bhattacharya2003large}
R.~Bhattacharya and V.~Patrangenaru.
\newblock Large sample theory of intrinsic and extrinsic sample means on
  manifolds. {I}.
\newblock {\em The Annals of Statistics}, 31:1--29, 2003.

\bibitem{bridges2019active}
R.~A. Bridges, A.~D. Gruber, C.~R. Felder, M.~Verma, and C.~Hoff.
\newblock Active manifolds: A non-linear analogue to active subspaces.
\newblock {\em arXiv preprint arXiv:1904.13386}, 2019.

\bibitem{Chatterjee2000}
A.~Chatterjee.
\newblock An introduction to the proper orthogonal decomposition.
\newblock {\em Current Science}, Vol. 78(No. 7), April 2000.

\bibitem{chen2023intrinsic}
B.~Chen, S.~Dai, and Z.~Yu.
\newblock Intrinsic minimum average variance estimation for sufficient
  dimension reduction with symmetric positive definite matrices and beyond.
\newblock {\em arXiv preprint arXiv:2302.13059}, 2023.

\bibitem{chow1939systeme}
Wei-Liang Chow.
\newblock {\"U}ber systeme von linearen partiellen differentialgleichungen
  erster ordnung.
\newblock {\em Mathematische Annalen}, 117:98--105, 1939.

\bibitem{Constantine2015}
P.~G. Constantine.
\newblock {\em Active Subspaces: Emergine Ideas in Dimension Reduction for
  Parameter Studies}.
\newblock Society for Industrial and Applied Mathematics, 3600 Market Street,
  6th Floor, Phildelphia, PA 19104-2688, 2015.

\bibitem{constantine2014active}
P.~G. Constantine, E.~Dow, and Q.~Wang.
\newblock Active subspace methods in theory and practice: Applications to
  kriging surfaces.
\newblock {\em SIAM Journal on Scientific Computing}, 36(4):A1500--A1524, 2014.

\bibitem{constantine2017near}
P.~G Constantine, A.~Eftekhari, J.~Hokanson, and R.~A. Ward.
\newblock A near-stationary subspace for ridge approximation.
\newblock {\em Computer Methods in Applied Mechanics and Engineering},
  326:402--421, 2017.

\bibitem{Constantine2015exploiting}
P.~G. Constantine, M.~Emory, J.~Larsson, and G.~Iaccarino.
\newblock Exploiting active subspaces to quantify uncertainty in the numerical
  simulation of the hyshot ii scramjet.
\newblock {\em Journal of Computational Physics}, 302:1--20, 2015.

\bibitem{constantine2014computing}
P.~G. Constantine and D.~Gleich.
\newblock Computing active subspaces with monte carlo.
\newblock {\em arXiv preprint arXiv:1408.0545}, 2014.

\bibitem{Constantine11c}
P.~G. Constantine, Q.~Wang, A.~Doostan, and G.~Iaccarino.
\newblock A surrogate accelerated bayesian inverse analysis of the {HyShot II}
  flight data.
\newblock {\em 52nd AIAA/ASME/ASCE/AHS/ASC Structures, Structural Dynamics and
  Materials Conference, Denver, Colorado}, 2011.
\newblock AIAA 2011-2037.

\bibitem{Cook1998}
R.~D. Cook.
\newblock {\em Regression Graphics: Ideas for Studying Regressions through
  Graphics}, volume Vol. 482.
\newblock John Wiley \& Sons. Inc., 605 Third Avenue, New York, 2009.

\bibitem{doCarmo2017}
M.~P. Do~Carmo.
\newblock {\em Differential Geometry of Curves and Surfaces: Revised and
  Updated Second Edition}.
\newblock Courier Dover Publications, 2016.

\bibitem{edelman1998geometry}
A.~Edelman, T.~A. Arias, and S.~T. Smith.
\newblock The geometry of algorithms with orthogonality constraints.
\newblock {\em SIAM journal on Matrix Analysis and Applications},
  20(2):303--353, 1998.

\bibitem{fletcher2004principal}
P.~T. Fletcher, C.~Lu, S.~M. Pizer, and S.~Joshi.
\newblock Principal geodesic analysis for the study of nonlinear statistics of
  shape.
\newblock {\em IEEE transactions on medical imaging}, 23(8):995--1005, 2004.

\bibitem{genovese2014nonparametric}
C.~R. Genovese, M.~Perone-Pacifico, I.~Verdinelli, and L.~Wasserman.
\newblock Nonparametric ridge estimation.
\newblock {\em The Annals of Statistics}, 42(4):1511--1545, 2014.

\bibitem{glaws2017inverse}
A.~Glaws, P.~Constantine, and R.~D. Cook.
\newblock Inverse regression for ridge recovery i: Theory.
\newblock {\em arXiv preprint arXiv}, 1702, 2017.

\bibitem{Glaws2018}
A.~Glaws and P.~G. Constantine.
\newblock Gauss--christoffel quadrature for inverse regression: applications to
  computer experiments.
\newblock {\em Statistics and Computing}, Jun 2018.

\bibitem{glawsMHD}
A.~Glaws, P.~G. Constantine, J.~N. Shadid, and T.~M. Wildey.
\newblock Dimension reduction in magnetohydrodynamics power generation models:
  Dimensional analysis and active subspaces.
\newblock {\em Statistical Analysis and Data Mining: The ASA Data Science
  Journal}, 10(5):312--325, 2017.

\bibitem{Golub1996}
G.H. Golub and C.F. Van~Loan.
\newblock {\em Matrix Computations}.
\newblock Johns Hopkins Univesity Press, 3rd ed. edition, 1996.

\bibitem{Grey2017}
Z.~J. Grey and P.~G Constantine.
\newblock Active subspaces of airfoil shape parameterizations.
\newblock {\em AIAA Journal}, 56(5):2003--2017, 2018.

\bibitem{grey2019}
Z.~J. Grey, P.~G. Constantine, and A.~White.
\newblock Enabling aero-engine thermal model calibration using active
  subspaces.
\newblock {\em AIAA Propulsion and Energy Forum}, 2019.

\bibitem{grey2023separable}
Z.~J. Grey, O.~A. Doronina, and A.~Glaws.
\newblock Separable shape tensors for aerodynamic design.
\newblock {\em Journal of Computational Design and Engineering},
  10(1):468--487, 2023.

\bibitem{grey2025explainable}
Z.~J. Grey, N.~Fisher, and A.~Glaws.
\newblock Explainable binary classification of separable shape ensembles.
\newblock {\em Journal of Mathematical Imaging and Vision}, 2025.
\newblock in revision (revised manuscript pending reviewer responses).

\bibitem{hang2009rigidity}
F.~Hang and X~Wang.
\newblock Rigidity theorems for compact manifolds with boundary and positive
  ricci curvature.
\newblock {\em Journal of Geometric Analysis}, 19(3):628--642, 2009.

\bibitem{hastie1989principal}
T.~Hastie and W.~Stuetzle.
\newblock Principal curves.
\newblock {\em Journal of the American Statistical Association},
  84(406):502--516, 1989.

\bibitem{hearon1967partially}
John~Z Hearon.
\newblock Partially isometric matrices.
\newblock {\em J. Res. Nat. Bur. Standards Sect. B}, 71:225--228, 1967.

\bibitem{hokanson2017}
J.~M Hokanson and P.~G Constantine.
\newblock Data-driven polynomial ridge approximation using variable projection.
\newblock {\em SIAM Journal on Scientific Computing}, 40(3):A1566--A1589, 2018.

\bibitem{hokanson2018data}
J.~M. Hokanson and P.~G Constantine.
\newblock Data-driven polynomial ridge approximation using variable projection.
\newblock {\em SIAM Journal on Scientific Computing}, 40(3):A1566--A1589, 2018.

\bibitem{hokanson2019lipschitz}
J.~M. Hokanson and P.~G Constantine.
\newblock The lipschitz matrix: A tool for parameter space dimension reduction.
\newblock {\em arXiv preprint arXiv:1906.00105}, 2019.

\bibitem{huckemann2010intrinsic}
S.~Huckemann, T.~Hotz, and A.~Munk.
\newblock Intrinsic shape analysis: Geodesic principal component analysis for
  {R}iemannian manifolds modulo isometric {L}ie group actions.
\newblock {\em Statistica Sinica}, 20(1):1--58, 2010.

\bibitem{Jolliffe2002}
I.~T. Jolliffe.
\newblock {\em Principal Component Analysis}.
\newblock Springer, 175 Fifth Avenue, New York., 2nd edition, 2002.

\bibitem{Joshi2007}
S.H. Joshi, E.~Klassen, A.~Srivastava, and I.~Jermyn.
\newblock A novel representation for riemannian analysis of elastic curves in
  $\mathbb{R}^n$.
\newblock {\em Proc. IEEE Computer Vision and Pattern Recognition {(CVPR)}},
  pages pp. 1--7, 2007.

\bibitem{kendall1984}
D.~G. Kendall.
\newblock Shape manifolds, procrustean metrics, and complex projective spaces.
\newblock {\em Bulletin of the London Mathematical Society}, 16(2):81--121,
  1984.

\bibitem{klassen2004analysis}
E.~Klassen, A.~Srivastava, M.~Mio, and S.~H. Joshi.
\newblock Analysis of planar shapes using geodesic paths on shape spaces.
\newblock {\em IEEE transactions on pattern analysis and machine intelligence},
  26(3):372--383, 2004.

\bibitem{Lee1997}
J.~Lee.
\newblock {\em Riemannian Manifolds: An Introduction to Curvature}.
\newblock Springer-Verlag, New York, 1997.

\bibitem{Lee2003}
J.~M. Lee.
\newblock {\em An Introduction to Smooth Manifolds, 2nd. ed.}
\newblock Springer, New York, 2003.

\bibitem{lee2018model}
K.~Lee and K.~Carlberg.
\newblock Model reduction of dynamical systems on nonlinear manifolds using
  deep convolutional autoencoders.
\newblock {\em arXiv preprint arXiv:1812.08373}, 2018.

\bibitem{lewis2016gradient}
A.~Lewis, R.~Smith, and B.~Williams.
\newblock Gradient free active subspace construction using morris screening
  elementary effects.
\newblock {\em Computers \& Mathematics with Applications}, 72(6):1603--1615,
  2016.

\bibitem{li2018sufficient}
B.~Li.
\newblock {\em Sufficient dimension reduction: Methods and applications with
  R}.
\newblock Chapman and Hall/CRC, 2018.

\bibitem{Li1992}
K-C. Li.
\newblock {On Principal {Hessian} Directions for Data Visualization and
  Dimension Reduction: Another Application of {Stein's} Lemma}.
\newblock {\em Journal of the American Statistical Association},
  87(420):1025--1039, 1992.

\bibitem{loudon2016}
T.~Loudon and S.~Pankavich.
\newblock Mathematical analysis and dynamic active subspaces for a long term
  model of {HIV}.
\newblock {\em arXiv preprint arXiv:1604.04588}, 2016.

\bibitem{Lukaczyk2014}
T.~Lukaczyk, F.~Palacios, Alonso~J. J., and P.~G. Constantine.
\newblock Active subspaces for shape optimization.
\newblock {\em 10th AIAA Multidisciplinary Design Optimization Conference},
  January 2014.

\bibitem{mallasto2018wrapped}
A.~Mallasto and A.~Feragen.
\newblock Probabilistic {R}iemannian submanifold learning with wrapped
  {G}aussian process latent variable models.
\newblock In {\em Proceedings of the 21st International Conference on
  Artificial Intelligence and Statistics (AISTATS)}, 2018.

\bibitem{mio2007shape}
W.~Mio, A.~Srivastava, and S.~H. Joshi.
\newblock On shape of plane elastic curves.
\newblock {\em International Journal of Computer Vision}, 73(3):307--324, 2007.

\bibitem{mukherjee2010learning}
S.~Mukherjee, Q.~Wu, D.-X. Zhou, et~al.
\newblock {Learning gradients on manifolds}.
\newblock {\em {Bernoulli}}, {16}({1}):{181--207}, {2010}.

\bibitem{nash1956imbedding}
J.~Nash.
\newblock The imbedding problem for riemannian manifolds.
\newblock {\em Annals of mathematics}, pages 20--63, 1956.

\bibitem{nilsson2007regression}
J.~Nilsson, F.~Sha, and M.~I. Jordan.
\newblock Regression on manifolds using kernel dimension reduction.
\newblock In {\em Proceedings of the 24th International Conference on Machine
  Learning (ICML)}, pages 697--704, 2007.

\bibitem{ozertem2011locally}
U.~Ozertem and D.~Erdogmus.
\newblock Locally defined principal curves and surfaces.
\newblock {\em Journal of Machine Learning Research}, 12:1249--1286, 2011.

\bibitem{panaretos2014principal}
V.~M. Panaretos, T.~Pham, and Z.~Yao.
\newblock Principal flows.
\newblock {\em Journal of the American Statistical Association},
  109(505):424--436, 2014.

\bibitem{pennec2006intrinsic}
X.~Pennec.
\newblock Intrinsic statistics on {R}iemannian manifolds: Basic tools for
  geometric measurements.
\newblock {\em Journal of Mathematical Imaging and Vision}, 25(1):127--154,
  2006.

\bibitem{pennec2018barycentric}
X.~Pennec.
\newblock Barycentric subspace analysis on manifolds.
\newblock {\em The Annals of Statistics}, 46(6A):2711--2746, 2018.

\bibitem{romor2023kernel}
F.~Romor, M.~Tezzele, A.~Lario, and G.~Rozza.
\newblock Kernel-based active subspaces with application to computational fluid
  dynamics parametric problems using the discontinuous {G}alerkin method.
\newblock {\em International Journal for Numerical Methods in Engineering},
  124(13):2947--2976, 2023.

\bibitem{Schulz2014}
V.~H. Schulz.
\newblock A riemannian view on shape optimization.
\newblock {\em The Journal of the Society for the Foundations of Computational
  Mathematics}, 2014.

\bibitem{serani2025survey}
A.~Serani and M.~Diez.
\newblock A survey on design-space dimensionality reduction methods for shape
  optimization.
\newblock {\em Archives of Computational Methods in Engineering}, 2025.

\bibitem{soto2026intrinsic}
C.~Soto, C.~Wang, Y.~Huang, and X.~Chen.
\newblock Intrinsic {R}iemannian cross-covariance for manifold-valued random
  objects.
\newblock {\em arXiv preprint arXiv:2606.10212}, 2026.

\bibitem{srivastava2002monte}
A.~Srivastava and E.~Klassen.
\newblock Monte carlo extrinsic estimators of manifold-valued parameters.
\newblock {\em IEEE Transactions on Signal Processing}, 50(2):299--308, 2002.

\bibitem{srivastava2016functional}
Anuj Srivastava and Eric~P. Klassen.
\newblock {\em Functional and Shape Data Analysis}.
\newblock Springer Series in Statistics. Springer, New York, 2016.

\bibitem{stinson2016}
K.~Stinson, D.~F. Gleich, and P.~G. Constantine.
\newblock A randomized algorithm for enumerating zonotope vertices.
\newblock {\em arXiv preprint arXiv:1602.06620}, 2016.

\bibitem{Walker2015}
S.W. Walker.
\newblock {\em The Shapes of Things: A Practical Guide to Differential Geometry
  and the Shape Derivative}.
\newblock SIAM, Philadelphia, PA, 2015.

\bibitem{white2019}
Andrew White, Sankaran Mahadevan, Zachary Grey, Jason Schmucker, and Alexander
  Karl.
\newblock Efficient calibration of a turbine disc heat transfer model under
  uncertainty.
\newblock {\em Journal of Thermophysics and Heat Transfer}, pages 1--11, 2020.

\bibitem{wu2010learning}
Q.~Wu, J.~Guinney, M.~Maggioni, and S.~Mukherjee.
\newblock Learning gradients: predictive models that infer geometry and
  statistical dependence.
\newblock {\em Journal of Machine Learning Research}, 11(Aug):2175--2198, 2010.

\bibitem{ying2020frechet}
C.~Ying and Z.~Yu.
\newblock {F}r\'echet sufficient dimension reduction for random objects.
\newblock {\em arXiv preprint arXiv:2007.00292}, 2020.

\bibitem{zahm2018}
O.~Zahm, P.~G. Constantine, Cl{\'e}mentine Prieur, and Y.~Marzouk.
\newblock Gradient-based dimension reduction of multivariate vector-valued
  functions.
\newblock {\em arXiv preprint arXiv:1801.07922}, 2018.

\bibitem{zanoni2025neural}
A.~Zanoni, G.~Geraci, M.~Salvador, A.~L. Marsden, and D.~E. Schiavazzi.
\newblock Neural active manifolds: nonlinear dimensionality reduction for
  uncertainty quantification.
\newblock {\em Journal of Scientific Computing}, 105:79, 2025.

\end{thebibliography}

\end{document}